\documentclass[aos]{imsart}

\RequirePackage{amsthm,amsmath,amsfonts,amssymb}
\RequirePackage[numbers,sort&compress]{natbib}
\RequirePackage[colorlinks,citecolor=blue,urlcolor=blue]{hyperref}
\RequirePackage{graphicx}

\usepackage{amsmath}
\usepackage{xcolor}
\usepackage[numbers]{natbib}
\usepackage[shortlabels]{enumitem}
\usepackage{amssymb}
\usepackage{makecell}
\usepackage{mathtools}
\usepackage[labelfont=bf]{caption}
\usepackage{tikz}
\usepackage{soul}
\usepackage{amsthm}
\usepackage{float}
\usetikzlibrary{arrows}

\usepackage{natbib} 
\usepackage{booktabs}
\usepackage{url}
\usepackage{bbm}
\usepackage{booktabs}
\usepackage{hyperref}
\usepackage{algorithm}
\usepackage{algorithmicx,algpseudocode}

\startlocaldefs
\algnewcommand\algorithmicinput{\textbf{Input: }}
\algnewcommand\INPUT{\State\algorithmicinput}
\algnewcommand\algorithmicoutput{\textbf{Output: }}
\algnewcommand\OUTPUT{\State\algorithmicoutput}

\theoremstyle{plain}

\newtheorem{theorem}{Theorem}[section]
\newtheorem{lemma}[theorem]{Lemma}

\newtheorem*{rem*}{Remark}

\newtheorem{cor}[theorem]{Corollary}
\theoremstyle{definition}

\newcommand{\R}[0]{\mathbb{R}}
\newcommand{\pa}[1]{\left( {#1} \right)}
\newcommand{\bc}[1]{\left\{ {#1} \right\}}
\newcommand{\ba}[1]{\left[ {#1} \right]}

\newcommand{\ip}[1]{\langle\, {#1} \rangle}
\newcommand{\norm}[1]{\left\lVert#1\right\rVert}
\newcommand{\abs}[1]{\left\lvert #1 \right\rvert}

\newcommand{\Ber}[0]{\operatorname{Ber}}

\newcommand{\InvGamma}[0]{\operatorname{Inverse-Gamma}}
\newcommand{\Gam}[0]{\operatorname{Gamma}}
\newcommand{\TV}[0]{\operatorname{TV}}
\newcommand{\KL}[0]{\operatorname{KL}}
\newcommand{\PG}[0]{\operatorname{PG}}
\newcommand{\IG}[0]{\operatorname{IG}}
\newcommand{\TN}[0]{\operatorname{TN}}
\newcommand{\diag}[0]{\operatorname{diag}}
\newcommand{\Probit}[0]{\operatorname{Probit}}
\newcommand{\Logit}[0]{\operatorname{Logit}}
\newcommand{\Lasso}[0]{\operatorname{Lasso}}
\newcommand{\Poly}[0]{\operatorname{Poly}}

\newcommand{\Ch}[0]{\operatorname{Ch}}
\newcommand{\sign}[0]{\operatorname{sign}}
\newcommand{\Erf}[0]{\operatorname{Erf}}

\newcommand{\revise}[1]{#1}
\newcommand{\edit}[1]{#1}

\endlocaldefs

\begin{document}

\begin{frontmatter}
\title{Fast Mixing of Data Augmentation Algorithms: Bayesian Probit, Logit, and Lasso Regression}
\runtitle{Fast Mixing of Data Augmentation Algorithms}
\runauthor{LEE AND ZHANG}

\begin{aug}
\author[A]{\fnms{Holden}~\snm{Lee} \ead[label=e1]{hlee283@jhu.edu}}
\and
\author[A]{\fnms{Kexin}~\snm{Zhang}\ead[label=e2]{kzhang91@jhu.edu}}

\address[A]{Department of Applied Mathematics and Statistics,
Johns Hopkins University\printead[presep={ ,\ }]{e1,e2}}

\end{aug}

\begin{abstract}
We propose using a modified conductance-based method to study the mixing time of an important class of two-block Gibbs samplers, the data augmentation (DA) algorithm. Using this method, we prove the first non-asymptotic polynomial upper bounds on mixing times of three important DA algorithms: DA algorithms for Bayesian Probit regression \cite{albert1993bayesian} (ProbitDA) and Bayesian Logit regression \cite{polson2013bayesian} (LogitDA), and Bayesian Lasso Regression \cite{park2008bayesian,rajaratnam2015scalable} (LassoDA). Concretely, for ProbitDA and LogitDA, we demonstrate a tight bound that explicitly depends on the design matrix and prior covariance matrix. Under the assumption that data are independently generated from either a sub-Gaussian or log-concave distribution and properly scaled, the bound implies that with $\eta$-warm start, parameter dimension $d$, and sample size $n$, with high probability over data, the two algorithms require $\mathcal{O}\pa{n\log \pa{\frac{\log \eta}{\epsilon}}}$ steps to obtain samples with at most $\epsilon$ error in TV, KL, or $\chi^2$ distance. Meanwhile, we show that under minimal data assumptions, LassoDA requires $\mathcal{O}\pa{d^2(d\log d +n \log n)^2 \log \pa{\frac{\eta}{\epsilon}}}$ steps to achieve $\epsilon$-accuracy in TV distance. The results are generally applicable to settings with large $n$ and large $d$, including settings with highly imbalanced response data in Probit and Logit regression. We compare them with the best known guarantees of Langevin Monte Carlo and Metropolis Adjusted Langevin Algorithm. We evaluate our theoretical results using numerical examples, and discuss the mixing times of the three algorithms under feasible initialization.
\end{abstract}



\end{frontmatter}


\section{Introduction} \label{s:intro}
A key task in Bayesian inference is to draw samples from posterior distributions. The data augmentation (DA) algorithm \cite{van2001art,hobert2011data, roy2024data} is a Markov Chain Monte Carlo (MCMC) method that generates auxiliary variables to enable a Gibbs sampling procedure. Ever since the DA algorithms were proposed \cite{tanner1987calculation}, they have been applied to a wide range of models. 
Some of the auxiliary variables are intrinsic to the model, including missing data, unobserved variables, and latent states (see e.g. \cite{dvorzak2016sparse, diebolt1994estimation, joseph1995bayesian, justiniano2008time, cohen2007estimating}). Others carry no explicit meaning. They are introduced purely to facilitate the sampling algorithm.
Although they vary across different models, a typical DA algorithm exhibits a two-block Gibbs sampling structure: To draw samples from the posterior $\pi(\beta|y)$, with $y \in \mathbb{R}^n$ denoting the observed data and $ \beta \in \mathbb{R}^d$ denoting the parameters, it alternatively updates the parameters $\beta$ and the auxiliary variables $z$. Specifically, at the $(m+1)^{th}$ iteration, the algorithm draws sample according to 
\begin{align} \label{e:two-block-model}
    z^{(m+1)} \sim \pi(z|\beta^{(m)}, y) \quad \beta^{(m+1)} \sim \pi(\beta|z^{(m+1)}, y),
\end{align}
where 
the superscript
denotes the iteration at which the sample is drawn. 

DA algorithms, like other Gibbs samplers, are favorable because they are automatic with no user-tuned parameters. This motivates researchers to design DA algorithms for many posterior distributions that are difficult to handle, especially in common Bayesian inference settings. A key challenge is to find auxiliary variables $z$ that make a full set of conditional distributions accessible. Concretely, under~\eqref{e:two-block-model}, an efficient DA algorithm requires: (i) the $\beta$-marginal of the joint distribution matches the posterior, i.e. $\pi(\beta|y)= \int \pi(\beta, z|y) dz $; and (ii) both conditionals $\pi(z|\beta, y)$ and $\pi(\beta|z, y)$ are easy to sample from. Despite the simplicity in implementation, the DA algorithm has a complex structure and additional variables, making its running time the central practical concern. 
This motivates a line of work on theoretical guarantees for the running time of DA algorithms. Roughly speaking, we can describe the running time as the product of the cost per iteration and the number of iterations needed. The cost per iteration is typically easily characterized, which leaves the number of iterations to be of central theoretical interest. In the context of MCMC algorithms, this refers to how fast the underlying Markov chain converges, quantified by the \emph{mixing time},  
the number of iterations needed to get samples within $\epsilon$-distance (in total variation) to the target distribution. 

Among various perspectives of mixing time analysis, a basic theoretical question is to understand the \textit{quantitative} relationship between mixing time and the quantities of interest. Typically, the focus is on how the mixing time scales with the parameter dimension $d$ and the sample size $n$ in nonasymptotic settings. Of particular interest is determining whether the chain has a polynomial dependency (\textit{rapid/fast mixing}) or exponential dependency (\textit{slow mixing}) in $n$ and $d$. Fast mixing results are desirable, as they guarantee the algorithm runs efficiently in high-dimensional and large-sample settings. 

This paper provides a framework to obtain tight quantitative theoretical guarantees for DA algorithms. We demonstrate it on the DA algorithms designed for sampling from posteriors of Bayesian Probit regression (ProbitDA, \cite{albert1993bayesian}), Bayesian Logit regression (LogitDA, \cite{polson2013bayesian}), and Bayesian Lasso regression (LassoDA, \cite{rajaratnam2015scalable, park2008bayesian}). The three algorithms are representative because 
they address standard settings, have attracted long-standing theoretical attention, and are widely used
(see e.g. \cite{nguyen2021financial, xun2017correlated, durante2018bayesian, griffin2020modelling, zhang2018bayesian, grant2016quantitative}). We will introduce them in Section~\ref{s:setup}. 

Notably, despite our specific focus, this study has broader implications. The DA algorithm’s simple two-block structure makes it a prototypical case in general MCMC theory \cite{rajaratnam2015scalable, qin2022wasserstein, qin2021limitations}. We are also among the first to analyze Gibbs samplers and statistically motivated sampling problems under isoperimetric-type assumptions (see Section~\ref{s:general}; see also \cite{ascolani2024entropy, montanari2024provably}), highlighting a promising direction for future work.

Before introducing the main algorithms in Section \ref{s:setup} and main theorems in Section \ref{s:main}, we first introduce the general framework we use to study mixing time.

\subsection{Past work on MCMC convergence theory concerning DA algorithms} \label{s:past} The convergence behaviors of DA algorithms have received long-standing attention, concentrating on ProbitDA, LogitDA, and LassoDA. Nevertheless, a theoretical understanding of this behavior
remains incomplete, especially on how the mixing time scales with $n$ and $d$. 

A large body of early works are devoted to proving geometric ergodicity using drift and minorization conditions (d\&m,  \cite{rosenthal1995minorization,jones2001honest}): \cite{roy2007convergence} for ProbitDA, \cite{choi2013polya} for LogitDA, \cite{khare2013geometric} for the original version of LassoDA in \cite{park2008bayesian},  and \cite{rajaratnam2015scalable} for LassoDA. Geometric ergodicity is a desirable convergence property, which refers to the existence of a geometric convergence rate of the total variation distance to the stationary distribution. These works are only sufficient to show the existence of such a geometric rate, without 
explicit dependence on $n$ and $d$,  or imply an upper bound on mixing time with exponential dependence on $n$ and $d$. The latter point is rigorously developed by \cite{rajaratnam2015mcmc}, who show that the provided geometric convergence rates in \cite{choi2013polya} and \cite{khare2013geometric} converge exponentially fast to one as $n \rightarrow \infty$ or $d \rightarrow \infty$. Furthermore, \cite{qin2022wasserstein} and \cite{qin2021limitations} point out the limitations of d\&m in obtaining tight dependence on $n$ and $d$. 

To improve the early convergence results, recent attention has been drawn towards the dependency of convergence on $n$ and $d$, which is referred to as the  ``convergence complexity''  analysis by \cite{rajaratnam2015mcmc}. In particular, \cite{rajaratnam2015mcmc} demonstrates that the geometric convergence rate of LassoDA's $v$-marginal chain is at most $\frac{d}{n+d-2}$, through constructing a lower bound on the correlation between consecutive $v$ samples and running numerical experiments. Albeit promising, the study does not address the convergence of the joint chain of $(\beta,v)$, which is the complete parameter set of LassoDA. Following \cite{rajaratnam2015mcmc}, \cite{qin2019convergence} improves upon \cite{roy2007convergence}, providing two sets of results supporting that the geometric convergence rate of ProbitDA can be bounded away from one when (i) $d$ is fixed, $n \rightarrow \infty$ or (ii) $n$ is fixed, $d \rightarrow \infty$. To address the problem with both $n$ and $d$ growing, the follow-up work \cite{qin2022wasserstein} demonstrates that the geometric convergence rate can be bounded away from one in particular settings: (i) $n$ and $d$ are arbitrary and the prior provides enough shrinkage, or (ii) $n \rightarrow \infty, d \rightarrow \infty$, and the design matrix has repeated structure. The joint dependency of $n$ and $d$ in general cases remains unknown. Although insightful, the asymptotic results generally have no direct implications for non-asymptotic settings. 

More recently, \cite{johndrow2019mcmc} shows that ProbitDA and LogitDA mix slowly with highly imbalanced response data. In a one-dimensional perfectly imbalanced model with all-one responses, they established conductance upper bounds of $\tilde{\mathcal{O}}(n^{-1/2})$ (which 
imply mixing time lower bounds of $\tilde \Omega(\sqrt{n})$), underperforming a Metropolis–Hastings algorithm. In contrast, we derive mixing time upper bounds in more general settings, including imbalanced responses.

The concurrent work \cite{ascolani2024entropy} pioneers non-asymptotic analysis of general $M$-block Gibbs samplers under strongly log-concave target assumptions. They prove an $\mathcal{O}(M\kappa)$ mixing time guarantee, where $\kappa$ is the condition number, which matches the dependency in our results. While their analysis applies to more general Gibbs samplers, our method has three key advantages. First, in the two-block case, we require an isoperimetric-type condition for only one block, whereas they require it for both, excluding important cases such as LogitDA where $\pi(\beta,z\mid y)$ is non-log-concave. Second, our Cheeger-type isoperimetric assumption is strictly weaker than strong log-concavity. Finally, our framework applies to deterministic-scan Gibbs samplers, which are standard in DA algorithms, while theirs applies to random-scan updates.

\subsection{Modified conductance-based method for controlling mixing time of DA algorithms} \label{s:general}
\footnote{The formal definitions of notations used in this section can be found in Section \ref{s:notation} for general notation, Section~\ref{s:iso} for $\Ch(\pi)$, and Section \ref{s:setup} for the mixing time.} 
To show fast mixing in terms of $n$ and $d$, we draw on a body of mixing time analysis based on convex geometry and isoperimetric inequalities, originating from sampling problems on convex bodies (e.g., \cite{dyer1991random, kannan1997random, lovasz2004hit, lovasz1999hit}). This literature has established fast mixing results for many important algorithms (e.g., \cite{lovasz2003hit, lovasz2007geometry, dalalyan2017further, dwivedi2019log, mangoubi2021mixing, chen2020fast}). Rather than focusing on specific target distributions, studies in this line typically consider general classes of targets satisfying assumptions such as bounded support, log-concavity (see \cite{chewi2023log} for a review), or more generally, isoperimetry.

In particular, we employ the conductance-based method \cite{lovasz1993random, sinclair1989approximate, jerrum1989approximating}, which has proven powerful for analyzing discrete-time Markov chains (e.g., \cite{dwivedi2019log, narayanan2016randomized, chen2018fast, lovasz1999hit, lovasz2007geometry, lovasz2004hit, mou2019sampling, chen2020fast}). We give a compact statement of this method in Theorem~\ref{l:general} and defer a detailed introduction to Sections~\ref{s:cond} and~\ref{s:improved}.

\begin{theorem}[Conductance-based mixing time upper bound] \label{l:general} Let $\Psi$ be a Markov chain with transition kernel $\mathcal{P}:\mathcal{M}_1(\mathbb{R}^d) \to \mathcal{M}_1(\mathbb{R}^d)$ and limiting distribution $\pi \in \mathcal{M}_1(\mathbb{R}^d)$. Suppose the following:
\begin{enumerate}
    \item (Isoperimetry) $\pi$ satisfies Cheeger-type isoperimetric inequality with constant $\Ch(\pi) > 0$.
    \item (One-step overlap) For some $h \in (0,1]$ and $\Delta > 0$, we have $\forall \beta_1, \beta_2 \in \mathbb{R}^d $ s.t. $ \|\beta_1-\beta_2\|_2 \le \Delta$, $\TV(\mathcal{P}_{\beta_1}, \mathcal{P}_{\beta_2}) \le 1-h$.
\end{enumerate}
With $\epsilon$-error tolerance and $\eta$-warm start, the mixing time of $\Psi$ satisfies 
$
t_{\Psi}(\eta, \epsilon) \le c \frac{\Ch(\pi)^2}{\Delta^2 h^2}\log \pa{\frac{\eta}{\epsilon}},
$ where $c$ is a universal constant. 
\end{theorem}

The conductance-based method is generally applicable. The isoperimetry condition guarantees a desirable property of the target distribution, indicating the absence of bottlenecks and a light tail. It potentially covers a large class of distributions and has been a preferred assumption to study sampling problems (see e.g. \cite{vempala2019rapid, wibisono2019proximal}). The one-step overlap condition quantifies the locally bounded variation of the transition kernel and is generally expected to hold for practical samplers. Indeed, if it fails, there exist two arbitrarily close points $\beta_1, \beta_2 \in \mathbb{R}^d$ such that $TV(\mathcal{P}_{\beta_1}, \mathcal{P}_{\beta_2})$ is arbitrarily close to 1, i.e., the distributions have almost disjoint supports. 

The main difficulty is obtaining explicit quantities in the two conditions under concrete settings, especially the one-step overlap for the two-step Gibbs kernels of the DA algorithms. We notice that the second-step kernel $\mathcal{P}_2$ is always a non-expansive map in TV distance. Using this, we prove two sufficient conditions for the one-step overlap in Theorem \ref{l:general}, each depending only on the first-step kernel $\mathcal{P}_1$, which greatly simplifies the analysis. We state these conditions below and defer their proofs to Section \ref{s:general-proof}. Our analyses of ProbitDA and LogitDA verify Condition 2a; for LassoDA we verify only the weaker Condition 2b. As a side remark, when the second step has a more tractable structure, one may instead analyze the flipped chain \cite{roberts2001markov}, which has the same mixing time.

\begin{cor} [Modified Conductance-Based Method for DA chains] \label{t:general} Let $\Psi$ be a DA chain that alternately samples from $\pi(z|\beta)$ and $\pi(\beta|z)$ with $\beta \in \mathbb{R}^d $ and $ z \in \mathbb{R}^k$. We break the kernel into two steps: let $\mathcal{P}_1: (\mathcal{M}_1(\mathbb{R}^d), TV) \to (\mathcal{M}_1(\mathbb{R}^k), TV)$ and $\mathcal{P}_2: (\mathcal{M}_1(\mathbb{R}^k), TV) \to (\mathcal{M}_1(\mathbb{R}^d), TV)$ given by $\nu\mathcal{P}_1(z)=\int \pi(z|\beta)\nu(d\beta)$ for $\nu \in \mathcal{M}_1(\mathbb{R}^d)$ and $\nu\mathcal{P}_2(\beta)=\int \pi(\beta|z)\nu(d z)$ for $\nu \in \mathcal{M}_1(\mathbb{R}^k)$, respectively. 
Either of the two conditions below is sufficient to the one-step overlap condition in Theorem~\ref{l:general}: For some $h \in (0,1]$ and $\Delta > 0$,  
\begin{enumerate} [label=2\alph*.]
    \item (Lipchitzness of the first step kernel)  $\forall \beta_1, \beta_2 \in \mathbb{R}^d$, $TV(\delta_{\beta_1}\mathcal{P}_1, \delta_{\beta_2}\mathcal{P}_1) \le \frac{1-h}{\Delta} \|\beta_1 - \beta_2\|_2$.
    \item (Nontrivial bounded variation of the first step kernel at scale $\Delta$) $\forall \beta_1, \beta_2 \in \mathbb{R}^d $ s.t. $ \|\beta_1-\beta_2\|_2 \le \Delta$, $TV(\delta_{\beta_1}\mathcal{P}_1, \delta_{\beta_2}\mathcal{P}_1) \le 1-h$. 
\end{enumerate}
That is, Theorem~\ref{l:general} still holds if we replace the one-step overlap condition by 2a or 2b. 
\end{cor}

For bounding the isoperimetric constant, the analysis is straightforward when the target is strongly log-concave (i.e., $\pi \propto e^{-f}$  with $f$ strongly convex), and the improved conductance method via log-isoperimetry (Section \ref{s:improved}) yields a double-logarithmic dependence on the warmness parameter. This covers ProbitDA and LogitDA. For weakly log-concave and non-log-concave targets, useful techniques exist but typically require specialized treatment. Beyond results for special cases {\cite[Section 2.3]{chewi2023log}, \cite{holley1986logarithmic, talagrand1991new, talagrand1996transportation, courtade2020bounds, barthe2019spectral}}, general approaches include constructing Lipschitz transport maps from measures with known isoperimetry (e.g., Gaussian) \cite{caffarelli2000monotonicity, kolesnikov2011mass, kim2011generalizationcaffarelliscontractiontheorem, dai2023lipschitz, mikulincer2024brownian}, using the KLS conjecture results \cite{klartag2023logarithmic, bobkov1999isoperimetric, kannan1995isoperimetric, lee2017eldan, eldan2013thin, alonso2015approaching}, and applying flexible transference inequalities \cite{barthe2013transference, milman2010isoperimetric, cattiaux2020poincare}. For LassoDA, whose target is non-log-concave, we propose a new chain transformation technique that preserves mixing time (Lemma \ref{l:transform}), converts the target into a tractable log-concave form, and enables a polynomial guarantee via a transference inequality \cite{milman2012properties}.

\subsection{Our contributions}
In summary, our main contributions are the following.
\begin{enumerate}
\item We propose a modified conductance method (Corollary \ref{t:general}) to control mixing time of DA algorithms, which is also applicable to general two-block Gibbs samplers. 
\item  We apply the method to ProbitDA, LogitDA, and LassoDA under different conditions on initial distributions and data distributions, with results detailed in Section \ref{s:main} and outlined in Table~\ref{tab:result-summary}. These are the first non-asymptotic polynomial guarantees in general settings, in contrast with many previous results with exponential dependency or in restricted settings.

\item We perform numerical experiments to evaluate the tightness of the bounds. The simulations correctly reflect the dependencies predicted by our theoretical results for ProbitDA and LogitDA. See Appendix \ref{s:experiments} for details. 

\item We compare the mixing time of the three DA algorithms with Langevin Monte Carlo and Metropolis Adjusted Langevin Algorithm in terms of cost per iteration and upper bounds of mixing time. See Appendix \ref{s:comparison} for details. 
\end{enumerate}

\begin{table} 
\caption{Summary of $\epsilon$-mixing time in TV distance of DA algorithms for sampling from posteriors of Bayesian probit regression (ProbitDA, \cite{albert1993bayesian}), Bayesian logit regression (LogitDA, \cite{polson2013bayesian}), and Bayesian Lasso (LassoDA, \cite{rajaratnam2015scalable, park2008bayesian}) These statements hide the dependency on parameters of the prior for LassoDA and the log-concave and sub-Gaussian distribution. We refer the readers to the links in the last column for the complete theorems.}
\begin{tabular}{ccccc} \hline
    Algorithm & Initialization & Data Distribution & Mixing Time & Theorem\\  \hline
    ProbitDA or & $\eta$-warm & / &$\mathcal{O}\pa{\|X\|_{\mathrm{op}}^2\lambda_{\max}(B)\log\pa{\frac{\log \eta}{\epsilon}}}$ & ~\ref{t:probitlogit} \\
     LogitDA & $\eta$-warm & \tiny{\makecell{log-concave/ sub-Gaussian, \\  independent and, properly scaled}} &$\mathcal{O}\pa{n\log\pa{\frac{\log \eta}{\epsilon}}}$ & ~\ref{t:main-cor}  \\  
     & feasible &\tiny{\makecell{log-concave/ sub-Gaussian, \\ independent, and properly scaled}} &$\mathcal{O}\pa{n\log\pa{\frac{d \log n}{\epsilon}}}$ & \ref{t:feasible-all} \\ \midrule
    LassoDA & $\eta$-warm & \tiny{\makecell{$\norm{X}_{\mathrm{op}}=\Poly(nd)$ \\  \& $\norm{y}=\Poly(nd)$}} &$\mathcal{O}\pa{d^2(d\log d + n \log n)^2\log\pa{\frac{\eta}{\epsilon}}}$ & ~\ref{t:lasso}\\ 
      & feasible & \tiny{\makecell{$\norm{X}_{\mathrm{op}}=\Poly(nd)$ \\  \& $\norm{y}=\Poly(nd)$}} & \scriptsize{\makecell{$\mathcal{O} \Big(d^2(d\log d + n \log n)^2  \cdot$ \\$
      \pa{d\log d + n \log n + \log \pa{\frac{1}{\epsilon}}}$\Big)}} & \ref{t:feasible-lasso}\\\hline
\end{tabular} 
 \label{tab:result-summary}
\end{table}

\subsection{Notations} \label{s:notation}
We reserve $c$, $c^{\prime}$, and $c^{\prime\prime}$ for universal constants, independent of all the parameters of interest (in particular $n$ and $d$), whose values can change from one occurrence to the other. We commonly employ superscripts $^{\Probit}, ^{\Logit},$ and $^{\Lasso}$ to restrict a general quantity to a particular algorithm, ProbitDA, LogitDA, and LassoDA, respectively. 

\paragraph*{Asymptotic} We say $f(x)=\mathcal{O}(g(x))$ if there exists a universal constant such that $f(x) \le c g(x)$ for all $x$. Similarly, $f(x)=\Omega(g(x))$ if there exists a universal constant such that $f(x) \ge c g(x)$ for all $x$. The notations $\tilde{\mathcal{O}}(g(x))$ and $\tilde{\Omega}(g(x))$ denote, respectively, $\mathcal{O}(g(x))$ and $\Omega(g(x))$ with logarithmic factors suppressed. We use $f(x)=\Poly(g(x))$ to express that $f(x)=\mathcal{O}(P(g(x)))$, where $P(g(x))$ is some polynomial of $g(x)$.

\paragraph*{Matrix} We denote the operator norm of a matrix $A$ by $\| A \|_{\textup{op}}$. If $A$ is a square matrix, we use $\lambda_{\max}(A)$ and $\lambda_{\min}(A)$ to represent its maximum and minimum eigenvalue, respectively. $\mathbb{I}_d$ is the $d$-dimensional identity matrix. $\mathbf{1}_n$ is the $n$-dimensional all-ones vector. 

\paragraph*{Markov chain} We use $\Psi$ to denote a general ergodic Markov chain on $\mathbb{R}^d$, with $\mathcal{P}$ being its Markov transition kernel, $\pi$ being its stationary distribution, and $\nu$ being its initial distribution. We use $\mathcal{P}_x$ as a shorthand for $\delta_x\mathcal{P}$, where $\delta_x$ is the Dirac measure centered at $x$.  

\paragraph*{Probabilistic distance} Let $\mathcal{M}_1(\mathbb{R}^d)$ denote the space of probability measures on $\mathbb{R}^d$. For $\mu_1, \mu_2 \in \mathcal{M}_1(\mathbb{R}^d)$, we use $\TV(\mu_1, \mu_2)$ to denote their total variation distance given by 
\begin{align}\label{e:tv-def}
    \TV(\mu_1, \mu_2) = \sup_{\text{ measurable } A \subseteq \mathbb{R}^d } |\mu_1(A) - \mu_2(A)|
\end{align}
Furthermore, we use $\KL(\mu_1||\mu_2) = \int \log \pa{\frac{d\mu_1}{d\mu_2}} d\mu_1$ and $\chi^2(\mu_1||\mu_2)=\int \pa{\frac{d\mu_1}{d\mu_2}-1}^2 d\mu_2$ to denote their Kullback-Leibler (KL) divergence and $\chi^2$-divergence, respectively.\\

The remainder of the paper is organized as follows. In Section~\ref{s:setup}, we formally introduce the notion of mixing time and the three DA algorithms under study.  In Section~\ref{s:main}, we present the main results of upper bounds on mixing times. Section~\ref{s:proof} is devoted to the proofs of the main results. We conclude in Section~\ref{s:conclude} by discussing several future research directions. We perform numerical
experiments to assess our guarantees and compare our results with the best known guarantees of alternative algorithms in Appendix \ref{s:experiments} and Appendix \ref{s:comparison}, respectively. 

\section{Problem setup}\label{s:setup}
This section is devoted to formally stating the goal of our analysis and introducing the algorithmic details of ProbitDA, LogitDA, and LassoDA. To dive straight into our topic, we assume familiarity with the basic concepts of Markov chains, a rigorous introduction of which can be found in \cite{levin2017markov}. 
\subsection{Mixing time with a warm start}\label{s:mixing-time-intro}

To sample from a target distribution $\pi$ on the state space $\mathbb{R}^d$, one can design a Markov chain $\Psi$ with a Markov transition kernel $\mathcal{P}$ such that starting from any distribution $\nu$, 
the distribution will eventually converge to $\pi$ as the number of iterations $k$ tends to infinity: 
$$
\nu \mathcal{P}^k \rightarrow \pi \quad \text{    as    } \quad k \rightarrow \infty.
$$ \par

The mixing time quantifies the speed of convergence as the number of iterations needed to get $\epsilon$-close to the target distribution. It is not hard to see that the mixing time depends on how close the initial distribution $\nu$ is to $\pi$. For ease of the analysis, we control and measure the distance between $\nu$ and $\pi$ by the notion of \textit{warm start}. Specifically,  for a scalar $\eta \ge 1$, a $\eta$-warm start requires the initial distribution to satisfy
$$
\sup_{A} \frac{\nu(A)}{\pi(A)} \le \eta < \infty
$$ where the supremum is taken over all measurable sets $A \subseteq \mathbb{R}^d$.  Throughout the paper, we denote the \textit{mixing time} of the Markov chain $\Psi$ with $\eta$-warm start to
$\epsilon$-accuracy in TV distance ($\epsilon \in  (0,1)$) by
\begin{align*}
    t_{\Psi}(\eta, \epsilon) = t^{\TV}_{\Psi}(\eta, \epsilon) := \inf \{k \in \mathbb{N}: \TV(\nu \mathcal{P}^k, \pi) \le \epsilon, \text{ for all }\nu \text{ that is a } \eta \text{-warm start}\}.
\end{align*}
We define similarly the mixing time with respect to $\KL$-divergence and $\chi^2$-divergence: for $\textup{d} \in \{\KL, \chi^2\}$
$$
t^{\textup{d}}_{\Psi}(\eta, \epsilon) := \inf \{k \in \mathbb{N}: \textup{d}(\nu \mathcal{P}^k, \pi) \le \epsilon, \text{ for all }\nu \text{ that is a } \eta \text{-warm start}\}.
$$
We aim to obtain an upper bound of mixing time in terms of the sample size $n$ and the dimension of the parameter space $d$. 

\subsection{ProbitDA} \label{s:probit-intro}
\paragraph*{Model} Given the binary response vector $y \in \mathbb{R}^n$, a design matrix $X \in \mathbb{R}^{n\times d}$, and a gaussian prior $\mathcal{N}(b, B)$ with $b \in \mathbb{R}^d$ and $ B \in \mathbb{R}^{d\times d}$, a typical model for Bayesian probit regression is 
\begin{align}  \label{e:probit-model}
    y_i &\sim \Ber(\Phi(x_i^T\beta)) \quad i=1,\ldots,n ,\nonumber\\
    \beta &\sim \mathcal{N}(b, B),
\end{align}
where we denote $\beta \in \mathbb{R}^d$ as the regression coefficients, $y_i$ as the $i^{th}$ entry of $y$, $x_i$ as the $i^{th}$ row of $X$, $\Ber(p)$ as the Bernoulli distribution with parameter $p$, and $\Phi(x)$ as the standard Gaussian c.d.f. at $x$. 

\paragraph*{Posterior} The posterior of this model is 
\begin{align}\label{e:probit-pos}
    \pi(\beta|y) \propto \pi(y|\beta) \pi(\beta) \propto \prod_{i=1}^n (1-\Phi(x_i^T\beta ))^{1-y_i}\Phi(x_i^T\beta)^{y_i} e^{-\frac{1}{2}(\beta-b)^TB^{-1}(\beta-b)}.
\end{align}
\par
\paragraph*{Auxiliary variables and the algorithm} The pioneering work \cite{albert1993bayesian} proposes to introduce $n$ independent latent variables $z$ at each iteration. The key design is to rewrite the data generating process, keeping the dependencies between $y$ and $z$ unchanged: $y_i = \mathbf{1}\{z_i \ge 1\}$, where $z_i \sim \mathcal{N}(x_i^T\beta, 1)$, $i=1,...,n$. Under this model, $\pi(\beta|z,y)=\pi(\beta|z)$ follows standard normal regression results, whereas $\pi(z|\beta, y)$ is distributed as independent $\textit{truncated normals}$.  We use the notation $\TN(\mu, 1; y)$ to denote the normal distribution $\mathcal N(\mu, 1)$ truncated to $[0,\infty)$ if $y=1$, and truncated to $(-\infty, 0]$ if $y=0$. Specifically, $\TN(\mu, 1; 1)$ has a density 
\begin{align}
    f(x) = \frac{e^{-\frac{1}{2}(x-\mu)^2}}{\sqrt{2 \pi} \Phi(\mu)} \mathbbm{1}\{x \ge 0\},
\end{align}
while the density of $\TN(\mu, 1; 0)$ is 
\begin{align}
    f(x) = \frac{e^{-\frac{1}{2}(x-\mu)^2}}{\sqrt{2 \pi} \Phi(-\mu)} \mathbbm{1}\{x \le 0\}.
\end{align}
With this notation, the concrete idea of ProbitDA is to augment the data 
\begin{align}\label{e:trun-normal}
    z_i|\beta, y \sim \TN(x_i^T\beta, 1; y_i) \quad i=1,\ldots,n.
\end{align} 
The ProbitDA goes by alternatively generate samples from $\pi(z|\beta, y)$ and $\pi(\beta|z, y)$ as in Algorithm~\ref{a:probitda}.
\begin{algorithm}[t]
\caption{ProbitDA}
\begin{algorithmic}[1]
\INPUT $X \in \mathbb{R}^{n \times d}, y \in \mathbb{R}^n, b \in \mathbb{R}^d, B \in \mathbb{R}^{d \times d}$
\State Draw $\beta^{(0)}$ from an initial distribution.
\For{$m = 1,2,\ldots$}
  \State Draw independently $z_i^{(m)} \sim \TN(x_i^T\beta^{(m-1)}, 1; y_i), \quad i=1,\ldots,n$.
  \State Draw $\beta^{(m)} \sim \mathcal{N}((B^{-1}+X^TX)^{-1}(B^{-1}b+X^Tz^{(m)}), (B^{-1}+X^TX)^{-1})$.
\EndFor
\end{algorithmic}
\label{a:probitda}
\end{algorithm}

\subsection{LogitDA}
\paragraph*{Model} Bayesian logistic regression has the same setting as Bayesian probit regression in Section~\ref{s:probit-intro} except for the link function. That is, the model becomes
\begin{align} \label{e:logit-model}
    y_i &\sim \Ber(l(x_i^T\beta)) \quad i=1,\ldots,n ,\nonumber\\
    \beta &\sim \mathcal{N}(b, B) , 
\end{align}
where $l(x)=\frac{e^{x}}{1+e^{x}}$ is the logit link function. 

\paragraph*{Posterior} The posterior of this model is 
\begin{align}\label{e:logit-pos}
    \pi(\beta|y) \propto \pi(y|\beta) \pi(\beta) \propto \prod_{i=1}^n \pa{\frac{e^{x_i^T\beta}} {1+e^{x_i^T\beta}} }^{y_i} \pa{\frac{1} {1+e^{x_i^T\beta}} }^{1-y_i} e^{-\frac{1}{2}(\beta-b)^TB^{-1}(\beta-b)}.
\end{align} \par 
\paragraph*{Auxiliary variables and the algorithm} Ever since \cite{albert1993bayesian}, there has been considerable effort devoted to designing an analogous DA algorithm for the Bayesian logistic regression (see e.g. \cite{held2006bayesian, fruehwirth2007auxiliary, polson2013bayesian, zens2023ultimate}). We focus on \cite{polson2013bayesian}. Instead of generating additional truncated normal variables, they 
propose using the \textit{P\'olya-Gamma} random variable and making $n$ independent draws from it in each iteration. The P\'olya-Gamma variables that take two arguments, denoted as $\PG(a,c)$, 
are infinite convolutions of Gamma variables, and have efficient samplers. Three facts about P\'olya-Gamma variables are most related to our study: First, their densities satisfy the following relationship 
\begin{align}\label{e:pg-density}
    f(x; a,c)=e^{-\frac{c^2}{2}x} \cosh^a\pa{\frac{c}{2}}f(x;a,0),
\end{align}
where $f(x; a,c)$ is the density of $\PG(a,c)$. Second, the mean of $\omega \sim \PG(a,c)$ is 
\begin{align}\label{e:pg-mean}
    \mathbb E(\omega)=\frac{a}{2c} \tanh\pa{\frac{c}{2}}.
\end{align} Third, binomial likelihood with logit link can be represented as mixtures of Gaussian with respect to $PG(1,0)$, such that 
\begin{align} \label{e:polya-gamma-integral}
    \frac{e^{\psi y}}{1+e^{\psi}} = \frac{1}{2} e^{(y-1/2)\psi}\int_0^\infty e^{-z\psi^2/2} f(z;1,0) d z
\end{align}
We can then plug \eqref{e:polya-gamma-integral} with $\psi=x_i^T\beta$, $y=y_i$, and $z=z_i$ into \eqref{e:logit-pos} to get $\pi(\beta, z|y)$. By  calculating $\pi(z|\beta,y)$ from $\pi(\beta, z|y)$,  we get the augmented data designed by \cite{polson2013bayesian},
\begin{align}\label{e:polya-draw}
    z_i|\beta,y \sim \PG(1, x_i^T\beta), \quad i=1,\ldots,n.
\end{align} 
The LogitDA 
proceeds by alternately 
generate samples from $\pi(z|\beta, y)$ and $\pi(\beta|z, y)$ as in Algorithm~\ref{a:logitda}.
\begin{algorithm}[t]
\caption{LogitDA}
\begin{algorithmic}[1]
\INPUT $X \in \mathbb{R}^{n \times d}, y \in \mathbb{R}^n, b \in \mathbb{R}^d, B \in \mathbb{R}^{d \times d}$
\State Let $\kappa=y-\frac{1}{2}\mathbf{1}_n$. Draw $\beta^{(0)}$ from an initial distribution.
\For{$m = 1,2,\ldots$}
  \State Draw independently $z_i^{(m)} \sim \PG(1, x_i^T\beta^{(m-1)}), \quad i=1,\ldots,n$.
  \State Let $\Omega^{(m)}=\diag(z^{(m)})$.
  \State Draw $\beta^{(m)} \sim \mathcal{N}((B^{-1}+X^T\Omega^{(m)} X)^{-1}(X^T\kappa+B^{-1}b), (B^{-1}+X^T\Omega^{(m)} X)^{-1})$.
\EndFor
\end{algorithmic}
\label{a:logitda}
\end{algorithm}

\subsection{LassoDA}\label{s:lasso-intro}
\paragraph*{Model}The Lasso \cite{tibshirani1996regression} estimates linear regression coefficients by $L_1$-constrained least squares. Concretely, consider a linear regression model, 
$$
y=\mu \mathbf{1}_n + X\beta+\epsilon,
$$ where $y \in \mathbb{R}^n$ is the response data,  $X \in \mathbb{R}^{n \times d}$ is the matrix of the regressors with centered columns, $\beta \in \mathbb{R}^d$ is the vector of coefficients, and $\epsilon$ is independent and identically distributed mean-zero Gaussian residuals. The Lasso estimates the coefficients by solving the following optimization problem
\begin{align}\label{e:lasso-opt}
    \min_{\beta} \|\tilde{y}-X\beta\|^2_2  + \lambda \|\beta\|_1,
\end{align}
where $\lambda \ge 0$ is a tuning parameter and $\tilde{y}=y-\bar{y} \mathbf{1}_n$ is the centered response vector. Because of the nature of the $L_1$ penalty, the solution of the problem~\eqref{e:lasso-opt} tends to have some coefficients being exactly zero. This excludes non-informative variables and hence makes Lasso useful for variable selection. \par 
\cite{tibshirani1996regression} points out that one can study the Lasso estimate from a Bayesian point of view. They interpret the solution of the problem~\eqref{e:lasso-opt} as the posterior mode of the coefficients under a Laplace (double-exponential) prior. \cite{park2008bayesian} formulate the Bayesian Lasso model as follows: 
\begin{align*}
    y &\sim \mathcal{N}(\mu+X\beta, v \mathbbm{I}_n) \\
    p(\mu) &\propto 1 && \text{independent flat (improper) prior of $\mu$} \\ 
    p(\beta|v) &= \prod_{j=1}^d \frac{\lambda}{2 \sqrt{v}}e^{- \lambda\frac{|\beta_j|}{\sqrt{v}}}  \quad &&\text{conditional Laplace prior of $\beta$}\\
    p(v) &\propto \frac{e^{-\xi/v}}{v^{\alpha+1}}  \quad &&\text{inverse gamma prior of $v$}
\end{align*}
\begin{algorithm}[t]
\caption{LassoDA}
\begin{algorithmic}[1]
\INPUT $X \in \mathbb{R}^{n \times d}, y \in \mathbb{R}^n, \lambda \in \mathbb{R}^{+}, \alpha \in \mathbb{R}^{+}, \xi \in \mathbb{R}^{+} \cup \{0\}$
\State Let $\tilde{y}=y-\bar{y} \mathbf{1}_n$. Draw $\beta^{(0)}, v^{(0)}$ from initial distributions.
\For{$m = 1,2,\ldots$}
  \State Draw independently $\frac{1}{z_j^{(m)}} \sim \IG\pa{\sqrt{\frac{\lambda^2 v^{(m-1)}}{(\beta_j^{(m-1)})^2}}, \lambda^2}$, $j=1,\ldots,d$. Let $D_z^{(m)}=\diag \pa{z^{(m)}}$.
  \State Draw $v^{(m)} \sim \InvGamma \pa{\frac{n+2\alpha-1}{2}, \xi+\frac{\tilde{y}^T(\mathbbm{I}_n-X(X^TX+(D_z^{(m)})^{-1})^{-1}X^T) \tilde{y}}{2}}$.
  \State Draw $\beta^{(m)} \sim \mathcal{N}\pa{\pa{X^TX+\pa{D_z^{(m)}}^{-1}}^{-1}X^T\tilde{y},v^{(m)}\pa{X^TX+\pa{D_z^{(m)}}^{-1}}^{-1}}$.
\EndFor
\end{algorithmic}
\label{a:lassoda}
\end{algorithm}

\begin{figure}[t]
\centering

\tikzstyle{para} = [circle, node distance=5em, minimum height=3 em, minimum width=3em]
\tikzstyle{latent} = [draw, circle, fill=gray!20, node distance=3em, minimum height=3em, minimum width=3em]
\tikzstyle{line} = [draw, -latex']
\tikzstyle{empty} = [circle, node distance=3em, minimum height=3 em, minimum width=3em]
\tikzstyle{block} = [rectangle, draw, fill=gray!20, text width=3.5em, text centered, minimum height=4em]
    
\begin{tikzpicture}[node distance = 4em, auto]
\node [block] (latent) {$ z_1^{(m+1)}$ \\ $\, \vdots$ \\ $\, \vdots$ \\$z_n^{(m+1)}$};
\node [para, left of=latent] (start) {$\beta^{(m)}$};
\node [para, right of=latent] (end) {$\beta^{(m+1)}$};
\node [empty, below of=end] (end2) {};
\path [line] (start) -- (latent);
\path [line] (start) -- (latent);
\path [line] (latent) -- (end);
\path [line] (latent) -- (end);
\end{tikzpicture} \hspace{5em} \tikzstyle{para} = [circle, node distance=3em, minimum height=3 em, minimum width=3em]
\tikzstyle{empty} = [circle, node distance=6em, minimum height=3 em, minimum width=3em]
\tikzstyle{latent} = [circle, fill=gray!40, node distance=3em, minimum height=3em, minimum width=3em]
\tikzstyle{block} = [rectangle, draw, fill=gray!20, text width=3.5em, text centered, minimum height=4em]
\tikzstyle{line} = [draw, -latex']
\begin{tikzpicture}[node distance = 4.5em, auto]
\node [empty, left of=latent] (startmid) {};
\node [block, right of=startmid] (latent) {$ z_1^{(m+1)}$ \\ $\, \vdots$ \\ $\, \vdots$ \\$z_d^{(m+1)}$};
\node [para, above of=startmid] (start) {$v^{(m)}$};
\node [para, below of=startmid] (start2) {$\beta^{(m)}$};

\node [empty, right of=latent] (endmid) {};
\node [para, above of=endmid] (end) {$v^{(m+1)}$};
\node [para, below of=endmid] (end2) {$\beta^{(m+1)} $};
\path [line] (start) -- (latent);
\path [line] (start2) -- (latent);
\path [line] (latent) -- (end);
\path [line] (end) -- (end2);
\path [line] (latent) -- (end2);
\end{tikzpicture} \\
\begin{tikzpicture}
    \tikzstyle{block} = [rectangle,  text width=20em, text centered, minimum height=2em];
    \node[block]{\normalsize{ProbitDA/ LogitDA}};
\end{tikzpicture}
\begin{tikzpicture}
    \tikzstyle{block} = [rectangle,  text width=20em, text centered, minimum height=2em];
    \node[block]{\normalsize{LassoDA}};
\end{tikzpicture}
\caption{Illustration of the transition kernels of ProbitDA, LogitDA, and LassoDA. Here, the arrow represents conditional dependency.}
\label{fig:all-kernel}
\end{figure}

\paragraph*{Posterior}
The model allows the users to perform inference for all three parameters, $\mu, \beta$, and $v$. The joint posterior is
\begin{align} 
    \pi(\mu, v, \beta|y) \propto \pi(y|\mu,\beta,v) \pi(\mu) \pi(\beta|v) \pi(v) 
    \propto \frac{1}{v^{(n+d+2\alpha+2)/2}} e^{-\frac{1}{2v}\|y-\mu  \mathbf{1}_n-X\beta\|_2^2-\lambda\frac{\|\beta\|_1}{\sqrt{v}}-\frac{\xi}{v}}. 
\end{align}
As $\mu$ is rarely of interest, \cite{park2008bayesian} marginalizes it out to consider the posterior of $\beta$ and $v$. Using the fact that $\tilde{y}$ and $X$ is centered or $\mathbf{1}_n^T(\tilde{y}-X\beta)=0$, we have 
\begin{align} \label{e:lasso-pos}
    \pi(\beta, v|y)&\propto \int_\mu \pi(\mu, v, \beta|y) d\mu \propto \frac{1}{v^{(n+d+2\alpha+1)/2}} e^{-\frac{1}{2v}\|\tilde{y}-X\beta\|_2^2-\lambda\frac{\|\beta\|_1}{\sqrt{v}}-\frac{\xi}{v}} .
\end{align} 

\paragraph*{Auxiliary variables and the algorithm}
To generate samples from this posterior, \cite{park2008bayesian} develops a DA algorithm by representing Laplace distribution by a scale mixture of normals with an exponential mixing density \cite{andrews1974scale} (see also later proposals \cite{hans2009bayesian, mallick2014new}), such that for $a>0$
\begin{align}\label{e:laplace-mixture}
\frac{a}{2}e^{-a|\psi|} = \int_0^\infty \frac{1}{\sqrt{2\pi z}} e^{-\psi^2/2z}\frac{a^2}{2} e^{-a^2z/2} dz
\end{align}
We get $\pi(\beta, v|y)$ by plugging in ~\eqref{e:laplace-mixture} with    $\psi=\beta/\sqrt{v}$ and $a=\lambda$ into \eqref{e:lasso-pos}. It turns out the augmented data $z$ is $d$ independent inverse of inverse Gaussian variables. We use IG as a shorthand for inverse Gaussian. Specifically, we have 
\begin{align}
    \frac{1}{z_j}\Big|\beta, v, y \sim \IG\pa{\sqrt{\frac{\lambda^2 v}{\beta_j^2}}, \lambda^2},
\end{align}
where the density of $\IG(\mu, \lambda^\prime)$ is 
$
f(x)=\sqrt{\frac{\lambda^\prime}{2\pi x^3}} \exp\left[ -\frac{\lambda^\prime(x-\mu)^2}{2\mu^2 x}\right], x > 0.
$
There are multiple ways to perform Gibbs sampling for the three blocks of variables $\beta,v,z$.  \cite{park2008bayesian} adopts a three-block structure to iteratively sample from $\pi(z|\beta,v,y), \pi(v|\beta, z, y),$ and $\pi(\beta|v,z,y)$. \cite{rajaratnam2015scalable} proposes an improvement of taking a two-block update, meaning to sample alternately from $\pi(z|\beta,v,y)$ and $\pi(\beta, v| z, y)$, with the latter step splitting into $\pi( v| z, y)$ and $\pi(\beta|v, z, y)$. We focus on this improved algorithm, given as Algorithm \ref{a:lassoda}.  

We provide illustrative graphics for the three algorithms in Figure~\ref{fig:all-kernel}. 

\section{Main results}\label{s:main}
This section presents our main results on mixing time upper bounds for the ProbitDA, LogitDA, and LassoDA. We show the mixing time guarantees with a warm start. Although it simplifies theoretical analysis, a good warm start is rarely available. Because of this, we also provide mixing time guarantees with a feasible starting distribution.



\paragraph*{Mixing time of ProbitDA and LogitDA} We will first present a bound in terms of the design matrix $X \in \mathbb{R}^{n\times d}$ and the prior variance $B \in  \mathbb{R}^{d\times d}$, followed by a corollary with specific dependencies on $n$ and $d$.

\begin{theorem}\label{t:probitlogit} Let $\Psi \in \bc{\Psi^{\Probit}, \Psi^{\Logit}}$ and $\textup{d} \in \{\TV, \KL, \chi^2\}$. We have for any $\eta \ge 1$ and $\epsilon \in (0,1)$, the mixing time of $\Psi$ with $\eta$-warm start and $\epsilon$-error tolerance  satisfies
$$
t^{\textup{d}}_{\Psi}(\eta, \epsilon) \le c \; \|X\|_{\mathrm{op}}^2 \lambda_{\max}(B)
 \log \pa{\frac{\log\eta}{\epsilon}},
$$
where $c$ is a universal constant.
\end{theorem}
See Section~\ref{s:probit-proof} for the proof of Theorem~\ref{t:probitlogit}. Theorem~\ref{t:probitlogit} reveals that the mixing time of ProbitDA and LogitDA is determined jointly by the scale of the design matrix and the prior variance, evaluated by their maximum operator norms. 


In Corollary~\ref{t:main-cor}, we present a reformulation of the bound in Theorem~\ref{t:probitlogit} in terms of $n$
and $d$, under a standard statistical setting. Specifically, we assume that the first column of the design matrix $X$ consists of ones (representing the intercept), while the remaining columns are mean-zero, isotropic random vectors. Notably, if the covariates do not satisfy these conditions, one can apply an affine transformation, using the sample mean and covariance, to bring them into this canonical form, with only a controllable error. Following common practices in high-dimensional Bayesian regression \cite{simpson2017penalising, fuglstad2020intuitive}, we scale the design matrix (excluding the intercept column) by $1/\sqrt{d}$, so that the variance of $x_i^T\beta$ remains roughly constant as $d$ grows. The intercept is left unscaled, so that it would not shrink to zero as $d$ grows. This allows the model to effectively represent unbalanced response data $y$. We further assume $\lambda_{\max}(B)=\mathcal{O}(1)$, which includes, for instance, the case where $B=c\mathbb{I}_d$. To control the operator norm $\norm{X}_{\mathrm{op}}$, we consult results from random matrix theory, specifically non-asymptotic bounds for covariance estimation under sub-Gaussian \cite[Exercise 4.7.3]{vershynin2018high} and log-concave \cite{adamczak2010quantitative, adamczak2011sharp} assumptions.

\begin{cor} \label{t:main-cor}
Suppose that $\Psi \in \bc{\Psi^{\Probit}, \Psi^{\Logit}}$, $\textup{d} \in \{\TV, \KL, \chi^2\}$, and $\lambda_{\max}(B)=\mathcal{O}(1)$. Consider $X=\ba{\mathbf{1}_n \; \frac{1}{\sqrt{d}}\tilde{X}}$, where the rows of $\tilde{X}$, $\{\tilde{x}_i\}_{i=1}^n$, are mean-zero random vectors independently generated from a common distribution $\mathcal{L}$ on $\mathbb{R}^{d-1}$. We denote $\Sigma= \mathbb{E} \ba{ \; \tilde{x_i} \tilde{x_i}^T}$. We have the following bounds on mixing time $t_{\Psi}(\eta, \epsilon)$ under different assumptions on $\mathcal{L}$:

\begin{enumerate}
\item (Sub-Gaussianity) If $\mathcal{L}$ is sub-Gaussian with sub-Gaussian norm $K$, with probability at least $1-2e^{-u}$,
$$t^{\textup{d}}_{\Psi}(\eta, \epsilon) \le c \bc{n + \frac{\|\Sigma\|_{\mathrm{op}}}{d} \ba{n +  c^\prime n K^2 \pa{\sqrt{\frac{d+u}{n}} + \frac{d+u}{n}}}} \log \pa{ \frac{\log \eta}{\epsilon}}.$$
 
\item (Log-concavity) If $\mathcal{L}$ is log-concave, with probability at least $1-\exp(-c^{\prime\prime}\sqrt{d})$, 
$$
t^{\textup{d}}_{\Psi}(\eta, \epsilon) \le c \bc{n + \frac{\|\Sigma\|_{\mathrm{op}}}{d} \ba{n +c^\prime n \pa{\sqrt{\frac{d}{n}}+\frac{d}{n}}}}   \log \pa{\frac{\log \eta}{\epsilon}}.
$$ 
\end{enumerate}
Here, $c, c^\prime, c^{\prime\prime}$ are universal constants.
\end{cor}

We defer the proof of Corollary~\ref{t:main-cor} to Appendix \ref{a:ind-proof}. If we consider the $K$, $\|\Sigma\|_{\mathrm{op}}$, and $u$ to be independent of $n$ and $d$, Corollary~\ref{t:main-cor} implies a $\mathcal{O}\pa{n\log\pa{\frac{\log\eta}{\epsilon}}}$ guarantee for the mixing time of ProbitDA and LogitDA. 

\paragraph*{Mixing time of LassoDA} Lastly, we provide a polynomial mixing time guarantee for LassoDA. The assumptions on the data $X$ and $y$ are mild, as they influence the final bound only through logarithmic factors.
\begin{theorem} \label{t:lasso} We assume that $\norm{X}_{\mathrm{op}}=\Poly(nd)$ and $\norm{y}_2=\Poly(n)$. Given that $n \ge 2-2\alpha$ and a proper prior for the variance parameter $v$ (i.e. $\xi > 0$ and $\alpha > 0$), we have for any $\eta \ge 1$ and $\epsilon \in (0,1)$, the mixing time of LassoDA with $\eta$-warm start and $\epsilon$-error tolerance satisfies
$$
t_{\Psi^{\Lasso}}(\eta, \epsilon) \le c d^2 (d\log d+n \log n)^2 \log \pa{\frac{\eta}{\epsilon}},
$$
where $c$ is a constant depending on $M$, $\lambda$, and $\xi$.
\end{theorem}
See Section~\ref{s:lasso-proof} for the proof of Theorem~\ref{t:lasso}. 

\paragraph*{Mixing time with a feasible start}
To obtain a theorem statement for implementable algorithms without reference to a warm start parameter $\eta$, 
we can explicitly construct initial distributions, bound their warmness parameters $\eta$, and 
plug them into the mixing time upper bounds above. Specifically, we show that there exist feasible initial distributions with $\eta = \mathcal{O}\pa{n^{\frac{d}{2}}}$ for ProbitDA and LogitDA, and $\eta = \mathcal{O}\pa{e^{d\log d + n\log n}}$ for LassoDA. Due to space limitations, we refer interested readers to Appendix \ref{s:main-cold} for the formal theorem statements.

\section{Proofs} \label{s:proof}
Our proofs for upper bounds on mixing times rely on isoperimetric inequalities and the conductance of Markov chains. We will first introduce the techniques and general ideas in Section~\ref{s:proof overview} and dive proofs of theorems in the rest of the subsections. 
\subsection{Proof strategy overview and preliminaries} \label{s:proof overview}

\subsubsection{Isoperimetry} \label{s:iso}
In order to define isoperimetric inequality, we first introduce the notion of the Minkowski content. The \textit{Minkowski content}, or the \textit{boundary measure}, of a measurable set $A \subseteq \mathbb{R}^d$ is defined as 
\begin{align*}
    \pi^{+}(A)=\lim_{r \rightarrow 0^+}\frac{\pi(A^r)-\pi(A)}{r}
\end{align*} where $A^r=\{x \in \mathbb{R}^d: \exists y \in A, \|x-y\| \le r\}$. We say the measure $\pi$ satisfies the \textit{Cheeger-type isoperimetric inequality} with constant $\Ch(\pi) > 0$ if for all measurable set $A \subseteq \mathbb{R}^d$, 
\begin{align*}
    \pi^{+}(A) \ge \frac{1}{\Ch(\pi)} \min \{\pi(A), \pi(A^c)\},
\end{align*} 
and this is the minimal such constant. 
We call $\Ch(\pi)$ the \textit{Cheeger constant} of $\pi$. We will employ the following lemmas to calculate or upper bound the Cheeger constants of the ProbitDA, LogitDA, and LassoDA's target distributions.
    
\begin{lemma} \label{l:known} Let $\pi$ be a probability measure on $\mathbb{R}^d$.
\begin{enumerate}
    \item \cite{bobkov1997isoperimetric,talagrand1991new} If $\pi$ is a product 
    of double exponential measures, that is $\pi(x) = \prod_{i=1}^d \frac{1}{2b} e^{-\frac{|x_i|}{b}}$, we have $\Ch(\pi)=\frac{1}{b}$.
    \item \cite{cousins2014cubic, milman2008isoperimetric}  If $\pi$ is $m$-strongly log-concave, we have $\Ch(\pi)=\mathcal{O}(\frac{1}{\sqrt{m}})$. 
\end{enumerate} 
\end{lemma}
\begin{lemma}[{\cite[Corollary 3.4 (1) and equation (3.7)]{milman2012properties}}] \label{l:transfer} Let $\mu_1, \mu_2$ be two log-concave probability measures. If $\|\frac{d\mu_2}{d\mu_1}\|_{L^\infty} \le \exp(D)$, then 
$
\Ch(\mu_2) \le \mathcal{O}(D)\Ch(\mu_1).
$
\end{lemma}
See Appendix \ref{a:proof-transfer} for the proof of Lemma~\ref{l:transfer}.

\subsubsection{Conductance and mixing time} \label{s:cond}
With the notion of isoperimetry, we are ready to introduce the conductance-based argument for studying the mixing times. Given an ergodic Markov chain on $\mathbb{R}^d$ with transition kernel $\mathcal{P}$ and stationary distribution $\pi$,  we define the  \textit{conductance} as 
\begin{align} \label{e:cond}
    \Phi = \inf_A \frac{\int_A \mathcal{P}_u(A^c)\pi(u) du}{\min\{\pi(A), \pi(A^c)\}}
\end{align}
where $A$ is any measurable set in $\mathbb{R}^d$. The conductance measures how much probability mass flows between measurable partitions of the state space relative to the stationary measure of the two components, whichever is smaller. By the definition, we can expect a high conductance to contribute to fast mixing. The relationship is stated formally in the next lemma. 
\begin{lemma}[Modified version of {\cite[Corollary 1.5]{lovasz1993random}}]\label{l:mtub} Given a reversible Markov chain with nonnegative spectrum,  assuming $\eta$-warm start $\nu$, we have
$$
\TV(\nu \mathcal{P}^k, \pi) \le \frac{1}{2}\sqrt{\eta} e^{-k\Phi^2/2}.
$$
\end{lemma}
\begin{rem*}
    The $\beta$-marginal chain of the DA chain in \eqref{e:two-block-model} is reversible \cite[Lemma 3.1]{liu1994covariance} and has nonnegative spectrum \cite[Lemma 3.2]{liu1994covariance}. 
\end{rem*}

See Appendix \ref{a:mtub-nolazy} for the proof of Lemma~\ref{l:mtub}. This lemma shows that a lower bound on conductance gives an upper bound for the mixing time. The following lemma provides a way to obtain a lower bound on the conductance.  
\begin{lemma}[{\cite[Lemma 7.4.6]{chewi2023log} and \cite[Lemma 2]{dwivedi2019log}}] \label{l:main-cond} Consider a Markov chain on $\mathbb{R}^d$ with transition kernel $\mathcal{P}$ and stationary distribution $\pi$ satisfying the following conditions:
\begin{enumerate}
    \item (Isoperimetry) $\pi$ satisfies a Cheeger-type isoperimetric inequality with $\Ch(\pi)>0$.
    \item 
    (One-step overlap) For all $x,y \in \mathbb{R}^d$ satisfying $\|x-y\|_2 \le \Delta,$ we have $\TV(\mathcal{P}_x,\mathcal{P}_y) \le 1-h$. 
\end{enumerate}
Then, the conductance of the Markov chain satisfies
$
\Phi = \Omega\pa{\frac{h\Delta}{\Ch(\pi)}}.
$
\end{lemma}
See Appendix \ref{a:proof-main} for the proof of Lemma~\ref{l:main-cond}.

One can obtain an upper bound on mixing time by applying the lower bound for $\Phi$ in Lemma~\ref{l:main-cond} to Lemma~\ref{l:mtub} to give that 
$
  \TV(\nu \mathcal{P}^k, \pi) \le \frac{1}{2}\sqrt{\eta} e^{-k\Phi^{2}/2} \le \frac{1}{2}\sqrt{\eta}e^{- c k \frac{\Delta^2}{Ch^2(\pi)}} .
$ For any error tolerance $\epsilon \in (0,1)$, there exists $k \le c \frac{\Ch(\pi)^2}{\Delta^2} \log \frac{\sqrt{\eta}}{\epsilon}$ such that $\TV(\nu \mathcal{P}^k, \pi) \le \epsilon$. This implies 
\begin{align}\label{e:mixingtime-upperbound}
        t_{\Psi}(\eta, \epsilon) \le c \frac{\Ch(\pi)^2}{\Delta^2 h^2} \log \pa{\frac{\eta}{\epsilon^2}}
\end{align}

\subsubsection{An improved technique based on conductance profile} \label{s:improved}
Lemma~\ref{l:mtub} and Lemma~\ref{l:main-cond} comprise the standard conductance-based method for bounding mixing times of Markov chains in general state space, which will result in logarithmic dependence on the warmness parameter (see equation~\eqref{e:mixingtime-upperbound}). Building upon this, \cite{chen2020fast} proposes a technique that leads to mixing time guarantees with double-logarithmic dependence on the warmness parameter. This is a significant improvement especially when the warmness parameter depends exponentially on dimension. The new technique avoids introducing additional polynomial dependence in $n$ or $d$ in this case. 

Instead of requiring the target distributions to satisfy a Cheeger-type isoperimetric inequality, the new technique applies to distributions satisfying a  
log-isoperimetric inequality. Formally, a distribution $\pi$ in $\mathbb{R}^d$ satisfies the \textit{log-isoperimetric inequality} with constant $\Ch_{1/2}(\pi)$ if for any measurable partition $\mathbb{R}^d=S_1 \sqcup S_2 \sqcup S_3$, we have 
\begin{align} \label{e:log-iso}
    \pi(S_3) \ge \frac{1}{2 \Ch_{1/2}(\pi)} d(S_1, S_2)  \min\{ \pi(S_1), \pi(S_2)\} \log^{1/2} \pa{1+\frac{1}{\min\{ \pi(S_1), \pi(S_2)\}}}
\end{align}
where $d(S_1, S_2)=\inf\{\|x-y\|_2:x\in S_1, y\in S_2\}$, and this is the minimal such constant. In particular, the class of strongly log-concave distributions satisfies the log-isoperimetric inequality, as shown in the next lemma.
\begin{lemma}[{\cite[Lemma 16]{chen2020fast}}] \label{l:log-iso-strong}
A $m$-strongly log-concave distribution $\pi$ satisfies the log-isoperimetric inequality \eqref{e:log-iso} with constant $\Ch_{1/2}(\pi)=\frac{1}{\sqrt{m}}$. 
\end{lemma}
With a log-isoperimetric inequality, \cite{chen2020fast} adapts the proof of Lemma~\ref{l:main-cond} to lower bound the whole spectrum of conductance instead of the worst-case conductance. Specifically, they derive a lower bound for the \textit{conductance profile} defined as 
$$
\Phi(v) = \inf_{\pi(A) \in (0,v]} \frac{\int_A \mathcal{P}(u, A^c)\pi(u)du}{\pi(A)} \quad \text{ for any $v \in \left(0, \frac{1}{2}\right]$}.
$$
One can see that the standard conductance in equation~\eqref{e:cond} is indeed the conductance profile with $v=\frac{1}{2}$ 
and is the least possible conductance profile over $(0, \frac{1}{2}]$. The next lemma states the lower bound on the conductance profile they obtain. 
\begin{lemma}[{\cite[Lemma 4]{chen2020fast}}]\label{l:main-improved}
Consider a Markov chain on $\mathbb{R}^d$ with transition kernel $\mathcal{P}$ and stationary distribution $\pi$ satisfying the following conditions:
\begin{enumerate}
    \item (Log-Isoperimetry) $\pi$ satisfies a log-isoperimetric inequality~\eqref{e:log-iso} with $\Ch_{1/2}(\pi)>0$.
    \item 
(One-step overlap) For all $x,y \in \mathbb{R}^d$ satisfying $\|x-y\|_2 \le \Delta,$ we have $\TV(\mathcal{P}_x,\mathcal{P}_y) \le 1-h$.
\end{enumerate}    
Then, the conductance profile of the Markov chain satisfies
$$
\Phi(v) = \Omega\Bigg(\frac{h\Delta}{\Ch_{1/2}(\pi)}\log^{1/2}\pa{1+\frac{1}{v}}\Bigg) \quad \text{ for any $ v \in \left(0, \frac{1}{2}\right]$.}
$$
\end{lemma}

Similar to conductance, the conductance profile can be used to upper bound the mixing time. This is formally stated in the next lemma, which utilizes the 
\textit{extended conductance profile $\tilde{\Phi}(v)$} defined as  $\tilde{\Phi}(v)= \left\{
\begin{array}{ll}
      \Phi(v) & v \in (0, \frac{1}{2}] \\
      \Phi(\frac{1}{2}) & v \in [\frac{1}{2}, \infty) \\
\end{array} \right.
.$

\begin{lemma}[{Modified Version of \cite[Lemma 3]{chen2020fast}}] \label{l:mixingtime-upperbound-improved} Consider a reversible, irreducible, and smooth\footnote{In our cases, the existence of a transition kernel guarantees the Markov chain to be smooth. We refer readers to \cite{chen2020fast} for the formal definition of smoothness.} Markov chain $\Psi$ with nonnegative spectrum and stationary distribution $\pi$. Then, for any error tolerance $\epsilon>0$, and a $\eta$-warm distribution, the mixing time of the chain in $\chi^2$ is bounded as 
$
t^{\chi^2}_{\Psi}(\eta, \epsilon) \le  \int_{4/\eta}^{8/\epsilon} \frac{ 16 d v}{v \tilde{\Phi}^2(v)}.
$
\end{lemma}

See Appendix \ref{a:mtub-nolazy-log} for the proof of Lemma~\ref{l:mixingtime-upperbound-improved}.

One can further lower bound the conductance profile in Lemma~\ref{l:main-improved} by $\Omega(\frac{h \Delta}{\Ch_{1/2}(\pi)} \log^{1/2}(\frac{1}{v}))$ and apply it to Lemma~\ref{l:mixingtime-upperbound-improved}, which implies the following useful bound on the mixing time: 
\begin{align}\label{e:mixingtime-upperbound-improved}
    t^{\textup{d}}_{\Psi}(\eta, \epsilon) \le c \frac{\Ch_{1/2}^2(\pi)}{\Delta^2 h^2}\log \pa{\frac{\log \eta}{\epsilon}} \quad \text{ for $\textup{d} \in \{\TV, \KL, \chi^2\}$.}
\end{align}
Here, we use $t^{\TV}_{\Phi}(\eta, \epsilon) \le t^{\chi^2}_{\Phi}(\eta, 4\epsilon^2)$ and $t_\Phi^{\KL}(\eta, \epsilon) \le t^{\chi^2}_{\Phi}(\eta, \epsilon)$,  which follow from the inequalities $2\TV(\mu_1, \mu_2) \le \sqrt{\chi^2(\mu_1||\mu_2)}$ and $\KL(\mu_1||\mu_2) \le \chi^2(\mu_1||\mu_2)$.

In the following sections, after proving Corollary~\ref{t:general}, we dive into the proofs for mixing time upper bound for ProbitDA, LogitDA, and LassoDA. Thanks to the strong log-concavity and Lemma~\ref{l:log-iso-strong}, we can use the improved technique in Section~\ref{s:improved} for ProbitDA and LogitDA. We turn to the standard method in Section~\ref{s:cond} to analyze LassoDA. 


\subsection{Proof of Corollary~\ref{t:general}}\label{s:general-proof}
\begin{proof}
Suppose Condition 2b is true, for any $\beta_1, \beta_2 \in \mathbb{R}^d$ s.t. $\|\beta_1-\beta_2\|_2\le \Delta$, using the data processing inequality, we have $\TV(\mathcal{P}_{\beta_1}, \mathcal{P}_{\beta_2}) = \TV(\delta_{\beta_1}\mathcal{P}_1 \mathcal{P}_2, \delta_{\beta_1}\mathcal{P}_1\mathcal{P}_2)\le \TV(\delta_{\beta_1}\mathcal{P}_1, \delta_{\beta_1}\mathcal{P}_1) \le 1-h$, giving the  one-step overlap condition. Condition 2a implies Condition 2b, and thus is sufficient for the one-step overlap condition.

\end{proof}

\subsection{Proof of Theorem~\ref{t:probitlogit}} 
\label{s:probit-proof}
We prove the bound for ProbitDA and LogitDA separately. 
\subsubsection{Proof of Theorem~\ref{t:probitlogit}: ProbitDA}
\begin{proof} 
The proof will be structured as verifying the two conditions in Lemma~\ref{l:main-improved} and then applying equation~\eqref{e:mixingtime-upperbound-improved}. 
\paragraph*{Log-isoperimetry} The posterior $\pi^{\Probit} \propto e^{-f^{\Probit}}$, defined in Equation~\eqref{e:probit-pos} and Equation~\eqref{e:logit-pos}, are strongly log-concave, which will be clear shortly. We will, therefore, establish log-isoperimetry of $\pi^{\Probit}$ using Lemma~\ref{l:log-iso-strong}. This requires us to calculate a lower bound of the minimum eigenvalue of the Hessian of $f^{\Probit}$, or $\lambda_{\min}(\nabla^2f^{\Probit})$. 

Let $\phi(x)$ be the standard Gaussian pdf at $x$. Noting that $\phi^{\prime}(x)=-x\phi(x)$, we have
\begin{align}
\nabla f^{\Probit}(\beta) &= -\sum_{i=1}^{n} y_i x_i\frac{\phi(x_i^T\beta) }{\Phi(x_i^T\beta)} + \sum_{i=1}^{n}(1-y_i) x_i\frac{\phi(x_i^T\beta)}{1-\Phi(x_i^T\beta)} + B^{-1}(\beta-b) \label{e:probit-log-gradient}\\
\nabla^2 f^{\Probit}(\beta) &= - \sum_{i=1}^{n} y_i x_i\pa{\frac{-x_i^T\beta\phi(x_i^T\beta)x_i\Phi(x_i^T\beta)-\phi^2(x_i^T\beta)x_i}{\Phi(x_i^T\beta)^2}}^T \nonumber\\
&\quad + \sum_{i=1}^{n} (1-y_i) x_i\pa{\frac{-x_i^T\beta\phi(x_i^T\beta)x_i(1-\Phi(x_i^T\beta))+\phi^2(x_i^T\beta)x_i}{(1-\Phi(x_i^T\beta))^2}}^T + B^{-1} \nonumber\\ 
&= \sum_{i=1}^{n} y_i \pa{\frac{\phi^2(x_i^T\beta)}{\Phi^2(x_i^T\beta)}+x_i^T\beta\frac{\phi(x_i^T\beta)}{\Phi(x_i^T\beta)}}x_i x_i^T \nonumber\\ & \quad + \sum_{i=1}^{n} (1-y_i)\pa{\frac{\phi^2(-x_i^T\beta)}{\Phi^2(-x_i^T\beta)}-x_i^T\beta\frac{\phi(-x_i^T\beta)}{\Phi(-x_i^T\beta)}} x_i x_i^T + B^{-1}. \nonumber\\
&= \sum_{i=1}^{n} y_i q(x_i^T\beta)x_i x_i^T + \sum_{i=1}^{n} (1-y_i)q(-x_i^T\beta)x_i x_i^T + B^{-1}. \nonumber 
\end{align}
where the quantity 
\begin{align}
\label{e:q}
q(x)&=\frac{\phi^2(x)}{\Phi^2(x)}+x\frac{\phi(x)}{\Phi(x)}
\end{align}
is the negative derivative of the inverse Mill's ratio of the standard normal distribution, which is bounded between $(0,1)$ \cite{sampford1953some}. This implies $\lambda_{\min}(\nabla^2f^{\Probit}) \ge \lambda_{\min}(B^{-1})$.

Indeed, $\pi^{\Probit}$ is strongly log-concave, because $$\lambda_{\min}(\nabla^2 f^{\Probit}) \ge \lambda_{\min}(B^{-1}) = \frac{1}{\sqrt{\lambda_{\max}(B)}} > 0. $$ By Lemma~\ref{l:log-iso-strong}, we have
$$
\Ch_{1/2}(\pi^{\Probit})\le \frac{1}{\sqrt{\lambda_{\min}(\nabla^2 f^{\Probit})}} \le \sqrt{\lambda_{\max}(B)}.
$$

\paragraph*{One-step overlap}  Let $\mathcal{P}_1$ be the first-step kernel as defined in Corollary~\ref{t:general}. Consider $\beta_1, \beta_2 \in \mathbb{R}^d$. We have 
\begin{align*}
\TV\pa{\mathcal{P}^{\Probit}_{\beta_1}, \mathcal{P}_{\beta_2}^{\Probit}} & \overset{(i)}{\le} \TV\pa{\delta_{\beta_1}\mathcal{P}^{\Probit}_1, \delta_{\beta_2}\mathcal{P}^{\Probit}_1}  \overset{(ii)}{\le} \sqrt{\frac{1}{2} \KL\pa{\delta_{\beta_1}\mathcal{P}^{\Probit}_1||\delta_{\beta_2}\mathcal{P}^{\Probit}_1)}} \\&\overset{(iii)}{=} \sqrt{\frac{1}{2} \sum_{i=1}^n \KL(TN(x_i^T\beta_1, 1; y_i)||TN(x_i^T\beta_2,1; y_i))} 
\end{align*}
where we obtain (i) by data processing inequality (DPI), (ii) by Pinsker's inequality, and (iii) by independence of auxiliary variables. This reduces the problem to studying the KL divergence of 1-dimensional distributions: $\KL(TN(x_i^T\beta_1, 1; y_i)||TN(x_i^T\beta_2,1; y_i))$. 

First, we consider $y_i=1$. Below, $\mathbb E_{\beta_1}$ denotes the expectation taken over $x\sim TN(x_i^T \beta_1, 1;1)$. \begin{align*}
    &\KL(TN(x_i^T\beta_1, 1; 1)||TN(x_i^T\beta_2,1; 1))= \mathbb E_{\beta_1} \log \pa{\frac{\frac{e^{-\frac{1}{2}(x-x_i^T\beta_1)^2}}{\sqrt{2 \pi} \Phi(x_i^T\beta_1)} \mathbbm{1}\{x \ge 0\}}{\frac{e^{-\frac{1}{2}(x-x_i^T\beta_2)^2}}{\sqrt{2 \pi} \Phi(x_i^T\beta_2)} \mathbbm{1}\{x \ge 0\}}} \\
    &= \mathbb E_{\beta_1} \ba{ -\frac{1}{2}(x-x_i^T\beta_1)^2 - \log\Phi(x_i^T\beta_1) + \frac{1}{2}(x-x_i^T\beta_2)^2 + \log\Phi(x_i^T\beta_2)} \\
    &= \log\Phi(x_i^T\beta_2) - \log\Phi(x_i^T\beta_1) + \mathbb E_{\beta_1} \ba{\frac{(x-x_i^T\beta_1 + x_i^T\beta_1 -x_i^T\beta_2 )^2-(x-x_i^T\beta_1)^2}{2}} \\
    &= \log\Phi(x_i^T\beta_2) - \log\Phi(x_i^T\beta_1) + x_i^T(\beta_1-\beta_2) \mathbb E_{\beta_1} \ba{x-x_i^T\beta_1} + \frac{1}{2}(\beta_1-\beta_2)^Tx_i x_i^T (\beta_1-\beta_2) \\
    &=\log\Phi(x_i^T\beta_2) - \log\Phi(x_i^T\beta_1) + x_i^T(\beta_1-\beta_2)\frac{\phi(x_i^T \beta_1)}{\Phi(x_i^T \beta_1)} + \frac{1}{2}(\beta_1-\beta_2)^Tx_i x_i^T (\beta_1-\beta_2) .
\end{align*} The last equation comes from the fact that $\mathbb E_{\beta_1}[x]=x_i^T\beta_1 + \frac{\phi(x_i^T \beta_1)}{\Phi(x_i^T \beta_1)}$. To study the dependency on $\|\beta_1-\beta_2\|_2$, we define the unit vector $u=\frac{\beta_1-\beta_2} {\|\beta_1-\beta_2\|_2}$ and a function $f_i(t)=\log \Phi(x_i^T(\beta_2+ut))$. One can check that $f_i(0)=\log \Phi(x_i^T\beta_2)$ and $f_i(\|\beta_1-\beta_2\|_2)=\log \Phi(x_i^T\beta_1)$. By taking the second-order Taylor expansion of $f_i(t)$ at $t=\|\beta_1-\beta_2\|_2$,  we have that there exists $t_i \in [0, \|\beta_1-\beta_2\|_2]$ such that \begin{align*}
\log &\Phi(x_i^T\beta_2) = \\&\log\Phi(x_i^T\beta_1) + \frac{\phi(x_i^T\beta_1)}{\Phi(x_i^T\beta_1)}x_i^T(\beta_2-\beta_1)-\frac{1}{2}q(x_i^T(\beta_2+ut_i))(\beta_1-\beta_2)^Tx_ix_i^T(\beta_1-\beta_2)
\end{align*} where $q$ is defined in \eqref{e:q}. 
Plugging this back into the KL divergence formula gives \begin{align}
    \KL(TN(x_i^T\beta_1, 1; 1)||TN(x_i^T\beta_2,1; 1)) &= \frac{1}{2}(1-q(x_i^T(\beta_2+ut_i))) (\beta_1-\beta_2)^Tx_i x_i^T(\beta_1-\beta_2) \nonumber\\ &\overset{(i)}{\le} \frac{1}{2} (\beta_1-\beta_2)^Tx_i x_i^T(\beta_1-\beta_2) \label{e:probit-kl1}
\end{align}
where $(i)$ is due to $q(x) \in (0,1)$ for all $x$. We can derive a similar formula for $y_i=0$: for some $t_i\in [0, \|\beta_1-\beta_2\|_2]$, 
\begin{align}
\KL(TN(x_i^T\beta_1, 1; 0)||TN(x_i^T\beta_2,1; 0)) &=\frac{1}{2}(1-q(-x_i^T(\beta_2+ut_i))) (\beta_1-\beta_2)^Tx_i x_i^T(\beta_1-\beta_2)\nonumber\\ &\le \frac{1}{2} (\beta_1-\beta_2)^Tx_i x_i^T(\beta_1-\beta_2).\label{e:probit-kl2}
\end{align}

Combining Equations \eqref{e:probit-kl1} and \eqref{e:probit-kl2}, we write the upper bound of $\TV(\mathcal{P}^{\Probit}_{\beta_1}, \mathcal{P}^{\Probit}_{\beta_2})$ as 
\begin{align*}
\TV(\mathcal{P}^{\Probit}_{\beta_1}, \mathcal{P}^{\Probit}_{\beta_2}) & \le c\sqrt{\sum_{i=1}^n (\beta_1-\beta_2)^Tx_i x_i^T(\beta_1-\beta_2)}= c \sqrt{(\beta_1-\beta_2)^TX^TX(\beta_1-\beta_2)} \\
&\le c \sqrt{\lambda_{\max}(X^TX)}\|\beta_1-\beta_2\|_2.
\end{align*}

If we choose $\Delta=\frac{1}{2 c \sqrt{\lambda_{\max}(X^TX)} }$ and $h=\frac{1}{2}$, we have $\TV(\mathcal{P}_{\beta_1}, \mathcal{P}_{\beta_2}) \le 1-h$ whenever $\|\beta_1-\beta_2\|_2 \le \Delta$. Theorem~\ref{t:probitlogit} follows if we substitute $\Delta=\frac{1}{2 c \sqrt{\lambda_{\max}(X^TX)}}$, $h=\frac{1}{2}$, and $\Ch_{1/2}(\pi) \le \sqrt{\lambda_{\max}(B)}$ into equation~\eqref{e:mixingtime-upperbound-improved}.
\end{proof}
\subsubsection{Proof of Theorem~\ref{t:probitlogit}: LogitDA}
\begin{proof} The proof proceeds similarly to ProbitDA's. 
\paragraph*{Log-isoperimetry}
We first show that  $\pi^{\Logit}$, defined in~\eqref{e:logit-pos}, is strongly log-concave. We have 
\begin{align}
\nabla f^{\Logit}(\beta) &= -X^Ty + \sum_{i=1}^n \frac{e^{x_i^T\beta}}{1+e^{x_i^T\beta}} x_i + B^{-1}(\beta-b) \label{e:logit-log-gradient}\\
\nabla^2 f^{\Logit}(\beta) &= \sum_{i=1}^n \frac{e^{x_i^T\beta}}{(1+e^{x_i^T\beta})^2}x_i x_i^T + B^{-1} = \sum_{i=1}^n \ba{\frac{1}{4}-\pa{\frac{e^{x_i^T\beta}}{1+e^{x_i^T\beta}} - \frac{1}{2}}^2 } x_i x_i^T +B^{-1} . \nonumber
\end{align}
Since $\frac{e^{x_i^T\beta}}{1+e^{x_i^T\beta}} \in (0,1)$, we can obtain $\lambda_{\min}(\nabla^2 f^{\Logit}) \ge \lambda_{\min}(B^{-1}) = 1/\lambda_{\max}\pa{B} >0$. By Lemma~\ref{l:known}, $\Ch_{1/2}(\pi^{\Logit}) \le \sqrt{\lambda_{\max}(B)}$
\paragraph*{One-Step Overlap} Similar to the ProbitDA's proof, we can obtain that 
$\TV(\mathcal{P}^{\Logit}_{\beta_1}, \mathcal{P}^{\Logit}_{\beta_2}) \le \sqrt{\frac{1}{2}\sum_{i=1}^n \KL(PG(1, x_i^T\beta_1)||PG(1,x_i^T\beta_2))}$. Below, $\mathbb E_{\beta_1}$ is the expectation taken over $PG(1, x_i^T\beta_1)$. Applying equations~\eqref{e:pg-density} and~\eqref{e:pg-mean}, we have $E_{\beta_1}\ba{x} = \frac{\tanh(x_i^T\beta_1/2)}{2x_i^T\beta_1}$ and 
\begin{align}
    &\KL(PG(1, x_i^T\beta_1)||PG(1,x_i^T\beta_2))= \mathbb E_{\beta_1} \log \pa{\frac{e^{-\frac{(x_i^T\beta_1)^2}{2}x} \cosh(\frac{x_i^T\beta_1}{2})f(x;1,0)}{e^{-\frac{(x_i^T\beta_2)^2}{2}x} \cosh(\frac{x_i^T\beta_2}{2})f(x;1,0)}}\nonumber\\
    &= \frac{(x_i^T\beta_2)^2-(x_i^T\beta_1)^2}{2} \mathbb E_{\beta_1}\ba{x} + \log \cosh(\frac{x_i^T\beta_1}{2}) - \log \cosh(\frac{x_i^T\beta_2}{2}) \nonumber \\
    &= \frac{(x_i^T\beta_2-x_i^T\beta_1+x_i^T\beta_1)^2-(x_i^T\beta_1)^2}{4x_i^T\beta_1} \tanh(\frac{x_i^T\beta_1}{2}) + \log \cosh(\frac{x_i^T\beta_1}{2}) - \log \cosh(\frac{x_i^T\beta_2}{2}) \nonumber\\
    &= \pa{\frac{(x_i^T\beta_2-x_i^T\beta_1)^2}{4 x_i^T\beta_1} + \frac{x_i^T(\beta_2-\beta_1)}{2}}\tanh(\frac{x_i^T\beta_1}{2}) + \log \cosh(\frac{x_i^T\beta_1}{2}) - \log \cosh(\frac{x_i^T\beta_2}{2}).\label{e:logit-kl}
\end{align} By Taylor expansion, we obtain that there exists $t_i \in [0, \|\beta_1-\beta_2\|_2]$ such that 
\begin{align*}
    &\log \cosh{\frac{x_i^T\beta_2}{2}}= \\& \; \log \cosh{\frac{x_i^T\beta_1}{2}} + \frac{\tanh(\frac{x_i^T\beta_1}{2})}{2} x_i^T(\beta_2-\beta_1) + \frac{1}{8 \cosh^2(\frac{x_i^T(\beta_2+ut_i)}{2})} (\beta_1-\beta_2)^T x_i x_i^T (\beta_1-\beta_2).
\end{align*} Plugging this back into the KL divergence formula~\eqref{e:logit-kl} yields
\begin{align} \label{e:logit-kl2}
    \KL(&PG(1, x_i^T\beta_1)||PG(1,x_i^T\beta_2))\nonumber\\&=\pa{\frac{\tanh(\frac{x_i^T\beta_1}{2})}{4 x_i^T\beta_1} -\frac{1}{8\cosh^2(\frac{x_i^T(\beta_2+ut_i)}{2})}} (\beta_1-\beta_2)^T x_i x_i^T (\beta_1-\beta_2)
\end{align}
Since $\cosh(x) \ge 1 $ and $ \frac{\tanh{x}}{x} \le 1$, we have $\pa{\frac{\tanh(\frac{x_i^T\beta_1}{2})}{4 x_i^T\beta_1} -\frac{1}{8\cosh^2\pa{\frac{x_i^T(\beta_2+ut_i)}{2}}}} \le \frac{1}{8}$. This gives, $\TV(\mathcal{P}^{\Logit}_{\beta_1}, \mathcal{P}^{\Logit}_{\beta_2}) \le c\sqrt{\lambda_{\max}(X^TX)}\|\beta_1-\beta_2\|_2$. We can, therefore, conclude the same as ProbitDA.  
\end{proof}

\subsection{Proof of Theorem~\ref{t:lasso}} \label{s:lasso-proof}
Direct analysis of the LassoDA could be complicated. Instead, we consider a one-to-one transformation of the Markov chain underlying LassoDA. The transformation simplifies the problem in two ways: (1) it makes the non-log-concave target of LassoDA log-concave, and (2) it simplifies the transition kernel. 

Next, we make precise the notion of transformation of a Markov chain. For simplicity of notation, given a Markov chain with state space $\Omega$, we define a \textit{Markov chain triple} as the composite of its target distribution $\pi$, its starting distribution $\nu$, and its transition kernel $\mathcal{P}$, denoted as $(\nu, \mathcal{P}, \pi)$. For any bijective measurable function $T: \Omega \rightarrow \Omega'$, we denote the \textit{$T$-transformed Markov chain} of $\Psi$ by $\Psi_T$. If $\Psi$ is the Markov chain triple $(\nu, \mathcal{P}, \pi)$, then $\Psi_T$ is the triple ($\nu_T$,$\mathcal{P}_T$,$\pi_T$) satisfying 
\begin{align*}
    \pi_{T}=T_{\#}\pi, \quad
    \nu_T=T_{\#}\nu, \quad \text{and} \quad
    \mathcal{P}_T(x,\cdot)=T_{\#}(\delta_{T^{-1}(x)} \mathcal{P}),
\end{align*}
where $\delta_a$ is the Dirac measure centered at $a$, and $T_{\#}\pi$ is the push-forward measure of $\pi$ by $T$. We call $\pi_T$ and $\mathcal{P}_T$ the \textit{$T$-transformed target distribution} and \textit{$T$-transformed transition kernel}, respectively. To validate the analysis under a transformed Markov chain, we establish the equivalence of the mixing time under one-to-one transformation in the following lemma. 

\begin{lemma} \label{l:transform} Suppose we have a Markov chain $\Psi$ on $\Omega$ with transition kernel $\mathcal P$ and stationary distribution $\pi$, and a bijection $T:\Omega \rightarrow \Omega'$. For any error tolerance $\epsilon \in (0,1)$ and warmness $\eta \ge 1$ of the initial distributions, we have that $\pi_T = T_\#\pi$ is the stationary distribution of $\mathcal P_T$ and
$t_{\Psi}(\eta, \epsilon)=t_{\Psi_T}(\eta, \epsilon).$
\end{lemma}

See Appendix \ref{a:proof-transform} for the proof of Lemma~\ref{l:transform}.

By Lemma~\ref{l:transform}, we can study the mixing time of the LassoDA on an equivalent one-to-one transformed chain. In particular, we use the same bijective map as in Appendix A of~\cite{park2008bayesian}:   $T:\mathbb{R}^{d}\times\R^+\rightarrow \mathbb{R}^{d}\times \R^+$ that transforms $(\beta, v)$ to a new parameter space $(\varphi, \rho)$ according to 
$$
\varphi=\frac{\beta}{\sqrt{v}}, \quad \rho=\frac{1}{\sqrt{v}}.
$$

We first analyze the effects of the transformation $T$ on the target and Markov transition kernel. Then, we develop an upper bound of the mixing time for the transformed Markov chain using the standard conductance-based argument introduced in Section~\ref{s:cond}.  

\paragraph*{$T$-transformed target distribution of LassoDA} In order to simplify notation, we drop the superscripts $^{\Lasso}$ from our notation for the rest of this section. We recall from~\eqref{e:lasso-pos} that the (non-log-concave) LassoDA target is
\begin{align*}
    \pi(v, \beta|y)  
    \propto \frac{1}{v^{(n+d+2\alpha+2)/2}} e^{-\frac{1}{2v}\|\tilde{y}-X\beta\|_2^2-\lambda\frac{\|\beta\|_1}{\sqrt{v}}-\frac{\xi}{v}}.
\end{align*}
Next, we will show that the transformation by $T$ makes a log-concave target. We have that
\begin{align*} \det(\nabla T^{-1}) = \det \begin{bmatrix}
        \frac{1}{\rho} & \cdots & 0 & 0\\ 
         \vdots & \ddots &\vdots & \vdots \\
         0 &  \cdots& \frac{1}{\rho} & 0 \\
         -\frac{\varphi_1}{\rho^2} &  \cdots & -\frac{\varphi_d}{\rho^2} & -\frac{2}{\rho^3}
    \end{bmatrix} = -\frac{2}{\rho^{3+d}}.
\end{align*} 
The $T$-transformed LassoDA target is
\begin{align} \label{e:transformed-target}
    \pi_T(\varphi, \rho|y) &\propto \rho^{n+2\alpha + d+1} \exp\pa{-\frac{1}{2}\|\rho y - X\varphi\|_2^2 - \lambda \|\varphi\|_1-\rho^2 \xi} |\det(\nabla T^{-1})| \nonumber\\
    &\propto  \rho^{n+2\alpha -2} \exp\pa{-\frac{1}{2}\|\rho y - X\varphi\|_2^2 - \lambda \|\varphi\|_1-\rho^2 \xi}.
\end{align}
It is not hard to see $\pi_T$ is log-concave for $n \ge 2-2\alpha$.  

\paragraph*{$T$-transformed transition kernel of LassoDA} 
The transformation also largely simplifies the Markov transition kernel. We claim that given the special structure of $T$-transformed LassoDA's kernel, it suffices to study the $\varphi$-marginal chain of the $T$-transformed LassoDA.

The $T$-transformed LassoDA's kernel $\mathcal{P}_T$ is illustrated below: 
\begin{align} \label{e:transform-kernel}
    (\varphi^{(m-1)}, \rho^{(m-1)}) \xrightarrow{T^{-1}} \underbrace{(\beta^{(m-1)}, v^{(m-1)}) \rightarrow z^{(m)} \rightarrow (\beta^{(m)}, v^{(m)})}_{\text{The original kernel of LassoDA}} \xrightarrow{T} (\varphi^{(m)}, \rho^{(m)}).
\end{align}
We first note that in \eqref{e:transform-kernel}, $z^{(m)}$ is sufficient for $\varphi^{(m)}$ and $\rho^{(m)}$. Furthermore, one can show that  $z^{(m)}$ depends only on $\varphi^{(m-1)}$, and is independent of $\rho^{(m-1)}$, because 
\begin{align}
z^{(m)} = \IG\pa{\sqrt{\frac{\lambda^2v^{(m-1)}}{(\beta^{(m-1)})^2}}, \lambda^2} = \IG\pa{\sqrt{\frac{\lambda^2}{(\varphi^{(m-1)})^2}}, \lambda^2}.
\end{align}
These altogether imply that the $\varphi$-sample is sufficient to generate next-step $\rho$ and $\varphi$ on the $T$-transformed LassoDA. The transformed kernel is illustrated in Figure \ref{fig:transform-kernel}. The structure has the following important implications.

First, the independence of $z^{(m)}$ on $\rho^{(m-1)}$ ensure that the $\varphi$-marginal chain is well-defined. Specifically, we use $\pa{\nu_{T_\varphi}, \mathcal{P}_{T_\varphi}, \pi_{T_\varphi}}$ to denote the Markov chain triple of the $\varphi$-marginal chain of the $T$-transformed LassoDA $\Psi_{T_\varphi}$.

Second, the sufficiency of $\varphi$ for the next-step $\rho$ enables us to control the mixing time of the $\Psi_T$ by that of $\Psi_{T_\varphi}$. To demonstrate the sufficiency of the $\varphi$-marginal chain, we consider another Markov chain that evolves according to the same kernel as in equation~\eqref{e:transform-kernel}, but starts from the stationary distribution $\pi_T$. 
Then the chain will remain at the distribution $\pi_T$. We use a subscript $\pi$ to indicate the samples are from this stationary chain. 

Using $\mathcal{P}_{T_{\varphi \rightarrow \rho}}$ to denote the transition kernel from $\varphi^{(m-1)}$ to $\rho^{(m)}$, we have that,
\begin{align*}   \TV\pa{\nu_T\mathcal{P}_T^m, \pi_T} &= \TV\pa{\pa{\varphi^{(m)}, \rho^{(m)}}, \pa{\varphi^{(m)}_\pi, \rho^{(m)}_\pi}} \\&\overset{(i)}{\le} \TV\pa{\varphi^{(m-1)}, \varphi_\pi^{(m-1)}} = \TV\pa{\nu_{T_\varphi}\mathcal{P}_{T_\varphi}^{m-1}, \pi_{T_\varphi}}
\end{align*}
where (i) is due to data processing equality. Overall, we have 
\begin{align} \label{e:suffice}
 \TV \pa{\nu \mathcal P^m, \pi} = \TV \pa{\nu_T \mathcal P^m_T, \pi_T} \le \TV \pa{\nu_{T_\varphi} \mathcal P_{T_\varphi}^{m-1}, \pi_{T_\varphi}}.
\end{align}
Equation \eqref{e:suffice} gives us a way to control the mixing time of the LassoDA by that of $\varphi$-marginal of its $T$-transformed chain.
Therefore, studying the mixing times of $\Psi_{T_\varphi}$ is sufficient. 

\begin{figure}[t]
    \centering
\tikzstyle{para} = [circle, node distance=3em, minimum height=3 em, minimum width=3em]
\tikzstyle{empty} = [circle, node distance=6em, minimum height=3 em, minimum width=3em]
\tikzstyle{latent} = [circle, fill=gray!40, node distance=3em, minimum height=3em, minimum width=3em]
\tikzstyle{block} = [rectangle, draw, fill=gray!20, text width=3.5em, text centered, minimum height=4em]
\tikzstyle{line} = [draw, -latex']
\begin{tikzpicture}[node distance = 4.5em, auto]
\node [block, right of=startmid] (latent) {$ z_1^{(m+1)}$ \\ $\, \vdots$ \\ $\, \vdots$ \\$z_d^{(m+1)}$};
\node [empty, left of=latent] (startmid) {};
\node [para, above of=startmid] (start) {$\rho^{(m)}$};
\node [para, below of=startmid] (start2) {$\varphi^{(m)}$};

\node [empty, right of=latent] (endmid) {};
\node [para, above of=endmid] (end) {$v^{(m+1)}$};
\node [para, below of=endmid] (end2) {$\beta^{(m+1)} $};
\node [empty, right of=endmid] (endendmid) {};
\node [para, above of=endendmid] (endend) {$\rho^{(m+1)}$};
\node [para, below of=endendmid] (endend2) {$\varphi^{(m+1)} $};
\path [line,dashed] (start2) -- (latent);
\path [line] (latent) -- (end);
\path [line] (end) -- (end2);
\path [line] (latent) -- (end2);
\path [line] (end2) -- (endend2);
\path [line] (end2) -- (endend);
\path [line] (end) -- (endend);

\end{tikzpicture} 
\caption{Illustration of the kernel of $T$-transformed LassoDA}
\label{fig:transform-kernel}
\end{figure}

\paragraph*{Mixing time of the $T$-transformed chain}  We perform the analysis using the standard conductance-based method in Section~\ref{s:cond}. For clarity, we extract the two main parts of the proof as lemmas below, and defer their proofs.

\begin{lemma}\label{l:lasso-iso}(Isoperimetry of $\pi_{T_\varphi}$) The Cheeger constant of the $\varphi$-marginal of the $T$-transformed LassoDA's target satisfies 
$$   \Ch(\pi_{T_\varphi}) = c(d \log d  + n \log n),$$ 
where $c$ is a constant depending on $M$, $\lambda$, and $\xi$.
\end{lemma}

We use Lemma~\ref{l:transfer} with $\mu_1 \propto e^{-\lambda \|\varphi\|_1} \, \mu_2=\pi_{T_\varphi}$ and Lemma~\ref{l:known} (1) to prove Lemma~\ref{l:lasso-iso}. The proof is deferred to Appendix \ref{s:proof-lasso-iso}. Although Lemma~\ref{l:lasso-iso} gives a polynomial bound in $n$ and $d$ for $\Ch(\pi_{T_\varphi})$, we suspect the bound is not tight. The distribution $\pi_{T_\phi}$ can be viewed as a log-concave perturbation of the double-exponential measure, i.e., $\pi_{T_\phi} \propto e^{-\lambda |\varphi|_1 - V(\varphi)}$ with convex $V(\varphi)$. Intuitively, such perturbations may enhance log-concavity, suggesting $\Ch(\pi_{T_\varphi})$ remains close to that of the double-exponential measure, which is constant in $n$ and $d$ by Lemma~\ref{l:known}(1). Indeed, \cite{barthe2019spectral} shows only a logarithmic deterioration when $V$ is even (see also \cite{cattiaux2022functional}). Although this does not yet cover the Lasso case, we hope further progress in this direction will yield sharper bounds for $\Ch(\pi_{T_\phi})$.

\begin{lemma} \label{l:lasso-overlap} (One-step overlap of $\Psi_{T_\varphi}$) The transition kernel of $\varphi$-marginal of the $T$-transformed LassoDA satisfies
$$
\TV\pa{\pa{\mathcal{P}_{T_\varphi}}_x, \pa{\mathcal{P}_{T_\varphi}}_y} \le \frac{1}{2} \text{ whenever } x,y \in \mathbb{R}^d \text{ and } \|x-y\|_2 \le \frac{c}{d},
$$
where $c$ is a universal constant. 
\end{lemma}
See Section~\ref{s:proof-lasso-overlap} for the proof of Lemma~\ref{l:lasso-overlap}. 

Using Lemma~\ref{l:lasso-iso}, Lemma~\ref{l:lasso-overlap}, and Lemma~\ref{l:main-cond}, we can obtain a lower bound on the conductance of the $\Psi_{T_\varphi}$ such that 
$
\Phi \ge c \frac{1}{d(d\log d+n\log n)}.
$ Then, Lemma~\ref{l:mtub} and \eqref{e:suffice} implies that with a $\eta$-warm start, we have
$
\TV(\nu \mathcal{P}^k, \pi) \le \frac{1}{2}\sqrt{\eta} e^{-c\frac{k-1}{d^2(d\log d +n\log n)^2}}.
$ To guarantee that $\TV(\nu \mathcal{P}^k, \pi)$ is within $\epsilon$, it suffice to ensure that $\sqrt{\eta} e^{-c\frac{k-1}{d^2(d\log d +n\log n)^2}} \le \epsilon$ or $k \ge 1+c d^2 (d\log d + n\log n)^2\log \pa{\frac{\sqrt{\eta}}{\epsilon}}$. Therefore, the mixing time of the LassoDA satisfies  
$
t_{\Psi^{\Lasso}}(\eta, \epsilon) \le c d^2(d \log d + n\log n)^2 \log \pa{\frac{\eta}{\epsilon}}.
$
Theorem~\ref{t:lasso} follows.

\subsubsection{Proof of Lemma~\ref{l:lasso-overlap}} \label{s:proof-lasso-overlap}
When studying the one-step overlap condition for ProbitDA and LogitDA, we upper bound the TV distance of the latent variables by the KL divergence for ease of calculation. This is not possible for the LassoDA at some extreme parameter values, as the KL divergence of the auxiliary inverse Gaussian random variables diverges. We use the following lemma to deal with the extreme cases. Intuitively, the lemma characterizes the limiting behavior of the IG variable with a growing mean and a fixed shape parameter. 

\begin{lemma} \label{l:tv_ig} Suppose $\mu_1, \mu_2, \lambda>0$. Then,
$
\TV(\IG(\mu_1, \lambda), \IG(\mu_2, \lambda)) \le \sqrt{\frac{4\lambda}{\pi \min\{\mu_1,\mu_2\}}}.
$
\end{lemma}
The proof of Lemma~\ref{l:tv_ig} can be found in Appendix \ref{s:ig-proof}.  

\begin{proof}[Proof of Lemma~\ref{l:lasso-overlap}]  For simplicity, we use $\mathcal{P}^\prime$ to denote $\mathcal{P}_{T_\varphi}$.
For any $\varphi_1,\varphi_2 \in \mathbb{R}^d$, let $z_i$ be the latent IG variables chosen for $\varphi_i$, $i=1,2$. By data processing inequality, we have 
\begin{align*} 
    \TV(\mathcal{P}_{\varphi_1}^{\prime}, \mathcal{P}_{\varphi_2}^{\prime}) &\le \TV(z_1, z_2)  \nonumber =  \TV\left(\left\{\IG\left(\frac{\lambda}{|\varphi_{1j}|}, \lambda^2\right)\right\}_{j=1}^d, \left\{\IG\left(\frac{\lambda}{|\varphi_{2j}|}, \lambda^2\right)\right\}_{j=1}^d\right).
\end{align*}
The TV distance of IG variables does not have a closed form. We can upper bound it by the KL divergence using Pinsker's inequality, as in the analysis of ProbitDA and LogitDA. We begin by showing that this is feasible only when either $|\varphi_{1j}|$ or $|\varphi_{2j}|$ is large. Below, $\mathbb{\mathop{E}}_{\varphi_{1j}}$ denotes the expectation taken over $\IG(\frac{\lambda}{|\varphi_{1j}|}, \lambda^2)$. 
Let $\mu_{ij} = \frac{\lambda}{|\varphi_{ij}|}$ for $1,2$. Using the fact $\mathbb {\mathop{E}}_{\varphi_{1j}} x = \frac{\lambda}{|\varphi_{1j}|}$, we have that \begin{align*}
& \KL\pa{\IG\left(\frac{\lambda}{|\varphi_{1j}|}, \lambda^2\right) \Big|\Big| \IG\left(\frac{\lambda}{|\varphi_{2j}|}, \lambda^2\right)} = \mathbb{\mathop{E}}_{\varphi_{1j}} \log \pa{\frac{\sqrt{\frac{\lambda^2}{2\pi x^3}} \exp\left\{-\frac{\lambda^2(x-\mu_{1j})^2}{2\mu_{1j}^2 x}\right\}}{\sqrt{\frac{\lambda^2}{2\pi x^3}} \exp\left\{-\frac{\lambda^2(x-\mu_{2j})^2}{2\mu_{2j}^2 x}\right\}}} \\
    &=\lambda^2 \ba{ \left(\frac{1}{2\mu^2_{2j}}-\frac{1}{2\mu^2_{1j}}\right)\mathbb {\mathop{E}}_{\varphi_{1j}} x +\left(\frac{1}{\mu_{1j}}-\frac{1}{\mu_{2j}}\right) } 
    =\lambda \ba{ \left(\frac{\varphi^2_{2j}}{2}-\frac{\varphi^2_{1j}}{2}\right)\frac{1}{|\varphi_{1j}|} +\left(|\varphi_{1j}|-|\varphi_{2j}|\right)} 
\end{align*}
One can see that we cannot use the KL divergence to perform the analysis when both $|\varphi_{1j}|$ and $|\varphi_{2j}|$ are small, as KL divergence diverges in this case. (Either $|\varphi_{1j}|$ or $|\varphi_{2j}|$ being small is sufficient because we can bound TV distance by KL divergence in either direction.) We separate this extreme case and deal with it using the bound in Lemma~\ref{l:tv_ig}. Let $m_j=\max\{|\varphi_{1j}|, |\varphi_{2j}|\}$ for $j=1,\ldots, d$. WLOG, we assume that for some $1\le k \le d$, $m_j \le \frac{C}{d^2}$ for $j=1,\ldots, k$ and $ m_j \ge \frac{C}{d^2}$ for $j=k+1,\ldots, d$, where $C=\frac{\pi}{64 \lambda}$. Then, by the independence of IG variables and Pinsker's inequality, we have 
\small{\begin{align} 
    \nonumber &\TV(\mathcal{P}_{\varphi_1}^\prime, \mathcal{P}_{\varphi_2}^\prime)  \le \TV\pa{\bc{\IG \pa{\frac{\lambda}{|\varphi_{1j}|}, \lambda^2}}_{j=1}^k, \bc{\IG \pa{\frac{\lambda}{|\varphi_{2j}|}, \lambda^2}}_{j=1}^k} \\& \quad \quad\quad \quad \quad\quad \quad+ \TV \pa{ \bc{\IG \pa{\frac{\lambda}{|\varphi_{1j}|}, \lambda^2}}_{j=k+1}^d, \bc{\IG \pa{\frac{\lambda}{|\varphi_{2j}|}, \lambda^2}}_{j=k+1}^d}\nonumber\\
    &\label{e:tv-sum}  \le \sum_{j=1}^k \underbrace{\TV \pa{\IG \pa{\frac{\lambda}{|\varphi_{1j}|}, \lambda^2}, \IG \pa{\frac{\lambda}{|\varphi_{2j}|}, \lambda^2}}}_{\text{The extreme case:  Using Lemma~\ref{l:tv_ig} to bound}} + \underbrace{\sqrt{\frac{1}{2}\sum_{j=k+1}^d \KL \pa{\IG \pa{\frac{\lambda}{|\varphi_{1j}|}, \lambda^2}, \IG\pa{\frac{\lambda}{|\varphi_{2j}|}, \lambda^2}}}}_{\text{The regular case: Using $\|\varphi_{1j}-\varphi_{2j}\|_2$ to bound}}
\end{align}}
\paragraph*{The extreme case}  For $j\le k$, we have that $\max\{|\varphi_{1j}|, |\varphi_{2j}|\} \le \frac{C}{d^2}$. By Lemma~\ref{l:tv_ig}, 
\begin{align} \label{e:lasso-tv}
    \TV \pa{\IG \pa{\frac{\lambda}{|\varphi_{1j}|}, \lambda^2}, \IG\pa{\frac{\lambda}{|\varphi_{2j}|}, \lambda^2}} \le \sqrt{\frac{4\lambda^2}{\pi 
\frac{\lambda}{\max\{|\varphi_{1j}|, |\varphi_{2j}|\}}}} \le \frac{1}{4d}.
\end{align} 

\paragraph*{The regular case} For $j \ge k+1$, WLOG, we assume that $|\varphi_{1j}| \ge |\varphi_{2j}| $, then 
$|\varphi_{1j}| \ge \frac{C}{d^2}$. To control $\KL \pa{\IG \pa{\frac{\lambda}{|\varphi_{1j}|}, \lambda^2} \Big|\Big| \IG \pa{\frac{\lambda}{|\varphi_{2j}|}, \lambda^2}}$ using $|\varphi_{1j}-\varphi_{2j}|$, expand
$
\frac{\varphi^2_{2j}}{2} = \frac{\varphi^2_{1j}}{2} + \varphi_{1j}(\varphi_{2j}-\varphi_{1j}) +\frac{1}{2}(\varphi_{2j}-\varphi_{1j})^2 .
$
Therefore, we have
\begin{align} 
    &\KL \pa{\IG \pa{\frac{\lambda}{|\varphi_{1j}|}, \lambda^2} \Big|\Big| \IG \pa{\frac{\lambda}{|\varphi_{2j}|}, \lambda^2}}  \nonumber\\ &= \lambda \ba{ \frac{1}{2} \frac{1}{|\varphi_{1j}|} (\varphi_{2j}-\varphi_{1j})^2 + \underbrace{\sign(\varphi_{1j})(\varphi_{2j}-\varphi_{1j}) + |\varphi_{1j}|-|\varphi_{2j}|}_{\le 0}} \le \frac{\lambda}{2C} d^2 (\varphi_{2j}-\varphi_{1j})^2. \label{e:lasso-KL}
\end{align}  
Using inequalities~\eqref{e:lasso-tv} and~\eqref{e:lasso-KL} in equation~\eqref{e:tv-sum}, we have 
\begin{align*}
    \TV(\mathcal{P}_{\varphi_1}^{\prime}, \mathcal{P}_{\varphi_2}^{\prime}) \le \frac{k}{4d} + cd \sqrt{\sum_{j=k+1}^d (\varphi_{2j}-\varphi_{1j})^2} \le \frac{1}{4} + cd \| \varphi_1-\varphi_2\|_2.
\end{align*} 
If we choose $\Delta=\frac{1}{4cd}$, $\|\varphi_1-\varphi_2\|_2 \le \Delta$ guarantees that $\TV(\mathcal{P}_{\varphi_1}^{\prime}, \mathcal{P}_{\varphi_2}^{\prime})  \le \frac{1}{2} $. 
\end{proof}

\section{Conclusion and discussion} \label{s:conclude} We adapt the conductance-based method, rooted in a long line of work on mixing times via convex geometry and isoperimetric inequalities, to the structure of DA chains. By using this method, we establish the first fast mixing guarantees for three important DA algorithms (i.e. ProbitDA, LogitDA, and LassoDA). This addresses the non-asymptotic aspect of the long-standing ``convergence complexity" problem \cite{rajaratnam2015mcmc} for the three DA algorithms. 

To conclude, we list a few directions that merit further investigation:

\paragraph*{Lower bounds} Obtaining mixing-time lower bounds for the three DA algorithms remains an interesting open problem, as such bounds would clarify the tightness of our results and enable a fuller comparison with alternative sampling methods.

\paragraph*{Isoperimetric constant and dependency on warmness for LassoDA} In contrast to ProbitDA and LogitDA, analyzing the isoperimetric constant and improving the dependence on the  warmness parameter for LassoDA is substantially more challenging. This is partially because many important underlying techniques that support the analysis for strongly log-concave distributions are not readily carried over to weakly log-concave settings. Specifically, although we have good control of the Cheeger constant over log-concave perturbations of strongly log-concave measures (Lemma~\ref{l:known} (2)), as noted below Lemma~\ref{l:lasso-iso}, comparable results are lacking for perturbations of the double-exponential distribution. Progress in this area could yield sharper bounds on the Cheeger constant of the marginal transformed Lasso target in Lemma~\ref{l:lasso-iso}. Moreover, one can make the dependence on the warmness parameter milder (e.g., double logarithmic) and hence allow good convergence from cold starts, if more results on log-isoperimetric inequalities for weakly log-concave distributions are available. 

Despite these obstacles, we believe our guarantees provide useful insights for empirical studies using DA algorithms. Moreover, we expect them to offer valuable perspectives to general MCMC theory and encourage more research on studying statistically oriented sampling problems under isoperiemetric assumptions. 

\begin{acks} We would like to thank Sam Power 
for enlightening discussions on the conductance method, and for suggesting Lemmas~\ref{l:mtub} and~\ref{l:mixingtime-upperbound-improved}.
\end{acks}
\newpage

\bibliographystyle{imsart-nameyear} 
\bibliography{bib} 




\newpage
\appendix
\section{Mixing time with a feasible start} \label{s:main-cold}
In this appendix, we prove mixing time guarantees for the three DA algorithms starting from known and implementable distributions. 

\subsection{Feasible starts for ProbitDA and LogitDA} From Equations \eqref{e:probit-log-gradient} and \eqref{e:logit-log-gradient}, we can obtain that the posteriors $\pi \; \propto \; e^{-f} $ are strongly log-concave and satisfy  
\begin{align}
    \lambda_{\min}(\nabla^2 f^{\Probit}) &\ge \lambda_{\min}(B^{-1})   \coloneq m^{\prime\Probit} \label{e:probit-convex}\\
    \lambda_{\min}(\nabla^2 f^{\Logit}) &\ge \lambda_{\min}(B^{-1})  \coloneq m^{\prime\Logit} \label{e:logit-convex}\\
        \lambda_{\max}(\nabla^2 f^{\Probit}) &\le \lambda_{\max}(X^TX) + \lambda_{\max}(B^{-1}) \coloneq L^{\prime\Probit}\label{e:probit-smooth}\\
    \lambda_{\max}(\nabla^2 f^{\Logit}) &\le 0.25 \lambda_{\max}(X^TX) + \lambda_{\max}(B^{-1}) \coloneq  L^{\prime\Logit}. \label{e:logit-smooth}
\end{align}
Utilizing the strong log-concavity of ProbitDA and LogitDA target distributions, we adopt the following feasible starting distribution for general strongly log-concave targets $\pi$ in $\mathbb{R}^d$ proposed by \cite{dwivedi2019log}, 
$$
\nu_{\star} = \mathcal{N}\pa{x_\star, \frac{1}{L^\prime}\mathbb{I}_d}
$$
where $x_\star$ is the mode of $\pi$. Following the steps in Section 3.2 of \cite{dwivedi2019log}, one can demonstrate that 
\begin{align} \label{e:feasible-eta}
    \sup_A \frac{\nu_{\star}(A)}{\pi(A)} \le \pa{\frac{L^{\prime}}{m^{\prime}}}^{\frac{d}{2}} = \eta_\star ,
\end{align}
where the supremum is taken over all measurable sets $A \subseteq \mathbb{R}^d$. Using the $m^{\prime}$ and $L^{\prime}$ defined in Equations~\eqref{e:probit-convex}, \eqref{e:logit-convex}, \eqref{e:probit-smooth},  and \eqref{e:logit-smooth}, we can obtain
\begin{align}\label{e:feasible-eta-specific}
    \eta_\star^{\Probit} = \pa{\frac{\|X^TX\| + \|B^{-1}\|}{\lambda_{\min}(B^{-1})}}^{\frac{d}{2}},   \quad \quad \eta_\star^{\Logit} = \pa{\frac{0.25\|X^TX\| + \|B^{-1}\|}{\lambda_{\min}(B^{-1})}}^{\frac{d}{2}} .
\end{align} 
We will consider the same setting as Corollary \ref{t:main-cor}, and assume additionally that $\lambda_{\min}(B) = \Omega(1)$. In this scenario, $\|X^TX\|_{\mathrm{op}}^{d/2}$ dominant the complexity of $\eta_{\star}$.  
We follow the proof of Corollary \ref{t:main-cor} to get high probability bounds for $\|X^TX\|_{\mathrm{op}}^{d/2}$ and substitute them into Corollary \ref{t:main-cor}. This gives the following corollary.

\begin{cor}\label{t:feasible-all}
Consider the same setting as Corollary \ref{t:main-cor}. We assume additionally that $\lambda_{\min}(B)=\Omega(1)$. We have that for any error tolerance $\epsilon \in (0,1)$, the mixing time in metric $\textup{d} \in \{\TV, \KL, \chi^2\}$ of ProbitDA starting from $\mathcal{N}\pa{x^{\Probit}_\star, \frac{\mathbb{I}_d}{\|X^TX\| + \|B^{-1}\|}}$ or LogitDA starting from $\mathcal{N}\pa{x^{\Logit}_\star, \frac{\mathbb{I}_d}{0.25\|X^TX\| + \|B^{-1}\|}}$ satisfies the following.
\begin{enumerate}
\item (Sub-Gaussianity) If $\mathcal{L}$ is sub-Gaussian with sub-Gaussian norm $K$, with probability at least $1-2e^{-u}$,
$$t^{\textup{d}}_{\Psi}(\eta, \epsilon) \le c A \log \pa{ \frac{d\log A }{\epsilon}},$$
where $A=\bc{n + \frac{\|\Sigma\|_{\mathrm{op}}}{d} \ba{n +  c^\prime n K^2 \pa{\sqrt{\frac{d+u}{n}} + \frac{d+u}{n}}}}$.
\item (Log-concavity) If $\mathcal{L}$ is log-concave, with probability at least $1-\exp(-c^{\prime\prime}\sqrt{d})$, 
$$
t^{\textup{d}}_{\Psi}(\eta, \epsilon) \le c B \log \pa{\frac{d\log B}{\epsilon}}, 
$$ 
where $B=\bc{n + \frac{\|\Sigma\|_{\mathrm{op}}}{d} \ba{n +c^\prime n \pa{\sqrt{\frac{d}{n}}+\frac{d}{n}}}}$.
\end{enumerate}
Here, $c, c^\prime, c^{\prime\prime}$ are universal constants.
where $c$ is a universal constant.
\end{cor}

We observe that if we consider $K$, $u$, and $\|\Sigma\|_{\mathrm{op}}$ to be independently of $n$ and $d$, we can get that $A=\mathcal{O}(n)$ and $B=\mathcal{O}(n)$. Therefore, either under sub-gaussian or log-concave assumptions, we can get a $\mathcal{O}\pa{n\log\pa{\frac{d\log n}{\epsilon}}}$ mixing time guarantee for both ProbitDA and LogitDA, with high probability over data.

\begin{rem*} $v_\star$ is a valid feasible start only if we can efficiently compute $x_\star$. \cite{dwivedi2019log} comments that a $\delta$-approximation of $x_\star$ can be obtained in $\mathcal{O}(\kappa \log\frac{1}{\delta})$ steps using standard optimization algorithms such as gradient descent, and discusses how an inexact $x_\star$ affects the mixing time. We refer interested readers to \cite[Section 3.2]{dwivedi2019log} for a detailed discussion. 
In the cases of ProbitDA and LogitDA, $\kappa \le \frac{L^{\prime}}{m^{\prime}} = \mathcal{O}(n)$ under the setting of Corollary~\ref{t:feasible-all} with $K$, $u$, and $\|\Sigma\|_{\mathrm{op}}$ being constant in $n$ and $d$. The computational complexity of optimization does not exceed that of sampling in Corollary~\ref{t:feasible-all}, and thus is ignorable. 
\end{rem*}

\subsection{A feasible start for LassoDA}
One analyzable feasible start for LassoDA is the following:
\begin{align}\label{e:lasso-feasible}
    \nu_\dagger(\beta, v|y) \propto \frac{1}{v^{\frac{n+d+2\alpha+1}{2}}} \exp\bc{-\frac{1}{2v}\|y-X\beta\|^2_2-\lambda \frac{\|\beta\|_2^2}{v}-\frac{\xi}{v}}.
\end{align}
Despite the complicated form, one can directly sample from $\nu_\dagger$ by noticing that $\nu_\dagger$ is a push-forward measure of the following $\nu_\dagger^\prime$ by the map $T^{-1}: (\varphi,\rho)  \mapsto  (\beta,v)$ such that $\beta=\varphi \sqrt{v}$ and $v=\frac{1}{\rho^2}$: 
\begin{align*}
 \nu_\dagger^\prime(\varphi,\rho|y) \propto \rho^{n+2\alpha-2} \exp\bc{-\frac{1}{2}\|\rho y - X \varphi\|_2^2- \lambda\|\varphi\|_2^2 -\rho^2\xi},
\end{align*}
and that under $\nu_\dagger^\prime$,
\begin{align*}
    \rho^2|y&\sim \Gam\pa{\frac{n+2\alpha-1}{2}, \xi+\frac{1}{2}y^T(\mathbb{I}_n-X(X^TX+2\lambda \mathbb{I}_d)^{-1}X^T)y} \\
    \varphi|\rho,y &\sim \mathcal{N}(\rho(X^TX+2 \lambda \mathbb{I}_d)^{-1}X^Ty, (X^TX+2\lambda \mathbb{I}_d)^{-1}).
\end{align*}
These altogether show a way to obtain samples from $\nu_\dagger(\beta, v|y)$, which we illustrate in Algorithm \ref{a:lasso-feasible}. 

\begin{algorithm}[t]
\caption{A Feasible Start for LassoDA}
\begin{algorithmic}[1]
\INPUT $X \in \mathbb{R}^{n \times d}, y \in \mathbb{R}^n, \lambda \in \mathbb{R}^{+}, \alpha \in \mathbb{R}^{+}, \xi \in \mathbb{R}^{+}$
\State Let $\tilde{y}=y-\bar{y} \mathbf{1}_n$.
\State Draw $\gamma^{(0)} \sim \Gam(\frac{n+2\alpha-1}{2}, \xi+\frac{1}{2}\tilde{y}^T(\mathbb{I}_n-X(X^TX+2\lambda \mathbb{I}_d)^{-1})X^T)\tilde{y})$.
\State Let $\rho^{(0)} = \sqrt{\gamma^{(0)}}$.
\State Draw $\varphi^{(0)} \sim \mathcal{N}(\rho^{(0)}(X^TX+2 \lambda \mathbb{I}_d)^{-1}X^Ty, (X^TX+2\lambda \mathbb{I}_d)^{-1})$.
\State Let $v^{(0)}=\frac{1}{(\rho^{(0)})^2}$.
\State Let $\beta^{(0)}=\varphi^{(0)}\sqrt{v^{(0)}}$.
\OUTPUT{$\beta^{(0)}, v^{(0)}$}
\end{algorithmic}
\label{a:lasso-feasible}
\end{algorithm}

The next lemma measures the distance between $\nu_\dagger(\beta, v|y)$ and the target of LassoDA. One can get an upper bound on mixing time starting from the feasible start~\eqref{e:lasso-feasible} by plugging in the estimate of $\eta$ in Lemma~\ref{l:lasso-feasible} to Theorem 3.3, as we will state in Corollary~\ref{t:feasible-lasso}.
\begin{lemma} \label{l:lasso-feasible} Suppose $n \ge 2-2\alpha$. We assume that $\norm{X}_{\mathrm{op}}=\Poly(nd)$ and $\norm{y}_2=\Poly(n)$. With a proper variance prior (i.e. $\xi > 0$), we have that $$\sup_A \frac{\nu_\dagger(A)}{\pi^{\Lasso}(A)} \le e^{c(d\log d + n\log n)},$$ where the supremum is taken over all the measurable sets $A \subseteq \mathbb{R}^d$, and $c$ is a constant depending on $M$ and $\xi$.
\end{lemma}
The proof of Lemma~\ref{l:lasso-feasible} is deferred to Section~\ref{a:lasso-feasible-proof}. 
\begin{cor}\label{t:feasible-lasso}
Suppose $n \ge 2-2\alpha$. Assuming that $\norm{X}_{\mathrm{op}}=\Poly(nd)$ and $\norm{y}_2=\Poly(n)$, we have for any error tolerance $\epsilon \in (0,1)$, the mixing time of LassoDA starting from $\nu_\dagger$ satisfies
$$
t_{\Psi^{\Lasso}}(\eta, \epsilon) \le c \pa{d^2 (d\log d+n \log n)^2 \pa{ d\log d+n \log n + \log \pa{\frac{1}{\epsilon}}}},
$$
where $c$ is a constant depending on $\xi$ and $M$.
\end{cor}

\section{Comparison to best known guarantees of alternatives} \label{s:comparison}

Apart from the DA algorithms, one can alternatively sample from the target distributions of the three DA algorithms using generic sampling algorithms, such as Metropolis-Hastings and gradient-based algorithms. It is a common problem in practice to decide which algorithm to choose. Certainly, without user-tuned parameters, the DA algorithms are the easiest to implement, as the Metropolis-Hastings and gradient-based algorithms usually require user-set proposal distribution or step size. Aside from the apparent advantage of convenience, it is important to compare the DA algorithms and the alternatives in terms of computational complexity. Furthermore, if the DA algorithms are slower, it is useful to specify how much the trade-off is for implementation convenience. One way that theoretical complexity analysis benefits empirical studies is by making quantitative and potentially comprehensive comparisons between alternative algorithms. 
We carry this out for the mixing time of the DA algorithm. 

We choose Langevin Monte Carlo (LMC, see Algorithm~\ref{a:lmc}) and Metropolis Adjusted Langevin Algorithm (MALA, see Algorithm~\ref{a:mala}) as representative examples of alternative sampling algorithms. The choice is based on a general classification of sampling algorithms as low-accuracy samplers or high-accuracy samplers. \textit{Low-accuracy samplers} refer to sampling algorithms obtained by discretization of stochastic processes, where the discretization introduces bias for the stationary distribution. Examples of low-accuracy samplers include Langevin Monte Carlo and Hamiltonian Monte Carlo. On the other hand, \textit{high-accuracy samplers} refer to sampling algorithms that have an unbiased stationary distribution, such as Gibbs samplers and Metropolis-Hasting algorithms. The DA algorithms are high-accuracy samplers. Considering the simplicity of theoretical results, we employ LMC as an example of an alternative low-accuracy sampler and MALA as an example of an alternative high-accuracy sampler. 

The comparison will be done on both mixing time and cost per iteration, presented in Section~\ref{s:mixing-time-compare} and Section~\ref{s:cost-per-iteration-compare}, respectively. 

\subsection{Mixing Time} \label{s:mixing-time-compare}
We begin by noting that a complete comparison of mixing times is not yet possible. Part of the challenge comes from the fact that a conclusive comparison relies on lower bound analysis, which is underdeveloped for DA algorithms and alternative algorithms. Specifically, to demonstrate that Algorithm A is faster than Algorithm B, one needs to show that an upper bound of Algorithm A is smaller than a lower bound of Algorithm B. As a compromise, we make the \emph{comparison based on upper bounds}: the upper bound of DA algorithms from this work and the best known upper bounds of the alternative algorithms in the literature. We remark on the possibility that the upper bounds could not be tight, failing to reflect the actual complexity, and thus making the comparison invalid. 

In addition, we remind the readers of the potential risk of understating the efficiency of the generic algorithm, if one directly applies the generic guarantees to specific algorithms. While the DA algorithms work for specific targets, most guarantees for alternative algorithms are proposed for a general class of distributions. They can be possibly improved for the three specific distributions.  Furthermore, without access to their exact values, we can only substitute the best attainable upper bounds of the important quantities, such as condition numbers and isoperimetric constants, into the guarantees of alternative generic sampling algorithms. This could worsen the guarantees. As a result of these limitations, we only take our comparison as a heuristic discussion, without drawing an affirmed conclusion of the superiority of any algorithm. 

We will focus on ProbitDA and LogitDA, as the target of LassoDA is not regular enough to fit in the settings of most existing analyses. Standard assumptions of the analysis on the generic sampling algorithm include a strong log-concavity constant $m>0$ and a gradient Lipschitz constant $L$ (i.e., the $L$-smoothness condition). It is not hard to generalize the strong log-concavity to isoperimetry, which is satisfied for the transformed LassoDA's target (Lemma 4.9). However, the transformed LassoDA target does not have a uniform gradient Lipschitz constant, making it difficult to apply the existing guarantees.

\paragraph*{ProbitDA/LogitDA v.s. LMC} Langevin Monte Carlo (LMC) is a canonical sampling algorithm, which iterates according to the discretization of the Langevin diffusion. Despite the long history, it was only analyzed in non-asymptotic settings recently (e.g. \cite{cheng2018convergence, durmus2019analysis, vempala2019rapid, durmus2017nonasymptotic, dalalyan2012sparse, dalalyan2017further, dalalyan2017theoretical, durmus2019high}). Among the works in the standard $m$-strongly log-concave and $L$-smooth setting, \cite{durmus2019analysis} obtains the mixing time guarantee $\tilde{\mathcal{O}}(\kappa d/\epsilon)$ in KL divergence for LMC with the Euler–Maruyama discretization, where the dependencies on both $d$ and $\kappa$ are currently the best. This can be translated into $\tilde{\mathcal{O}}(\kappa d/\epsilon^2)$ in TV distance using Pinsker's inequality. Using the results in the Equations \eqref{e:probit-convex}, \eqref{e:logit-convex}, \eqref{e:probit-smooth}, and \eqref{e:logit-smooth}, considering the same setting and procedure of Corollary \ref{t:main-cor} to specify dependency of $\|X\|_{\mathrm{op}}$, and assuming that $\lambda_{\min}(B) = \Omega(1)$ and $\lambda_{\max}(B) = \mathcal{O}(1)$, we can obtain that $\kappa^{\Probit} \le \frac{L^{\prime \Probit}}{m^{\prime \Probit}} =\mathcal{O}(n)$ and $ \kappa^{\Logit} \le \frac{L^{\prime \Logit}}{m^{\prime \Logit}}=\mathcal{O}(n)$. This results in a $\tilde{\mathcal{O}}(n d/\epsilon^2)$ mixing time guarantee for LMC on the targets of ProbitDA and LogitDA. We first note that the LMC result has a polynomial dependence on the error parameter $\epsilon$ while our results for ProbitDA and LogitDA have a superior logarithmic dependence
on $\epsilon$. Furthermore, the guarantee for LMC has an extra $d$ dependence compared to our results for ProbitDA and LogitDA.

Some more sophisticated designs could potentially make LMC faster. Motivated by the acceleration phenomenon in optimization, the Underdamped LMC (ULMC) is an important variant of LMC in which the momentum is refreshed continuously. The current best mixing time guarantees for ULMC is $\tilde{\mathcal{O}}\pa{\frac{\kappa^{\frac{3}{2}}\sqrt{d}}{\sqrt{\epsilon}}}$ in KL divergence \cite{ma2021there, zhang2023improved}, equivalently  $\tilde{\mathcal{O}}\pa{\frac{\kappa^{\frac{3}{2}}\sqrt{d}}{\epsilon}}$ in TV distance. Using the same method as in LMC, the bound becomes $\tilde{\mathcal{O}}\pa{\frac{n^{\frac{3}{2}}\sqrt{d}}{\epsilon}}$ for the targets of ProbitDA and LogitDA, which is worse than our guarantees for ProbitDA and LogitDA. Equipping ULMC with the randomized midpoint discretization, \cite{shen2019randomized} obtains a mixing time guarantee $\tilde{\mathcal{O}}\pa{\frac{\kappa d^{\frac{1}{3}}}{\epsilon^{2/3}}+\frac{\kappa^{\frac{7}{6}}d^{\frac{1}{6}}}{\epsilon^{1/3}}}$ in 2-Wasserstein distance, which translates into $\tilde{\mathcal{O}}\pa{\frac{nd^{\frac{1}{3}}}{\epsilon^{2/3}}+\frac{n^{\frac{7}{6}}d^{\frac{1}{6}}}{\epsilon^{1/3}}}$ for ProbitDA and LogitDA. We further note that,  using the inequality $W_2^2(\mu_1, \mu_2) \le \frac{2}{m} \chi^2(\mu_1||\mu_2)$ under the assumption that $\mu_2$ is $m$-strongly log-concave (see \cite[Corollary 9.3.2]{bakry2014analysis}), our guarantees in Corollary \ref{t:main-cor}, namely $\mathcal{O}\pa{n \log\pa{\frac{\log \eta}{\epsilon}}}$, also extend to the $2$-Wasserstein distance. Therefore, compared to our results, the bound in \cite{shen2019randomized} exhibits superlinear dependence on $n$ and an additional dependence on $d$. 

\paragraph*{ProbitDA/LogitDA v.s. MALA} Metropolis Adjusted Langevin Algorithm (MALA) is a fundamental high-accuracy sampler. MALA runs an additional Metropolis accept-reject step in each iteration of LMC, which adjusts the bias in stationary distribution. Among the recent line of works analyzing the mixing time of MALA \cite{altschuler2024faster, wu2022minimax, chen2023does, chen2020fast, dwivedi2019log, lee2020logsmooth, chewi2021optimal}, \cite{wu2022minimax, altschuler2024faster} obtain the state-of-the-art $\mathcal{O}(\kappa d^{1/2})$ complexity bound in TV distance for MALA in $m$-strongly log-concave and $L$-smooth setting. Following the same argument as in our discussion of LMC, the bound can be translated into $\mathcal{O}(n d^{1/2})$ for the targets of ProbitDA and LogitDA. We note that the MALA guarantee has an extra $d^{1/2}$ dependence compared to our results for the two DA algorithms. 
\\

Despite the obstacles, the comparison provides insight into the superiority of the mixing time of ProbitDA and LogitDA over some generic sampling algorithms. We leave a more thorough and more conclusive comparison for future research.

\begin{algorithm}[t]
\caption{LMC}
\begin{algorithmic}[1]
\INPUT The target distribution $\pi(x) \; \propto \; e^{-f(x)}$, on $\mathbb{R}^d$, the step size $h > 0$
\State Draw $x^{(0)}$ from an initial distribution.
\For{$m = 1,2,\ldots$}
  \State Draw $\xi^{(m)} \sim \mathcal{N}(0, \mathbb{I}_d)$
  \State Compute $x^{(m)}=x^{(m)} - h \nabla f \pa{x^{(m-1)}} + \sqrt{2h} \xi^{(m)}$
\EndFor
\end{algorithmic}
\label{a:lmc}
\end{algorithm}

\begin{algorithm}[t]
\caption{MALA}
\begin{algorithmic}[1]
\INPUT The target distribution $\pi(x) \; \propto \; e^{-f(x)}$, on $\mathbb{R}^d$, the step size $h > 0$
\State Draw $x^{(0)}$ from an initial distribution.
\For{$m = 1,2,\ldots$}
  \State Propose $y^{(m)}$ from the Langevin step: $y^{(m)}=x^{(m)} - h \nabla f \pa{x^{(m)}} + \sqrt{2h} \xi^{(m)}$ 
  \State Let $\mathcal{N}(x; \mu, \Sigma)$ be the pdf of $\mathcal{N}(\mu, \Sigma)$ evaluated at x. Compute the acceptance probability $a=\min\bc{1, \frac{\pi(y^{(m)})p(y^{(m)}, x^{(m+1)})}{\pi(x^{(m-1)})p(x^{(m-1)},y^{(m)})}}$, 
  where $p(x,y)=\mathcal{N}(y; x - h \nabla f (x), 2h\mathbb{I}_d)$. 
  \State Draw $u\sim \textup{Unif}[0,1]$. If $u \le a$, $x^{(m)} = y^{(m)}$. Otherwise, $x^{(m)} = x^{(m-1)}$
\EndFor
\end{algorithmic}
\label{a:mala}
\end{algorithm}

\subsection{Cost per iteration} \label{s:cost-per-iteration-compare} This subsection presents a comparison of the computational complexities per iteration among the three DA algorithms and the LMC/MALA methods. We set aside the cost of computing the inverse $B$ for ProbitDA and LogitDA, since it is shared across all algorithms. In line with practical implementations, we adopt the naive method for matrix multiplication, which yields a complexity of $\mathcal{O}(ndk)$ for multiplying an $\mathbb{R}^{n \times d}$ matrix with an $\mathbb{R}^{d \times k}$ matrix, and $\mathcal{O}(d^3)$ for inverting an $\mathbb{R}^{d \times d}$ matrix, although better theoretical bounds \cite{alman2021limits, le2014powers} are available using more advanced algorithms.

\paragraph*{LMC/MALA} The dominant cost in each iteration of LMC and MALA is computing the log-gradient of the posterior. Other significant computations include evaluating the density and sampling from a $d$-dimensional Gaussian with diagonal covariance for MALA, which both have cost $\mathcal{O}(d)$. As we will see shortly, these costs do not exceed that of computing the log-gradient.

We start with the log-gradients of the posterior for ProbitDA in Equation~\eqref{e:probit-log-gradient} and LogitDA in Equation~\eqref{e:logit-log-gradient}. For ProbitDA, multiplying $B^{-1}$ by $\beta-b$ costs $\mathcal{O}(d^2)$, while computing $-\sum_{i=1}^{n} y_i x_i\frac{\phi(x_i^T\beta) }{\Phi(x_i^T\beta)} + \sum_{i=1}^{n}(1-y_i) x_i\frac{\phi(x_i^T\beta)}{1-\Phi(x_i^T\beta)}$ costs $\mathcal{O}(nd)$. Overall, the per-iteration cost is $\mathcal{O}(d\max\{n,d\})$. Similarly, computing the log-gradient of LogitDA has the same $\mathcal{O}(d\max\{n,d\})$ cost. 

For LassoDA, there are two approaches to generate samples using generic sampling algorithms: one can either sample directly from the original target or sample from the transformed target and then transform the samples back. The log-gradient of the original LassoDA target in Equation \eqref{e:lasso-pos} is 
\begin{align*}
  \frac{\partial f^{\Lasso}}{\partial \beta} &\propto \frac{1}{2v}(X^TX\beta-2X\tilde{y}) + \lambda \frac{\sign(\beta)}{\sqrt{v}} \\
\frac{\partial f^{\Lasso}}{\partial v} &\propto \frac{n+d+2\alpha+1}{2v} - \frac{1}{2v^2}\|\tilde{y}-X\beta\|_2^2 -\frac{\xi}{v^2} - \frac{1}{2\sqrt{v^3}}  .
\end{align*} The dominant computations are $X^TX\beta$ and $X\tilde{y}$, each with complexity $\mathcal{O}(nd)$. For the transformed target in Equation \eqref{e:transformed-target}, the log-gradient is
\begin{align*}
    \frac{\partial f_T^{\Lasso}}{\partial \varphi} &\propto \frac{1}{2}(X^TX\varphi - 2\rho y^TX) +\lambda \sign(\varphi) \\ \frac{\partial f_T^{\Lasso}}{\partial \rho} &\propto   -\frac{n+2\alpha-2}{\rho} - y^TX\varphi + \rho y^Ty + 2\xi \rho.
\end{align*}
Here, the dominant computations are $X^TX\varphi$ and $y^TX$, both with complexity $\mathcal{O}(nd)$. In either case, the per-iteration cost for LassoDA is $\mathcal{O}(nd)$. 

\paragraph*{ProbitDA}
Several expensive computations only need to be performed once for ProbitDA, and we evaluate this pre-computation cost separately. Specifically, $(B^{-1}+X^TX)^{-1}$ only needs to be computed once and reused in every iteration, which costs $\mathcal{O}(nd^2)$ for forming $X^TX$ and $\mathcal{O}(d^3)$ for direct inversion via Cholesky factorization (i.e., $(B^{-1}+X^TX)^{-1}=LL^T$, where $L$ is lower-triangular). Overall, the pre-computation cost for ProbitDA is $\mathcal{O}(d^2\max\{n,d\})$.

Each iteration can then be carried out in $\mathcal{O}(d\max\{n,d\})$ if we reuse both the inverse $(B^{-1}+X^TX)^{-1}$ and its Cholesky factorization $L$. In particular, the multiplication $(B^{-1}+X^TX)^{-1}X^Tz$ requires only $\mathcal{O}(nd)$, which is the same order as sampling $n$ truncated normals. In addition, $\beta$ can be sampled in $\mathcal{O}(d^2)$ using
$$\beta=(B^{-1}+X^TX)^{-1}(X^Tz+B^{-1}b) + L \xi, \text{ where } \xi \sim \mathcal{N}(0, \mathbb{I}_d).$$
Other computations are negligible compared to these.

\paragraph*{LogitDA} The most expensive computation for LogitDA is $(B^{-1}+X^T\Omega X)^{-1}$, which must be recomputed at every iteration. Since $\Omega$ is diagonal, forming $X^T\Omega X$ costs $\mathcal{O}(nd)$, and the inversion costs $\mathcal{O}(d^3)$ via direct Cholesky factorization, which is required for sampling. Other operations are negligible in comparison. Overall, the per-iteration cost is $\mathcal{O}(nd+d^3)$.

\paragraph*{LassoDA} Similarly, the dominant cost for LassoDA is computing $(X^TX+D_z^{-1})^{-1}$. We can precompute $X^TX$ in $\mathcal{O}(nd^2)$, while in each iteration the inversion of $(X^TX+D_z^{-1})$ requires $\mathcal{O}(d^3)$.

In conclusion, ProbitDA has the same per-iteration cost as LMC and MALA once precomputation is performed, whereas the cost of computing each iteration in LogitDA and LassoDA is higher than that of LMC and MALA.

\section{Numerical experiments}\label{s:experiments}

In this section, we study the dependencies of the mixing time of three DA algorithms on $n$ and $d$ through computer simulations. Specifically, we investigate the following three scenarios:

\begin{enumerate}[leftmargin=3cm, label={\text{Scenario} \arabic*}]
    \item(Both $n$ and $d$ grow): $n=d=50,100,150,\ldots,1000$.
    \item($d$ fixed, $n$ grows): $d=500$, $n=50,100,150,\ldots,1000$. 
    \item($n$ fixed, $d$ grows): $n=500$, $d=50,100,150,\ldots,1000$. 
\end{enumerate}

We will introduce the notion of relaxation time, a proxy for mixing time in Section~\ref{s:sim-tauto}. We then present the simulation settings and results for ProbitDA and LogitDA in Section~\ref{s:sim-probit}, and LassoDA in Section~\ref{s:sim-lasso}. 

\subsection{Relaxation Time}\label{s:sim-tauto}
Due to the difficulty in calculating TV distance, a good estimator for mixing time is not easily obtainable. We instead study a closely related quantity, $L^2$ relaxation time.
 
To give formal definitions, we consider samples from a Markov chain with transition kernel $\mathcal{P}$ starting from the stationary distribution $\pi$: $\theta_0, \theta_1, \theta_2,\ldots$ with $\theta_0 \sim \pi$. We restrict ourselves to reversible chains with non-negative spectrum, which include the DA chains \cite[Lemma 3.2]{liu1994covariance}. Let $L^2(\pi)$ be the space of square integrable functions under the function $\pi$ with inner product $\ip{f,g}_\pi = \int f g d\pi.$ Then, the \textit{relaxation time} can be defined as the inverse of the spectral gap,
 \[
 t_{\textup{rel}} = \frac{1}{1-\lambda}, 
 \]
where 
$\lambda = \sup_{f \in L^2_0(\pi)} \frac{\ip{f, \mathcal{P}f}_{\pi}}{\ip{f, f}_{\pi}}$ and $L^2_0(\pi) = \{f \in L^2(\pi): \int f d\pi = 0 \}$. We assume $\lambda<1$.

Suppose $\mathcal{G}$ is the inverse operator of the generator $\mathbb{I}-\mathcal{P}$. One can show that $\mathcal{G}$ satisfies $\mathcal{G}f = \sum_{m=0}^\infty \mathcal{P}^m f$ for $f\in L^2(\pi)$. Then, we have 
\begin{align}\label{e:trel_tauto}
    t_{\textup{rel}}&=\sup_{f \in L_0^2(\pi)} \frac{\ip{f, Gf}_{\pi}}{\ip{f, f}_{\pi}}= \sup_{f \in L_0^2(\pi)} \frac{\sum_{m=0}^\infty \textup{Cov}_\pi(f(\theta_0), f(\theta_m))}{\textup{Var}_{\pi}(f(\theta_0))} \\
    &= \sup_{f \in L^2_0(\pi)} \sum_{m=0}^\infty \textup{Corr}_\pi(f(\theta_0), f(\theta_m)):= \sup_{f \in L_0^2(\pi)} t_{\textup{rel},f}. \nonumber
\end{align} Here, we define $t_{\textup{rel},f}:=\sum_{m=0}^\infty \textup{Corr}_\pi(f(\theta_0), f(\theta_m))$. 

We can estimate $t_{\textup{rel},f}$ by summing up Pearson correlations calculated using samples after a certain burn-in period. Specifically, with maximum iteration $N$, burn-in period $n_0$, and maximum lag $M$, we have $\widehat{t_{\textup{rel},f}} = \sum_{m=0}^{\min\{m_0,M\}} \gamma_f(m), $
where $m_0 =\max\{m: \gamma_f(m) > 0
\}$, and $\gamma_f(m)$ is the Pearson correlation between $\{f(\theta_i)\}_{i=n_0}^{N-m}$ and $\{f(\theta_i)\}_{i=n_0+m}^N$.
 That is, we only sum sample correlations up to the point when the correlation first crosses the zero-axis or the lag reaches the maximum lag. We take $N=1000, n_0=200$, and $M=100$ in our simulations. 
 
It is impossible to calculate $t_{\textup{rel},f}$ with respect to every possible test function in the space $L^2_0(\pi)$. Therefore, in this simulation, we restrict our analysis to the projection maps onto each coordinate. We then calculate $t_{\textup{rel},f}$ for each individual coordinate and use the maximum of these results as a proxy for the overall relaxation time. Specifically, 
 \begin{align*}
     t^{\textup{Proj}}_{\textup{rel}}(\mathcal{D}) &= \max_{1 \le k \le d}\widehat{t_{\textup{rel},\operatorname{Proj}_k}} 
 \end{align*}
 where $\operatorname{Proj}_k(\theta)$ with $k =1,...,d$ is the projection map to the $k^{th}$ coordinate of $\theta$, and $\mathcal{D} = [X, y]$ is the dataset the simulation is run on.  

Because that we only use a subset of test functions in $L^2_0(\pi)$ and that relaxation time is usually smaller than mixing time \cite[Theorem 12.5]{levin2017markov},  $t^{\textup{Proj}}_{\textup{rel}}(\mathcal{D})$ serves as a lower bound for mixing time. If in the simulation results, this quantity scales as our guarantees for mixing time, we obtain empirical evidence supporting the tightness of our bounds.

Furthermore, to account for the randomness in data generation, we generate 100 datasets and take the average of the resulting estimates. That is 
 \begin{align*}
\overline{t^{\textup{Proj}}_{\textup{rel}}} &= \sum_{i=1}^{100} t^{\textup{Proj}}_{\textup{rel}}(\mathcal{D}_i)
 \end{align*}

\subsection{Results for ProbitDA and LogitDA} \label{s:sim-probit}
We consider the following prior information and data-generating process: 
\begin{align*}
    b &= \mathbf{0}, \quad B = \mathbb{I}_d,  \\
    \beta_0 &\sim \mathcal{N}([1 \quad \mathbf{0}]^T, \mathrm{diag}([0 \quad \mathbf{1}_{d-1}])),\quad  a \in \mathbb{R},\\
    x_i &\mathop{\sim}\limits^{\mathrm{i.i.d.}} [1 \quad \mathcal{N}(0, \mathbb{I}_{d-1})/\sqrt{d}]^T, \quad y_i \sim \Ber(\Phi(x_i^T\beta_0)), \quad i=1,..., n.
\end{align*}
We note that we only generate one fixed $\beta_0$ for each dataset $\mathcal{D} = [X,y]$, but vary the value of $\beta_0$ across datasets. 

\paragraph*{The worst case scenario} For the response vector $y$, we first consider 
$$
y_i = 1, \quad i=1,...,n,
$$
which has been identified as the hardest case in \cite{johndrow2019mcmc,ascolani2025mixing}.

We present the plots of $\overline{t^{\textup{Proj}}_{\textup{rel}}}$ for the three scenarios in Figure~\ref{fig:probit_imbalanced} (ProbitDA) and Figure~\ref{fig:logit_imbalanced} (LogitDA). We also fit a linear regression to the points of $\overline{t^{\textup{Proj}}_{\textup{rel}}}$, plot the resulting line, and report both the slope and the corresponding p-value testing the null hypothesis that the slope is zero.

We first discuss Figure~\ref{fig:probit_imbalanced} for ProbitDA. In Scenarios 1 and 2, we observe upward trends with statistically significant positive slopes at the 0.01 level, matching the linear dependence on $n$ predicted by Corollary \ref{t:main-cor}. Meanwhile, Scenario 3 shows no positive slope. This suggests that the constant bound in $d$ from Corollary \ref{t:main-cor} is also likely tight for ProbitDA.

\begin{figure}[t]
    \centering
    \includegraphics[width=0.33\linewidth]{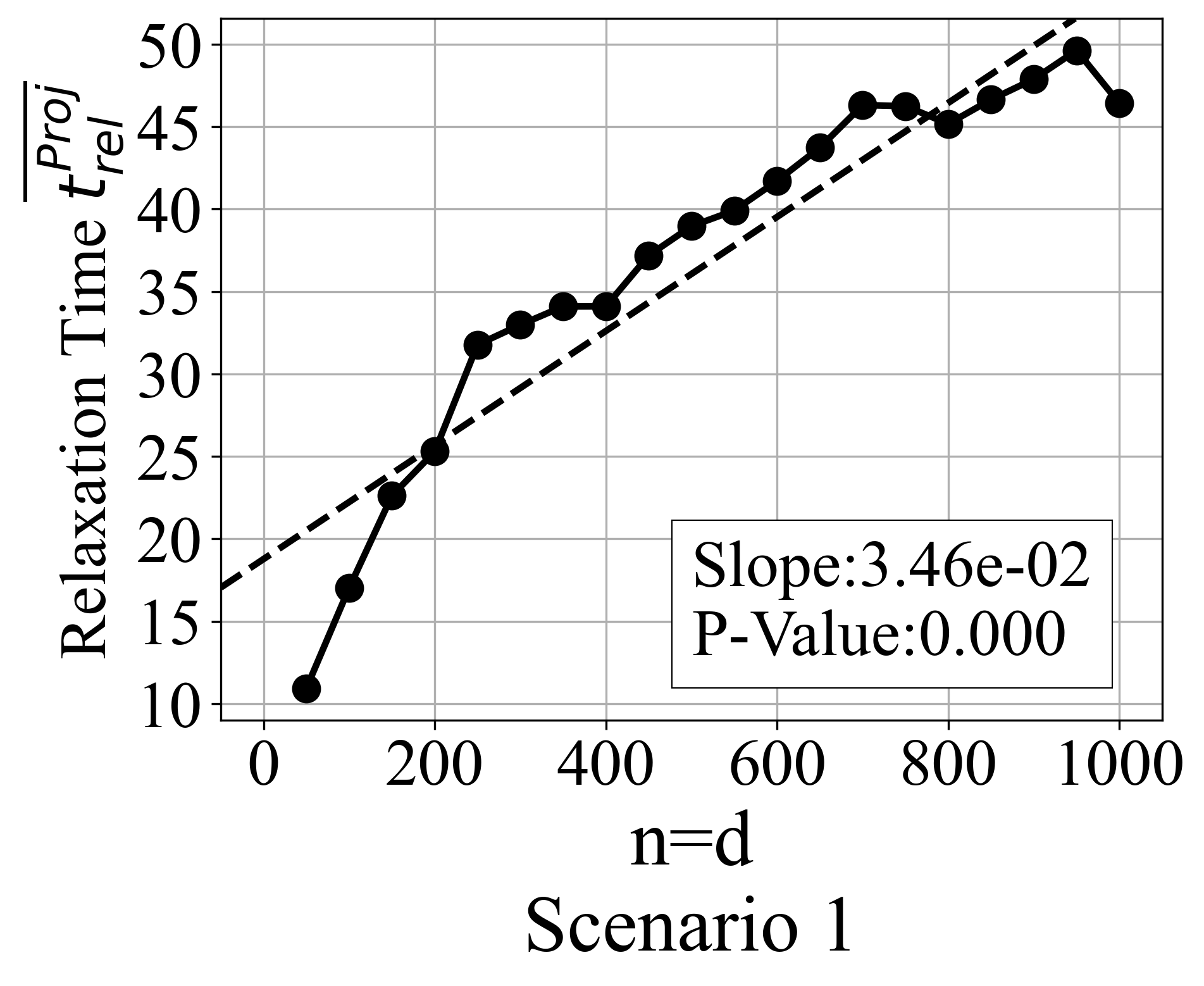}
    \includegraphics[width=0.32\linewidth]{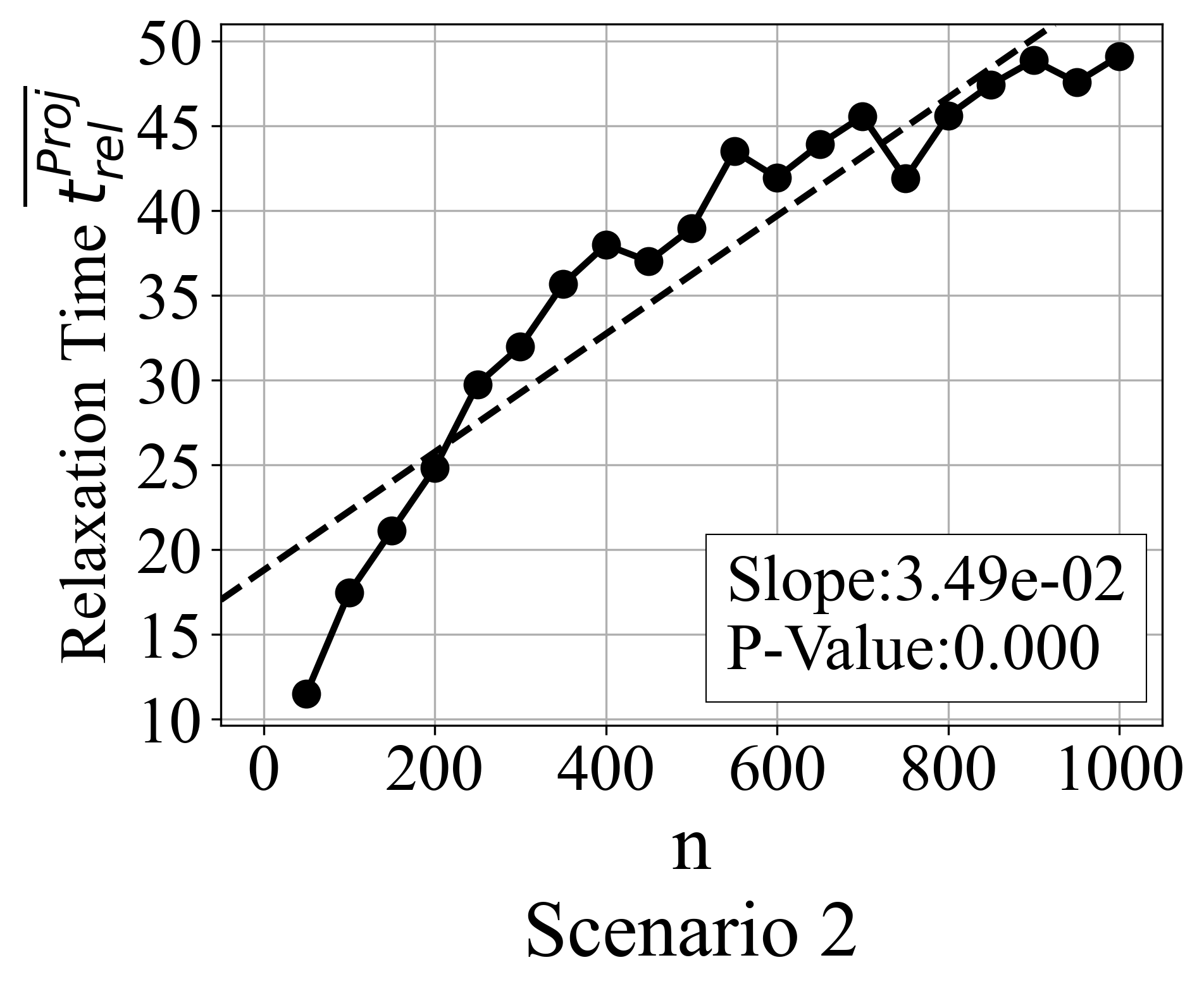}
    \includegraphics[width=0.32\linewidth]{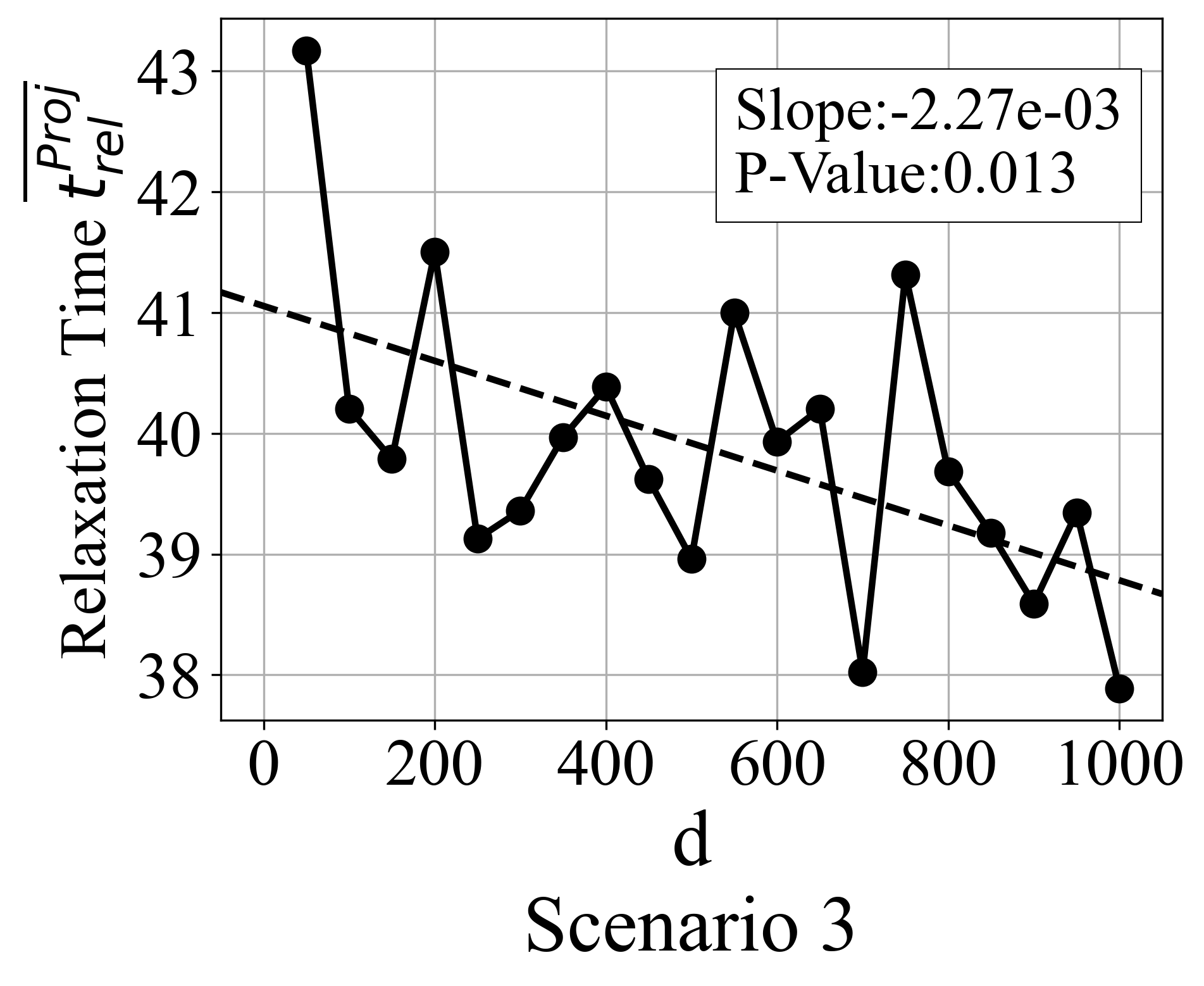}
    \caption{Simulation results for ProbitDA with $y=1$.}
    \label{fig:probit_imbalanced}
\end{figure}

\begin{figure}[t]
    \centering
    \includegraphics[width=0.33\linewidth]{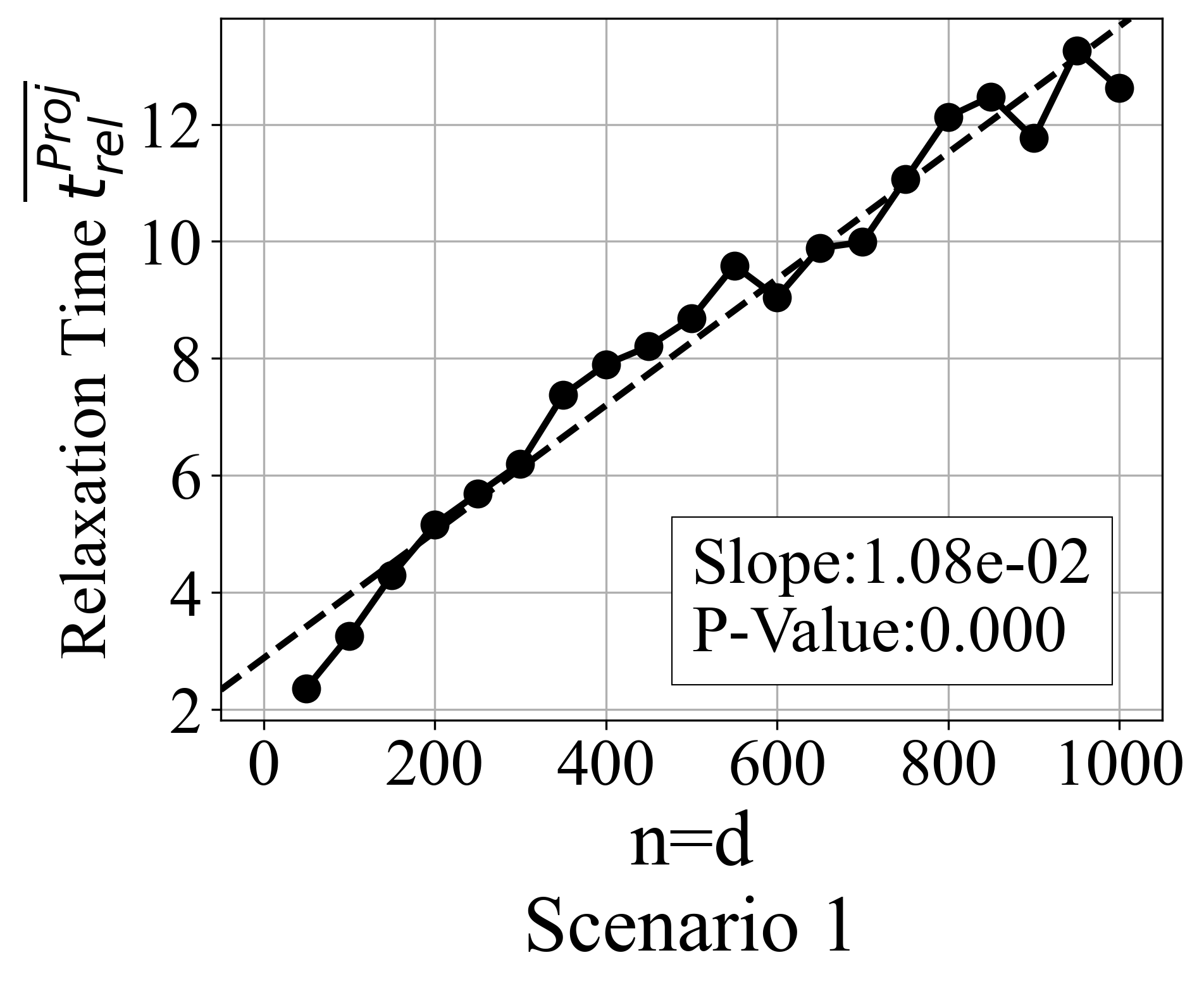}
    \includegraphics[width=0.32\linewidth]{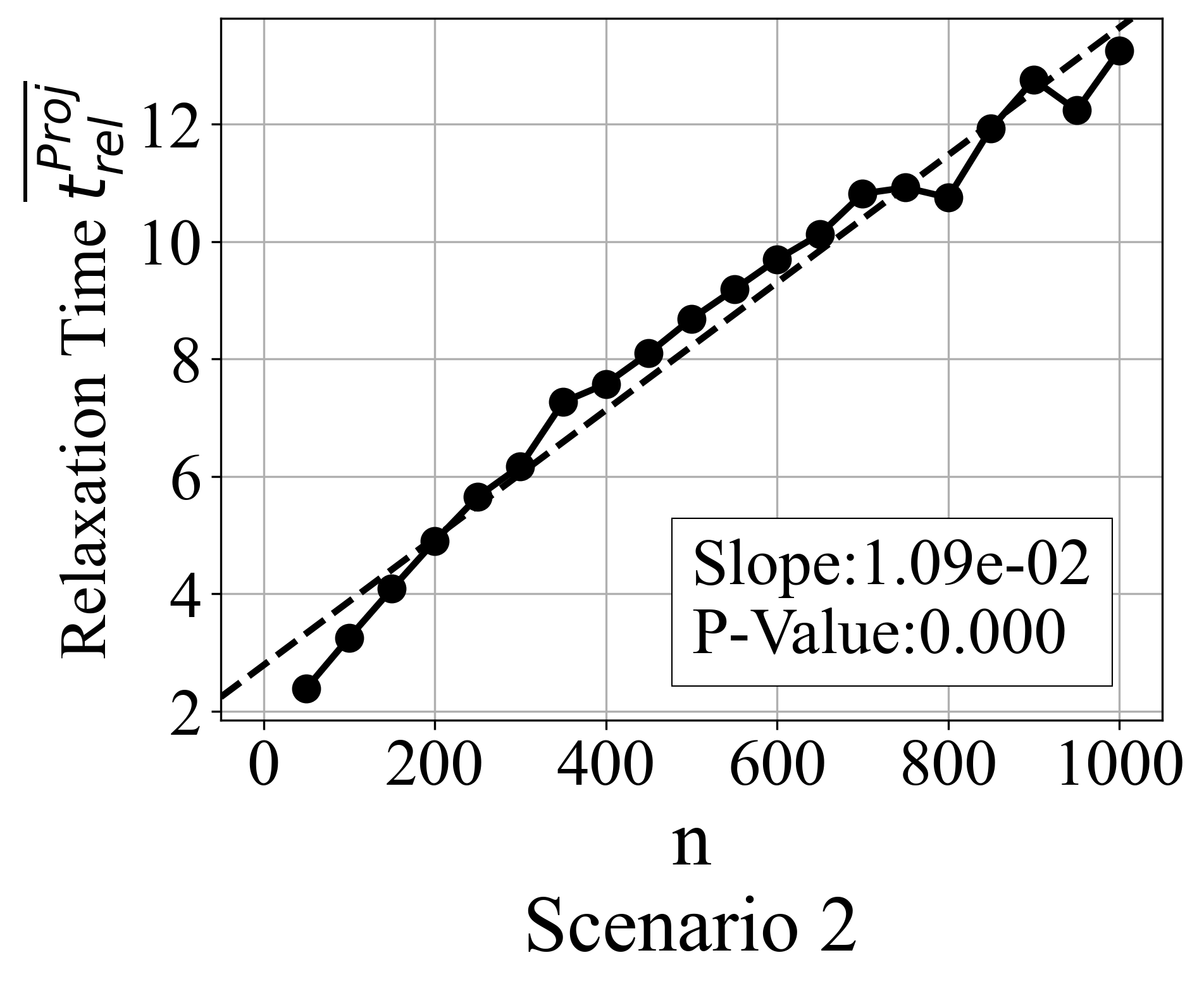}
    \includegraphics[width=0.32\linewidth]{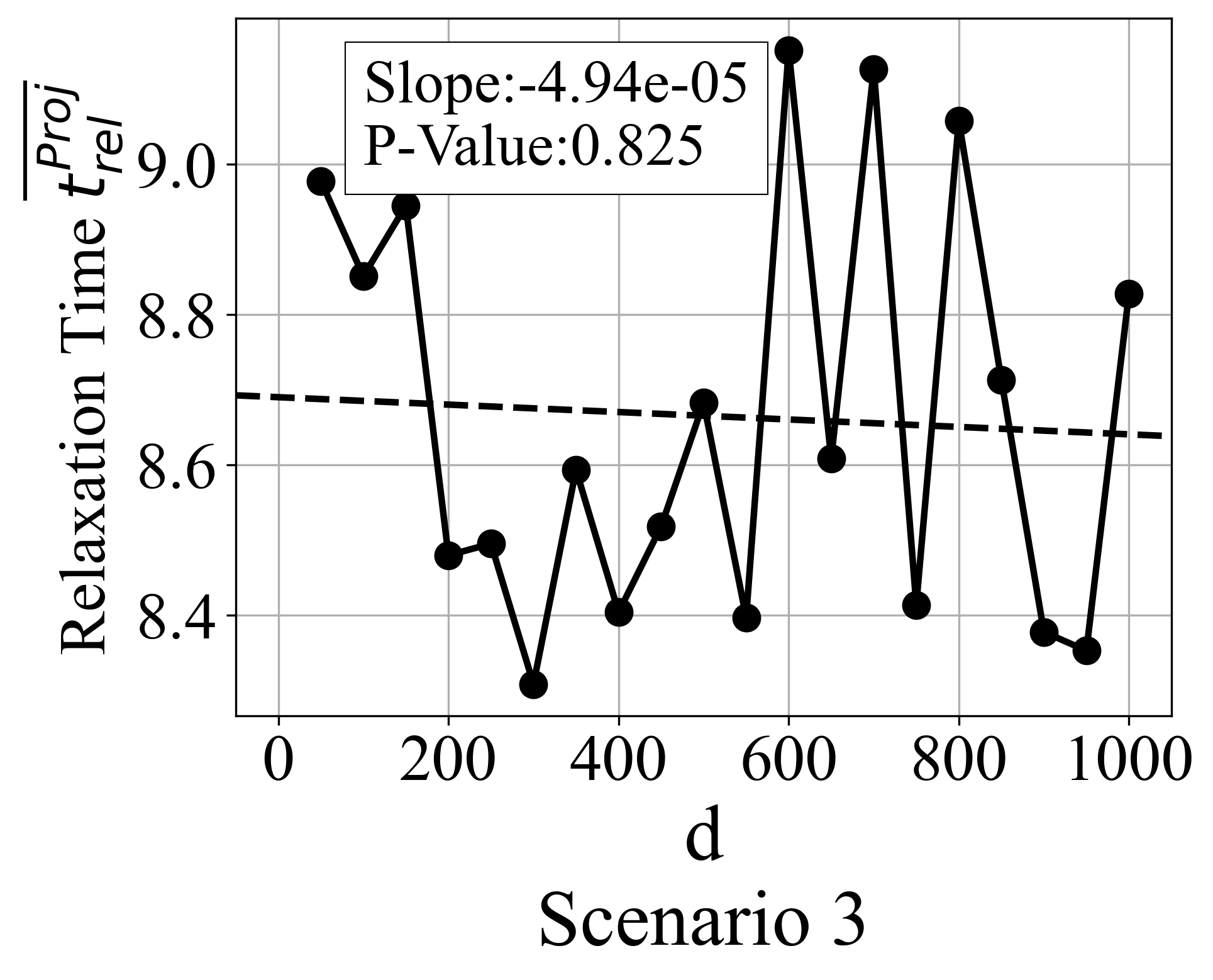}
    \caption{Simulation results for LogitDA with $y=1$. }
    \label{fig:logit_imbalanced}
\end{figure} 

In Figure~\ref{fig:logit_imbalanced}, we observe similar patterns for LogitDA: clear linear growth in Scenarios 1 and 2, and no positive slope in Scenario 3. These observations lead to the same conclusions as for ProbitDA.

\paragraph*{The average case} We also report the case where the response data is generated by the model: For ProbitDA, 
$$
y_i \sim \Ber(\Phi(x_i^T\beta)), \quad i =1,..., n
$$
and for LogitDA,  
$$
y_i \sim \Ber\left(\frac{1}{1+e^{-x_i^T\beta_0}}\right), \quad i=1, \ldots, n. 
$$
We present the plots of $\overline{t^{\textup{Proj}}_{\textup{rel}}}$ for the three scenarios in Figure~\ref{fig:probit} and Figure~\ref{fig:logit} for ProbitDA and LogitDA, respectively.

\begin{figure}[t]
    \centering
    \includegraphics[width=0.33\linewidth]{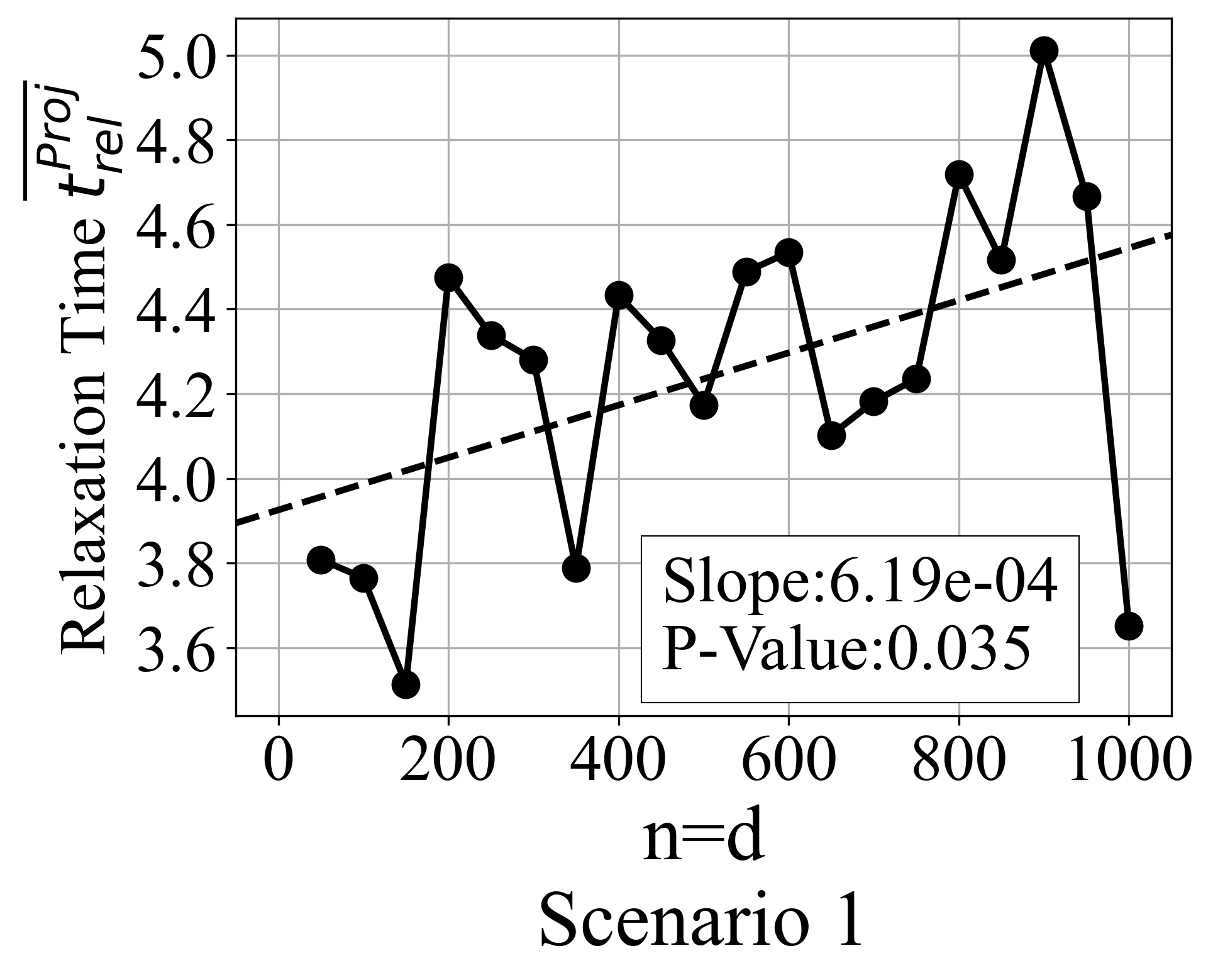}
    \includegraphics[width=0.32\linewidth]{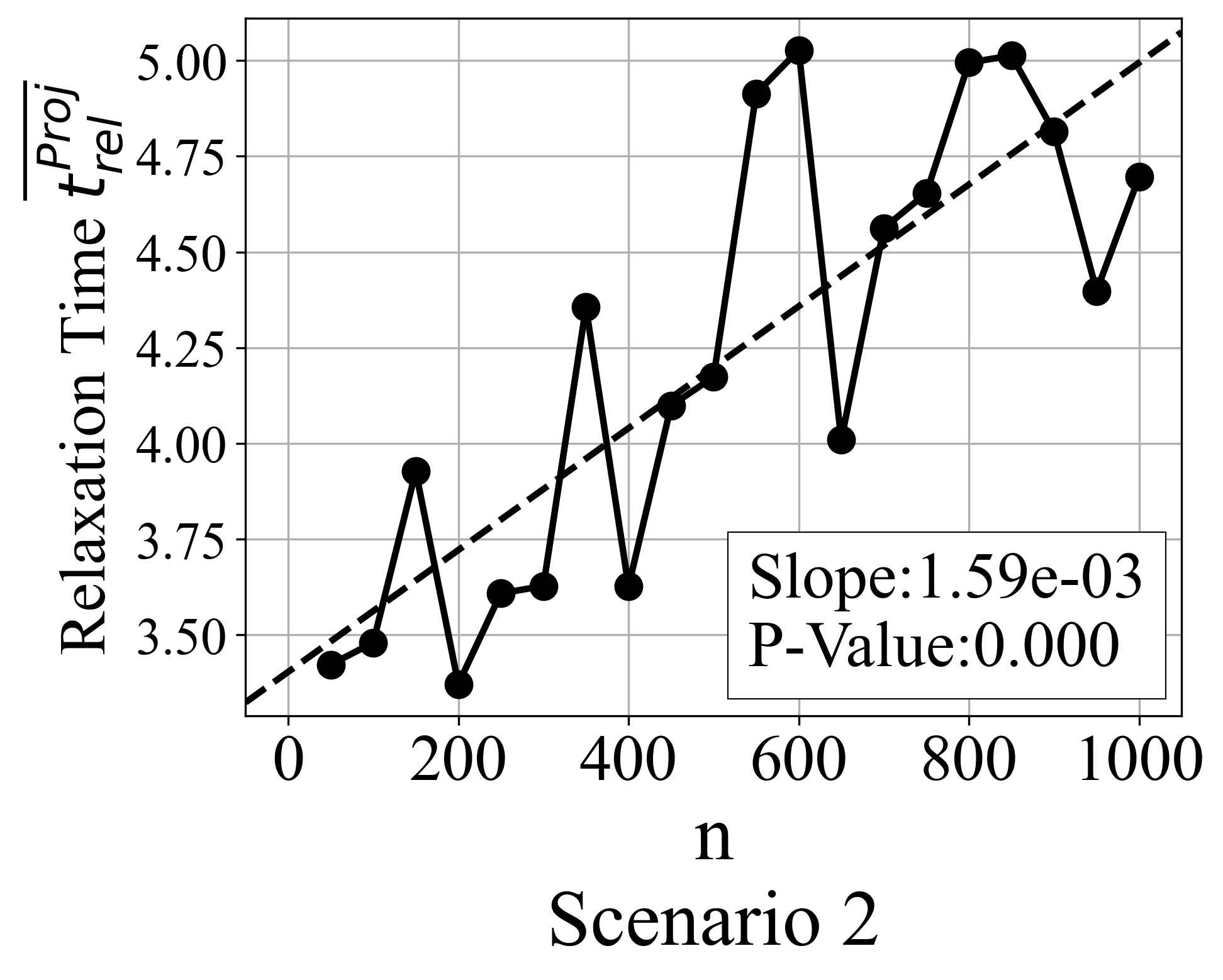}
    \includegraphics[width=0.32\linewidth]{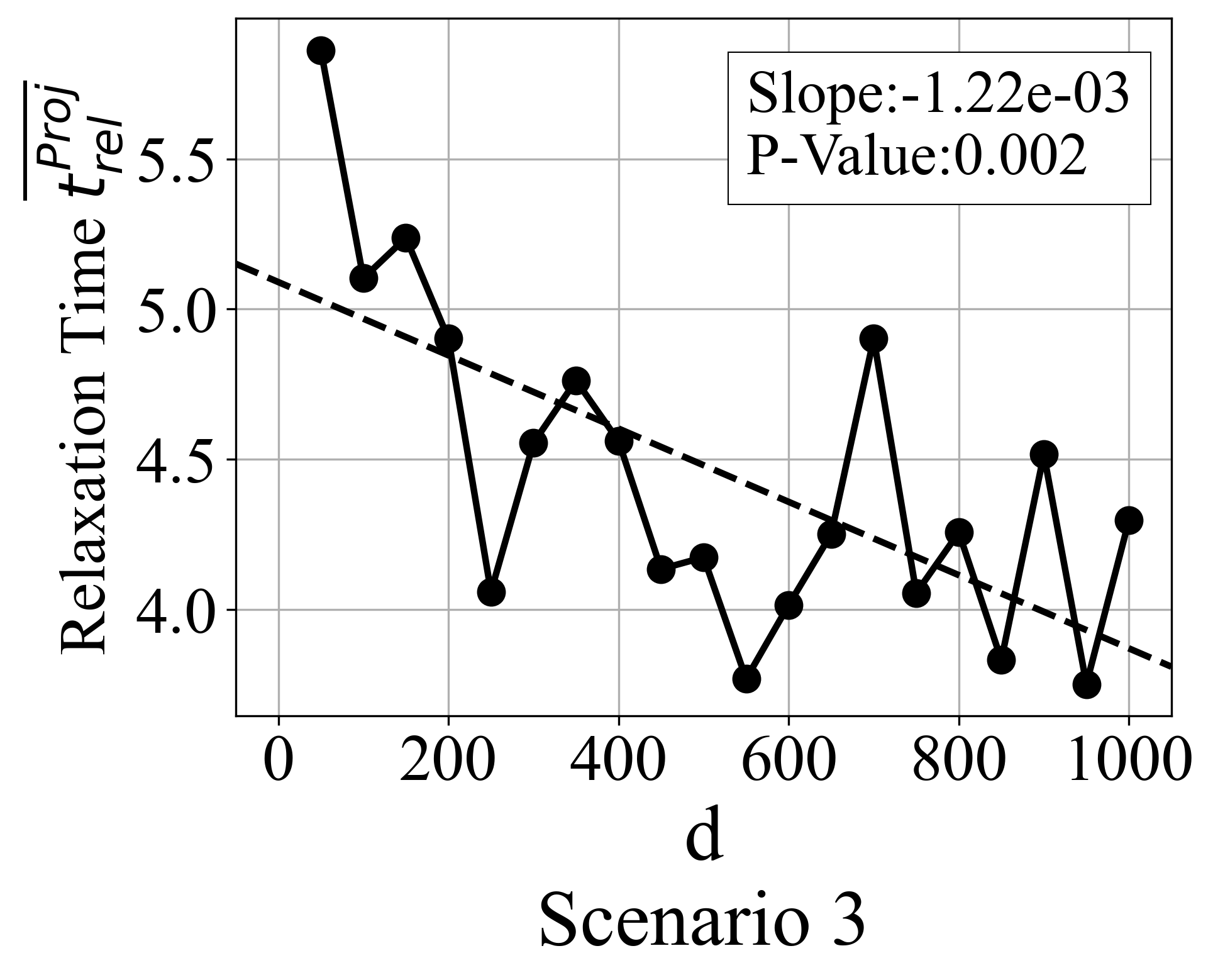}
    \caption{Simulation results for ProbitDA.}
    \label{fig:probit}
\end{figure}

\begin{figure}[t]
    \centering
    \includegraphics[width=0.33\linewidth]{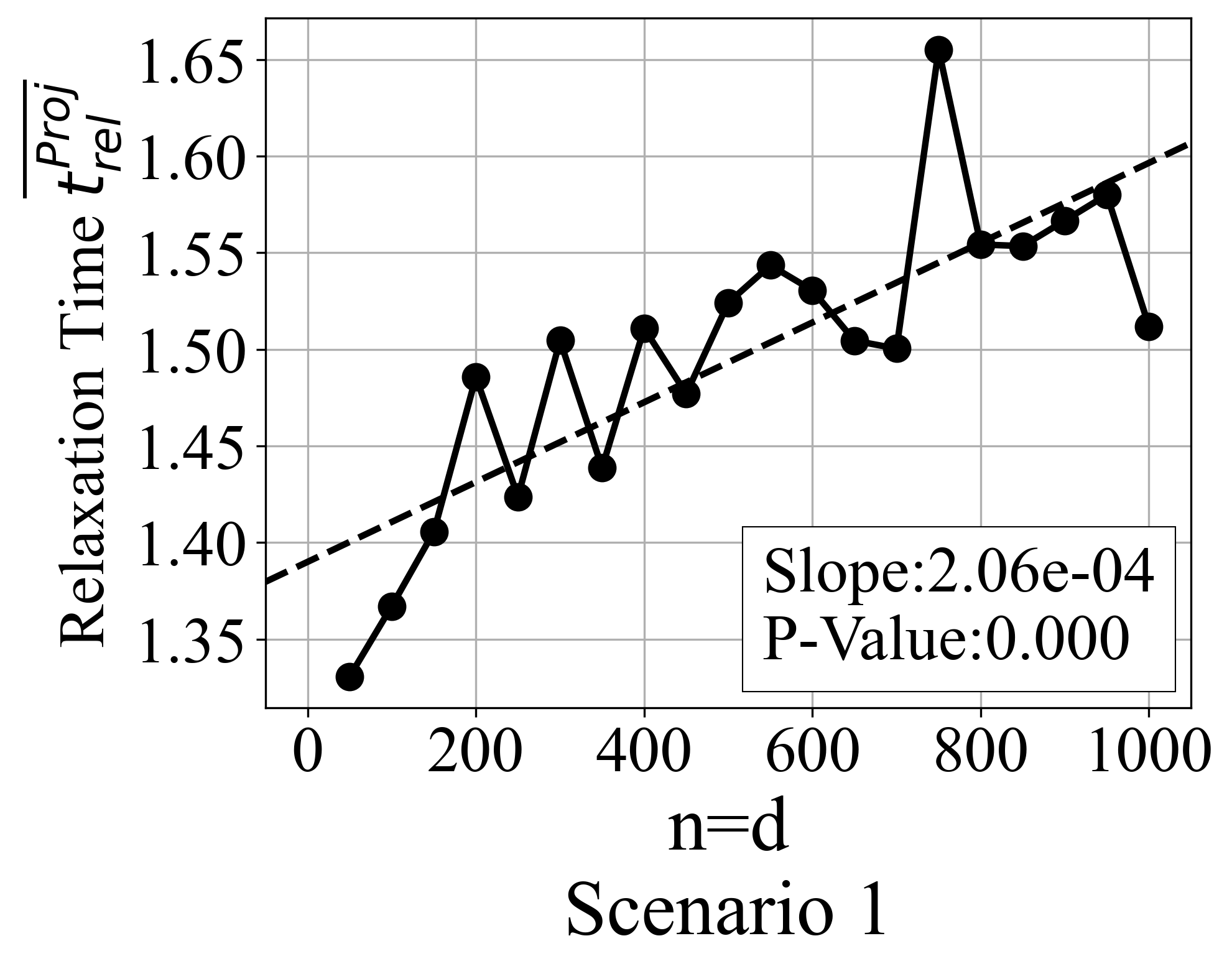}
    \includegraphics[width=0.32\linewidth]{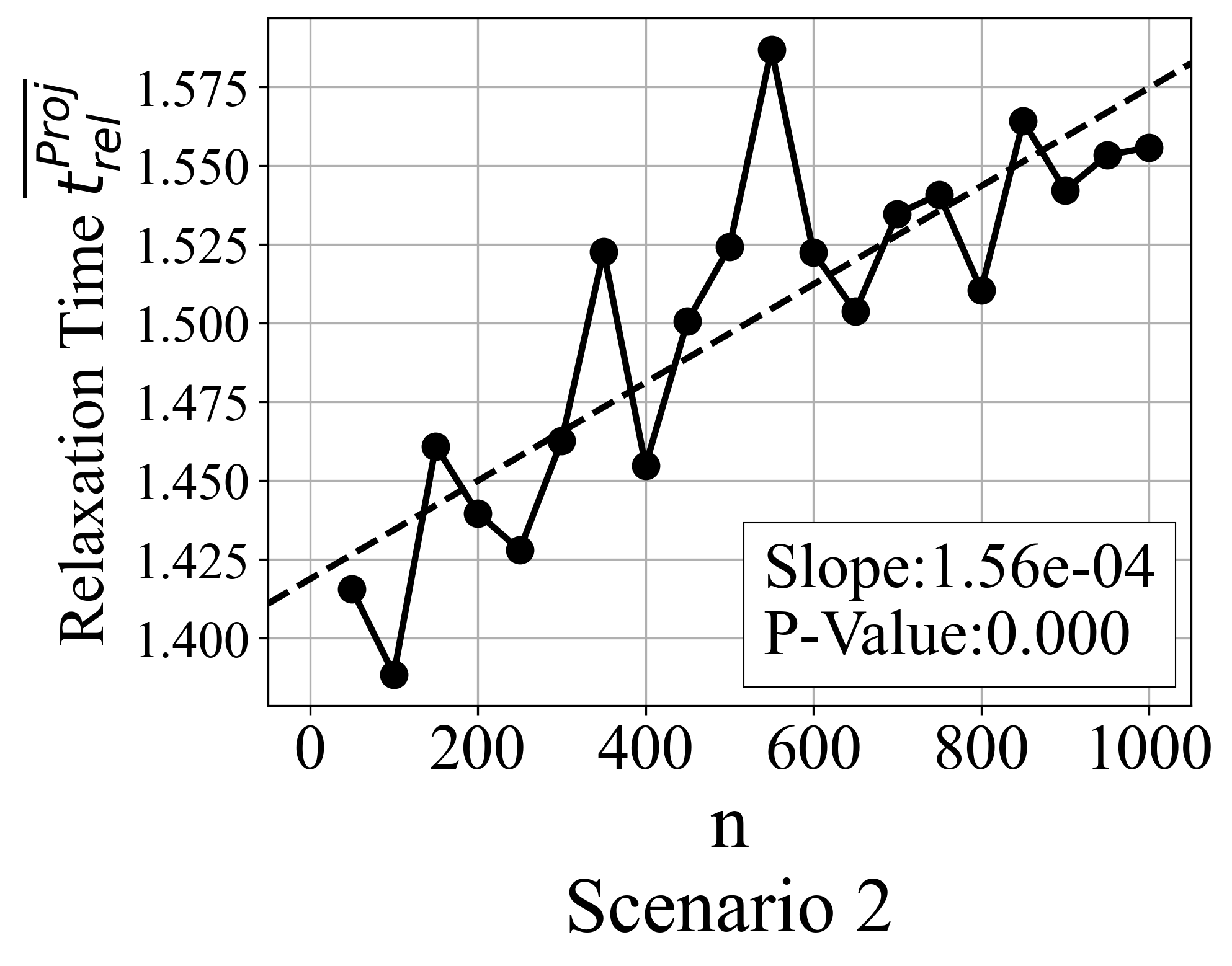}
    \includegraphics[width=0.32\linewidth]{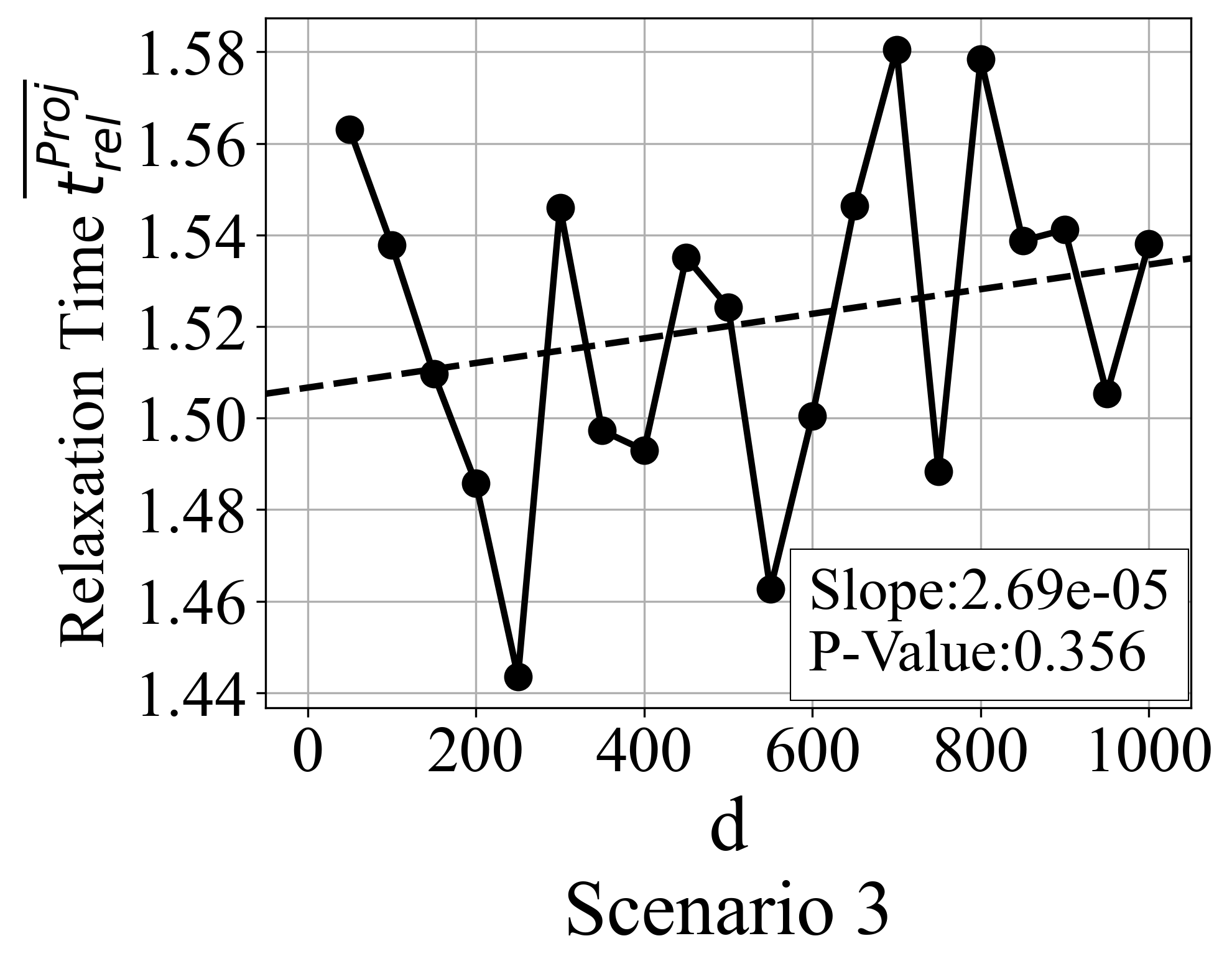}
    \caption{Simulation results for LogitDA. }
    \label{fig:logit}
\end{figure} 

We observe similar patterns to the worst-case setting, but the magnitude is much smaller and the dependence is milder. We note that the linear growth observed in Scenario 2 does not contradict \cite{qin2019convergence}, which shows that when $
d$ is fixed, the mixing time of ProbitDA remains bounded as $n\to \infty$. Since our simulations only cover a limited range of $n$ (up to 1000), the observed trend may level off for larger $
n$ and remain bounded asymptotically.

In view of \cite{qin2019convergence} and the mild dependence seen in the simulations, we suspect that our bound is not tight in the average-case setting, and we leave this question for future work.

\subsection{Results for LassoDA} \label{s:sim-lasso}
We consider the following prior information and data-generating process for LassoDA: 
\begin{align*}
    \xi &= 1, \quad \alpha = 2, \quad \lambda = 1, \quad
    \beta_0  \sim \mathcal{N}(\mathbf{0},  \mathbb{I}_d), \quad v_0 =1, \\ 
    x_i &\mathop{\sim}\limits^{\mathrm{i.i.d.}} \mathcal{N}(\mathbf{0}, \mathbb{I}_d)/\sqrt{d}, \quad y_i \sim \mathcal{N}(x_i^T\beta_0, v_0), \quad i=1,\ldots, n.
\end{align*}

Similarly, for each dataset $\mathcal{D} = [X,y]$, we only generate one $\beta_0$ and keep it fixed throughout. To account for potentially different growth behaviors between $v$ and $\beta$, we plot $\overline{t^{\textup{Proj}}_{\textup{rel}}}$ separately for the $\beta$ and $v$ coordinates in Figure~\ref{fig:lassobeta} and Figure~\ref{fig:lassov}, respectively. 

Overall, the results show similar patterns for $v$ and $\beta$. The results in Scenario 1 show a roughly at most linear joint dependency for $n$ and $d$, so we suspect our bound for LassoDA is not tight. In Scenario 2, we observe a complicated pattern. That is, as $n$ or $d$ grows, $\overline{t^{\textup{Proj}}_{\textup{rel}}}$ first rises and then drops. Our theoretical results do not explain this complex pattern. We leave it for future investigation. In Scenario 3, we observe that $\overline{t^{\textup{Proj}}_{\textup{rel}}}$ exhibits a square-root–type dependence, suggesting a lower bound on the mixing time of LassoDA that scales like $\sqrt{d}$. 

\begin{figure}[t]
    \centering
    \includegraphics[width=0.32\linewidth]{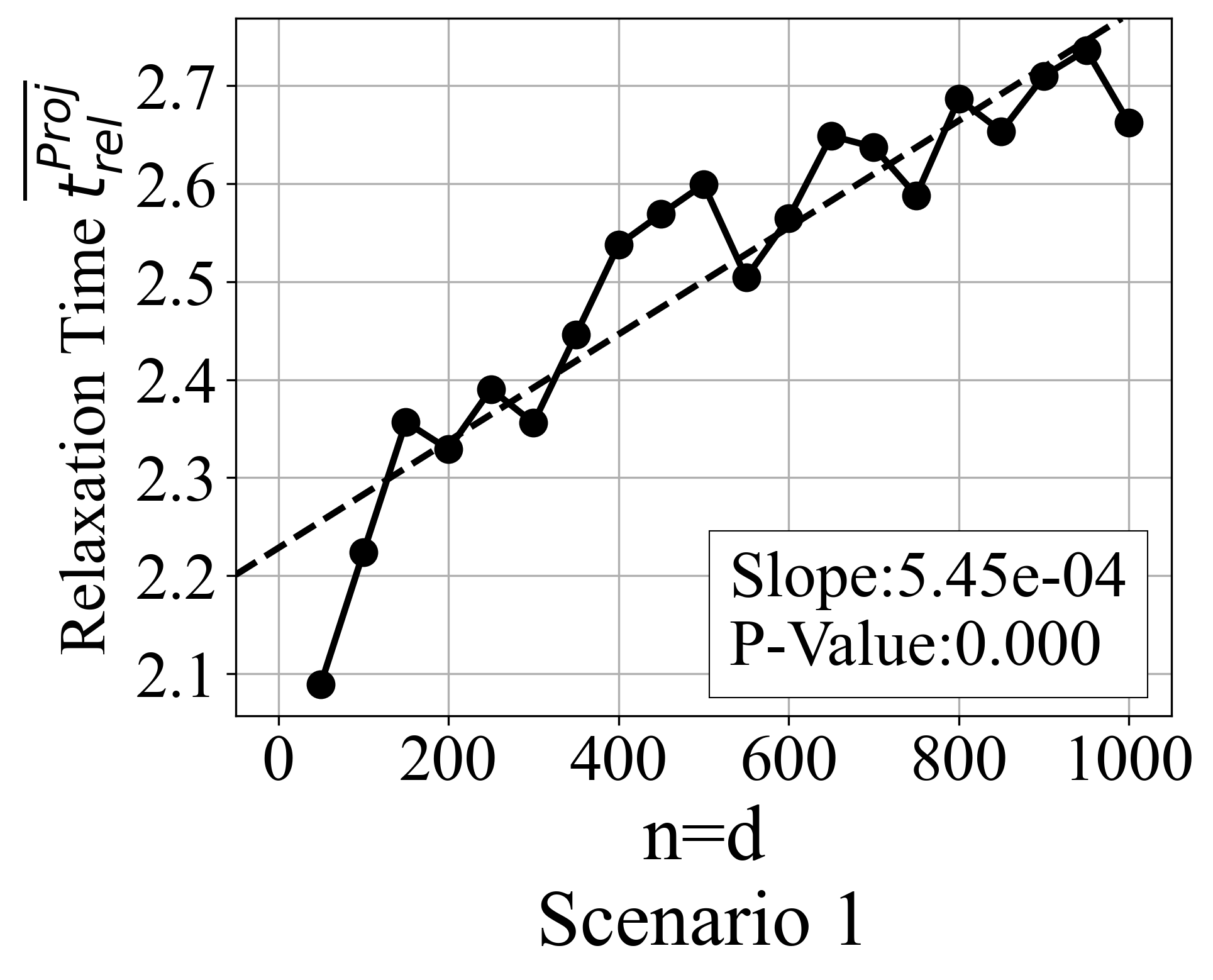}
    \includegraphics[width=0.32\linewidth]{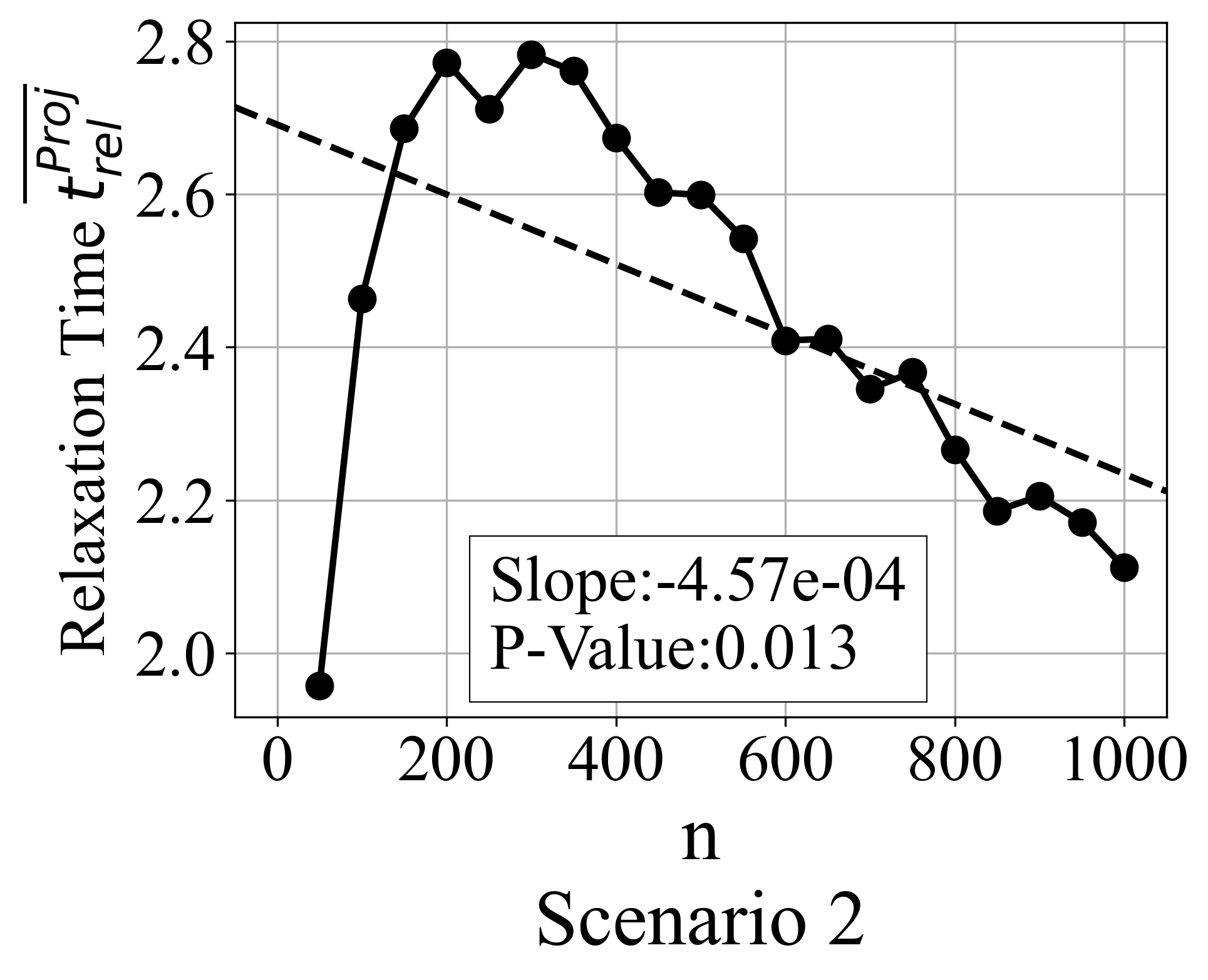}
    \includegraphics[width=0.32\linewidth]{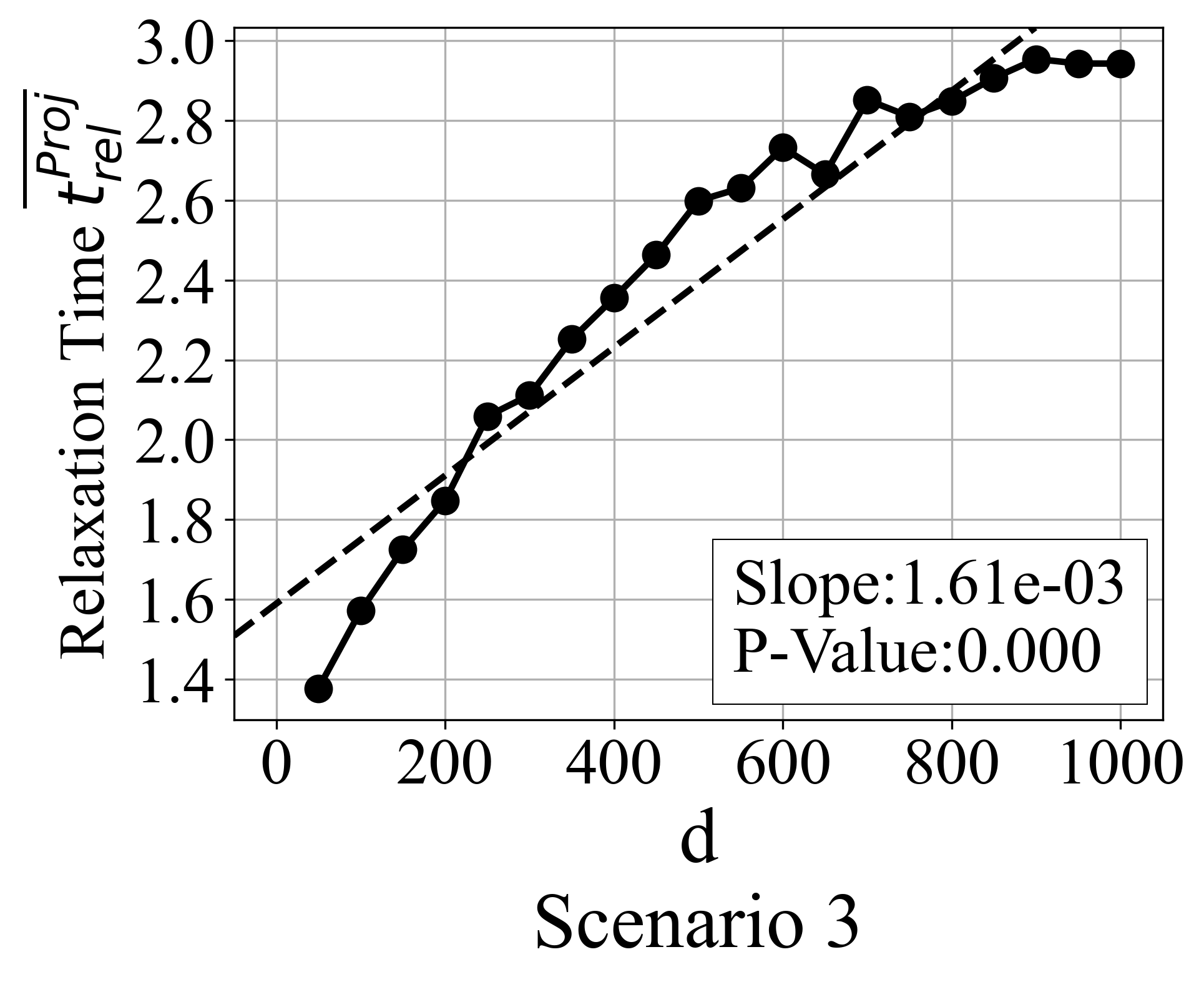}
    \caption{Simulation results for LassoDA for the $\beta$ coordinates.}
    \label{fig:lassobeta}
\end{figure} 

\begin{figure}[t]
    \centering
    \includegraphics[width=0.32\linewidth]{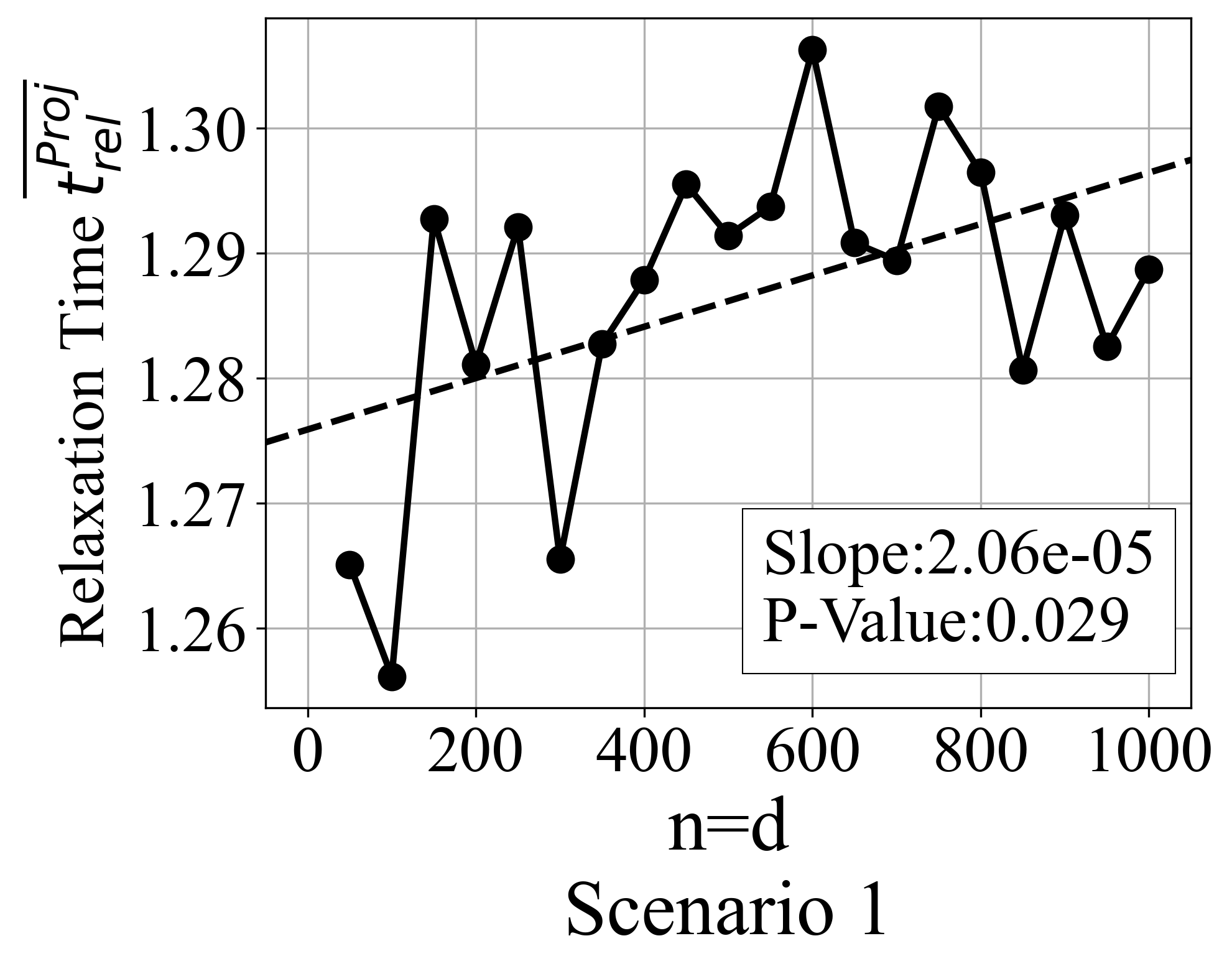}
    \includegraphics[width=0.32\linewidth]{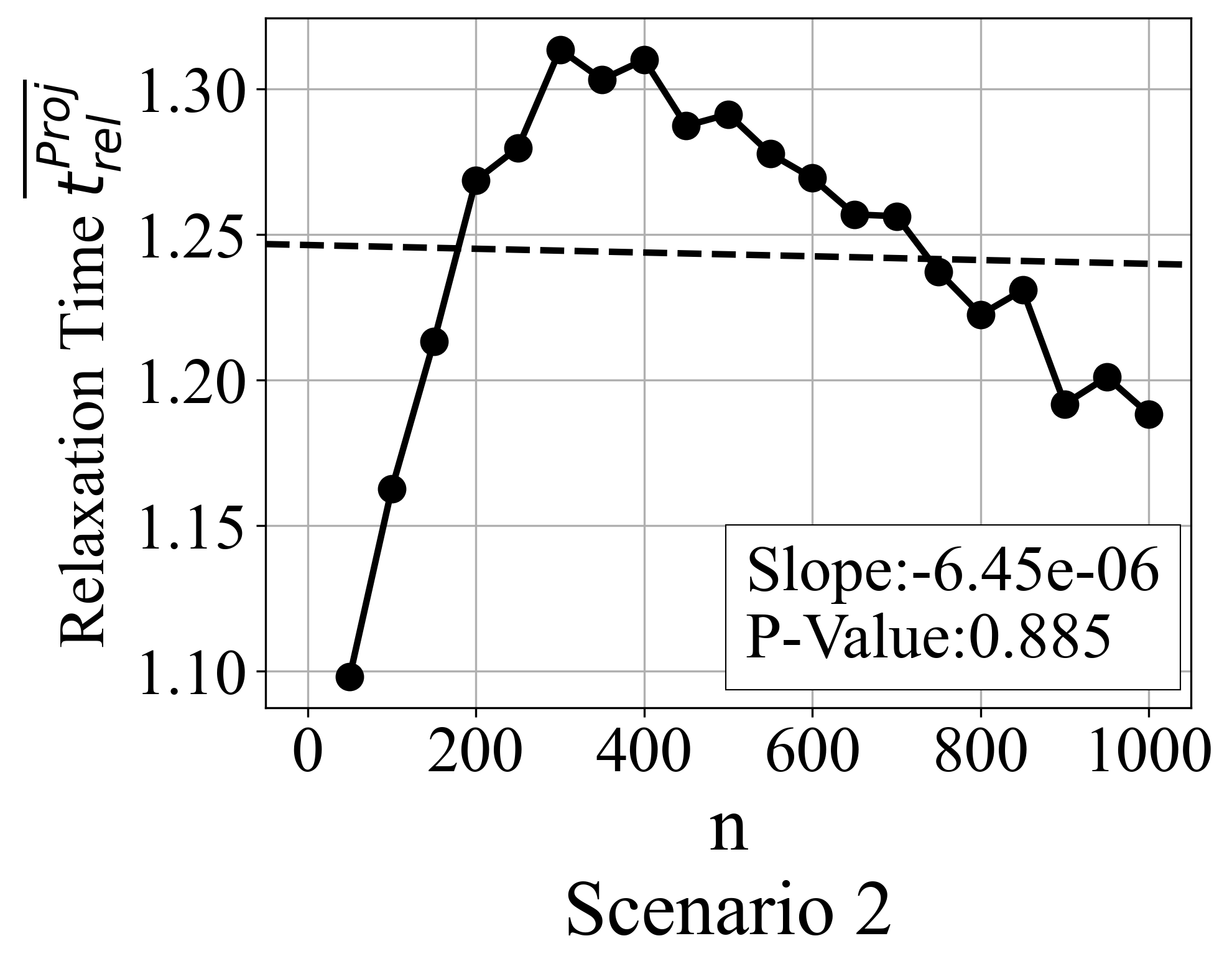}
    \includegraphics[width=0.32\linewidth]{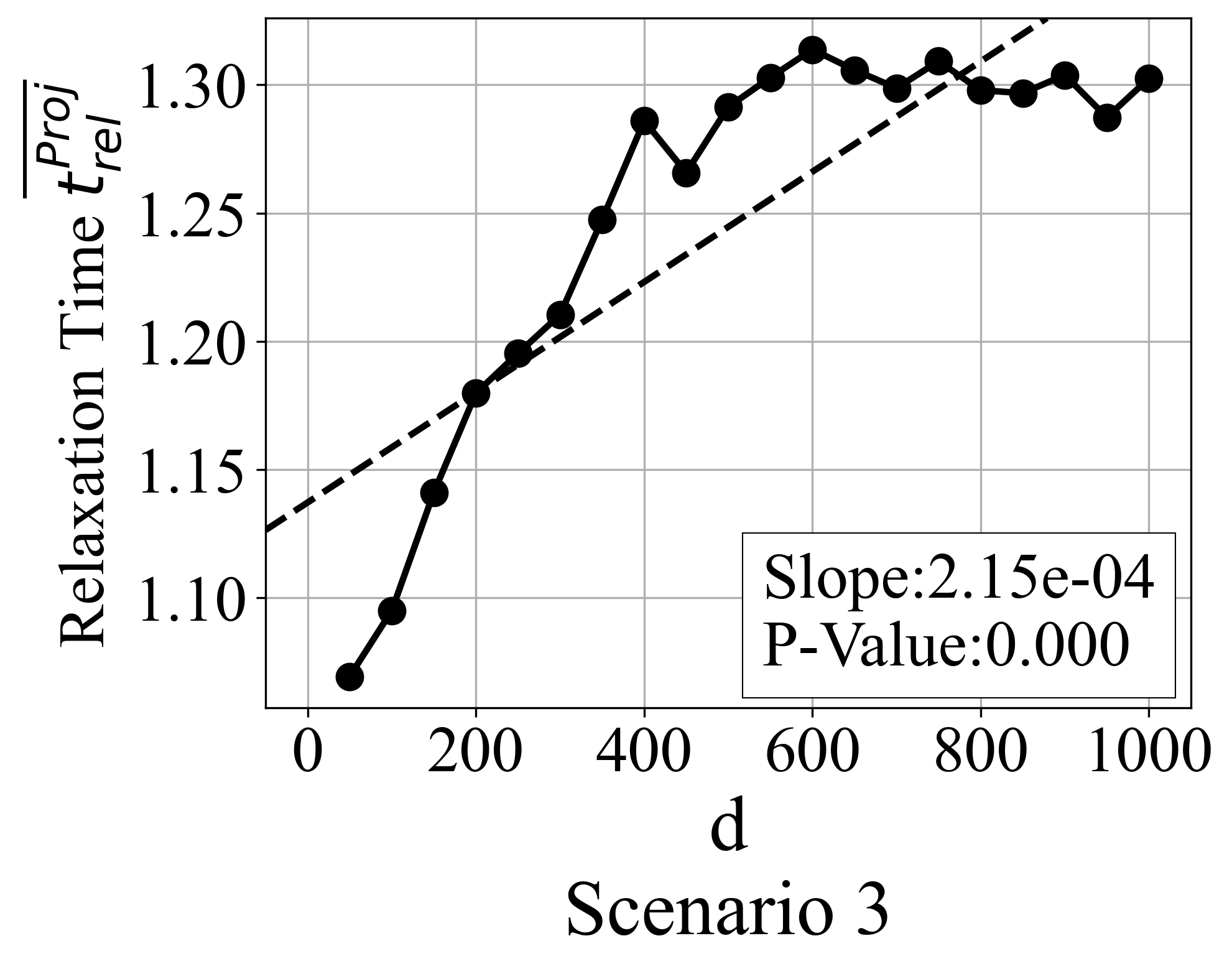}
    \caption{Simulation results for LassoDA for the $v$ coordinate. }
    \label{fig:lassov}
\end{figure}

\section{Deferred Proofs}
\subsection{Proof of Corollary \ref{t:main-cor}}\label{a:ind-proof}
We derive Corollary \ref{t:main-cor} by obtaining a high probability bound for $\|X\|^2_{\mathrm{op}} = \lambda_{\max}(X^TX)$ and then using the bound in Theorem \ref{t:probitlogit}. Suppose $x_{ij}$ is the $(ij)^{th}$ entry of $X$. We note that 
\begin{align*} X^TX = \begin{bmatrix}
        n & \sum_{i=1}^n \frac{x_{i2}}{\sqrt{d}}&\cdots& \sum_{i=1}^n \frac{x_{id}}{\sqrt{d}}\\ 
        \sum_{i=1}^n \frac{x_{i2}}{\sqrt{d}}& && \\
        \vdots & &\textbf{0} &\\
        \sum_{i=1}^n \frac{x_{id}}{\sqrt{d}} & & 
    \end{bmatrix} + \begin{bmatrix}
        0 & 0&\cdots& 0\\ 
        0& && \\
        \vdots & &\frac{1}{d}\tilde{X}^T\tilde{X} &\\
        0 & & 
    \end{bmatrix} 
\end{align*}  
By Weyl's inequality, \begin{align} \label{e:weyl-relation}
    \lambda_{\max}\pa{X^TX} \le n + \frac{1}{d}\lambda_{\max}\pa{\tilde{X}^T\tilde{X}}
\end{align}. Because $\{\tilde{x_i}\}_{i=1}^n$ are zero mean, $\tilde{X}^T\tilde{X}=\sum_{i=1}^n \tilde{x}_i\tilde{x}_i^T$ can be considered as $n$ multiple of the sample covariance matrix $\tilde{\Sigma}=\frac{1}{n}\sum_{i=1}^n \tilde{x}_i\tilde{x}_i^T$. 
Therefore, we can also draw upon a rich literature of high probability error bounds for covariance estimation that control $\norm{\tilde{\Sigma} - \Sigma}_{\mathrm{op}}$. We first cite the techniques, and then use them to prove Corollary \ref{t:main-cor}. 

\begin{lemma}[{Covariance Estimation for Sub-Gaussian Distributions, \cite[Exercise 4.7.3]{vershynin2018high}}] \label{l:sub-Gaussian-cov}  Let $X$ be a sub-gaussian random vector in $\mathbb{R}^d$. More precisely, assume that there exists $K \ge 1$ such that 
$$
\| \langle\,X,x \rangle \|_{\psi_2} \le K \sqrt{\mathbb E\langle\,X,x \rangle^2} \text{   for all $x \in \mathbb{R}^d$}.
$$
Then for all $u \ge 0$, 
$$\norm{ \tilde{\Sigma} - \Sigma }_{\mathrm{op}} \le c K^2 \left(\sqrt{\frac{d+u}{n}} + \frac{d+u}{n}\right) \| \Sigma \|_{\mathrm{op}}
$$
with probability at least $1-2e^{-u}$.
\end{lemma}
\begin{lemma}[Covariance Estimation for Log-concave Isotropic Measures, \cite{adamczak2010quantitative,adamczak2011sharp}] \label{l:logconcave-cov}Assume X is a log-concave isotropic random vector in $\mathbb{R}^d$. Then, there exists absolute constants $K$ and $\psi$ such that
\begin{enumerate}
    \item If $d \le n$, 
$$
\norm{\tilde{\Sigma}-\mathbb{I}_d}_{\mathrm{op}}  \le c (\psi + K)^2 \sqrt{\frac{d}{n}}
$$ 
\item 
If $d > n$,
$$
\norm{\tilde{\Sigma}-\mathbb{I}_d}_{\mathrm{op}} \le c (\psi + K)^2 \frac{d}{n}
$$
\end{enumerate}    
with probability at least $1-\exp(-c\sqrt{d})$.
\end{lemma}
\begin{proof}[Proof of Corollary \ref{t:main-cor}] We first prove a bound for $\|\tilde{X}^T\tilde{X}\|_{\mathrm{op}}$. 
\paragraph*{Under Sub-Gaussianity} By Lemma~\ref{l:sub-Gaussian-cov}, we have 
\begin{align} \label{e:sub-gauss-bound}
   \norm{\tilde{X}^T\tilde{X}}_{\mathrm{op}} \le n \| \Sigma\|_{\mathrm{op}} + n\norm{ \tilde{\Sigma} - \Sigma }_{\mathrm{op}} \le n \| \Sigma\|_{\mathrm{op}} + c n K^2 \pa{\sqrt{\frac{d+u}{n}} + \frac{d+u}{n}} \| \Sigma \|_{\mathrm{op}}. 
\end{align}
This holds with probability at least $1-2\exp(-u)$.

\paragraph*{Under Log-concavity}
We consider a general log-concave random vector $X$ with covariance $\Sigma$. Applying Lemma~\ref{l:logconcave-cov} to the isotropic random vector $\Sigma^{-1/2} X$, we have with probability at least $1-\exp(-c'\sqrt{d})$, 
$$
\norm{\Sigma^{-1/2}\tilde{\Sigma}\Sigma^{-1/2}-\mathbb{I}_d}_{\mathrm{op}}  \le c \pa{\sqrt{\frac{d}{n}} + \frac{d}{n}}
$$
By left- and right-multiplying both sides by $\|\Sigma^{1/2}\|_{\mathrm{op}}$, we have 
$$
\norm{\tilde{\Sigma}-\Sigma}_{\mathrm{op}} \le \norm{\Sigma^{1/2}}_{\mathrm{op}} \norm{\Sigma^{-1/2}\tilde{\Sigma}\Sigma^{-1/2}-\mathbb{I}_d}_{\mathrm{op}} \norm{\Sigma^{1/2}}_{\mathrm{op}} \le c\pa{\sqrt{\frac{d}{n}} + \frac{d}{n}} \norm{\Sigma}_{\mathrm{op}}.
$$ 
Then, we can get a high probability bound for $\|\tilde{X}^T\tilde{X}\|_{\mathrm{op}}$ such that with probability at least $1-\exp(-c'\sqrt{d})$, we have 
\begin{align} \label{e:log-concave-bound}
 \norm{ \tilde{X}^T\tilde{X}}_{\mathrm{op}} \le n \| \Sigma\|_{\mathrm{op}} + n\norm{ \tilde{\Sigma} - \Sigma }_{\mathrm{op}} \le n \| \Sigma\|_{\mathrm{op}} + cn\pa{\sqrt{\frac{d}{n}} + \frac{d}{n}} \|\Sigma\|_{\mathrm{op}}
\end{align}

We obtain the corollary by using the bounds \eqref{e:sub-gauss-bound} and \eqref{e:log-concave-bound} in \eqref{e:weyl-relation} and then applying Theorem 3.1. 
\end{proof}

\subsection{Proof of Lemma 4.3} 
\label{a:mtub-nolazy}
We use a different method from the proof in \cite{lovasz1993random}. This method is closely related to the proof of Lemma 3 in \cite{chen2020fast}, which considers the conductance profile instead of conductance. With an additional assumption of nonnegative spectrum, this method allows us to drop the laziness condition in the original statement in \cite{lovasz1993random}. 

We start by introducing some notations. Let $L^2(\pi)$ be the space of square integrable functions under function $\pi$ with inner product 
$$
\ip{f,g}_\pi = \int f g d\pi. 
$$ The expectation $\mathbb E_\pi: L^2(\pi) \rightarrow \mathbb{R}$ and the variance $\textup{Var}_\pi: L^2(\pi) \rightarrow \mathbb R$ with respect to the measure $\pi$ are given by
$$
\mathbb E_\pi (f) = \int f d \pi, \text{ and } \textup{Var}_\pi (f) = \int (f - E_\pi (f))^2  d \pi. 
$$
\begin{proof}[Proof of Lemma 4.3]
Suppose that $\mathcal{P}$ has spectral gap $\gamma$. That is, with $L^2_0(\pi) = \{\pi \in L^2(\pi): \mathbb E_\pi(f)=0\}$, 
$$
\gamma = \inf_{f \in L_0^2(\pi)} \frac{\ip{f,(\mathbb{I}-\mathcal{P})f}_\pi}{\ip{f,f}_\pi}. 
$$ Combining this with the Cheeger's inequality \cite{lawler1988bounds}, which states $\gamma \ge \frac{\Phi^2}{2}$, we obtain that for any $f \in L^2_0(\pi)$, 
\begin{align}  \label{e:eq1}
    \ip{f, (\mathbb{I}-\mathcal{P})f}_\pi \ge \frac{\Phi^2}{2} \ip{f,f}_\pi. 
\end{align}
The reversibility of $\mathcal{P}$ implies that $\mathbb{I} -\mathcal{P}, \mathbb{I} +\mathcal{P},$ and $(\mathbb{I} -\mathcal{P})^{1/2}$ all commute. Then, suppose that $\lambda_{\min} = \inf_{f \in L^2(\pi)} \frac{\ip{f, \mathcal{P}f}_\pi}{\ip{f,f}_\pi}$, we have 
\begin{align}
    \ip{f, \pa{\mathbb{I}-\mathcal{P}^2}f}_\pi
    &= \ip{f, (\mathbb{I}-\mathcal{P})^{1/2}
    (\mathbb{I}+\mathcal{P})
    (\mathbb{I}-\mathcal{P})^{1/2}f}_\pi
    = \ip{(\mathbb{I}-\mathcal{P})^{1/2} f, 
    (\mathbb{I}+\mathcal{P})
    (\mathbb{I}-\mathcal{P})^{1/2}f}_\pi \nonumber\\
    &\ge (1+\lambda_{\min}) \ip{(\mathbb{I}-\mathcal{P})^{1/2} f, (\mathbb{I}-\mathcal{P})^{1/2} f}_\pi \nonumber \\
    &= (1+\lambda_{\min}) \ip{f, (\mathbb{I}-\mathcal{P})f}_\pi. \label{e:eq2}
\end{align}
Let $r=\frac{(1+\lambda_{\min})\Phi^2}{2}$. Combining~\eqref{e:eq1} and~\eqref{e:eq2}, we obtain that 
\begin{align*}
    \ip{\mathcal{P}f, \mathcal{P}f}_\pi = \ip{f, \mathcal{P}^2f}_\pi \le  (1-r)\ip{f, f}_\pi.
\end{align*}
Taking $f=\frac{\nu}{\pi}-1$, we have 
\begin{align}\label{e:eq3}
    \textup{Var}_\pi \pa{\frac{\nu \mathcal{P}}{\pi}} \le (1-r) \textup{Var}_\pi \pa{\frac{\nu}{\pi}}.
\end{align}
Iterating~\eqref{e:eq3} gives 
\begin{align}
 \textup{Var}_\pi \pa{\frac{\nu \mathcal{P}^k}{\pi}} \le (1-r)^k \textup{Var}_\pi \pa{\frac{\nu}{\pi}}.   
\end{align}
Therefore, we have 
\begin{align*}
    \TV(\nu \mathcal{P}^k, \pi) &= \frac{1}{2} \int |\nu\mathcal{P}^k - \pi| = \frac{1}{2} \int \abs{\frac{\nu\mathcal{P}^k}{\pi} - 1}d\pi \\
    &\le_{(i)} \frac{1}{2} \sqrt{\textup{Var}_\pi\pa{\frac{\nu\mathcal{P}^k}{\pi}}} \\
    &\le_{(ii)} \frac{1}{2}\sqrt{(1-r)^k \textup{Var}_\pi\pa{\frac{\nu}{\pi}}} \\
    &\le_{(iii)} \frac{1}{2} \sqrt{(1-r)^k \eta} \le \frac{1}{2} \sqrt{\eta} e^{-rk} \\
    &\le_{(iv)} \frac{1}{2} \sqrt{\eta} e^{-\frac{\Phi^2}{2}k}, 
\end{align*}
where we obtain $(i)$ by Cauchy-Schwarz inequality, $(ii)$ by~\eqref{e:eq3}, $(iii)$ by $\textup{Var}_\pi(\frac{\nu}{\pi}) \le \int \pa{\frac{\nu}{\pi}}^2 d\pi \le \sup \frac{\nu}{\pi} \int d\nu \le \eta$, and $(iv)$ by $ \lambda_{\min} \ge 0$. 
\end{proof}

\subsection{Proof of Lemma 4.7} \label{a:mtub-nolazy-log}
Compared to the original statement of Lemma 3 in \cite{chen2020fast}, Lemma 4.7 drops the laziness assumption and adopts the additional condition of nonnegative spectrum. To justify this new statement, we need a one-line modification of the original proof in \cite{chen2020fast}. That is, we use a different way to lower bound the two-step Dirichlet form by the one-step Dirichlet form in equation (72)-(i) of \cite{chen2020fast}. 

Specifically, we let $L^2(\pi)$ be the space of square integrable functions under function $\pi$ with inner product $\ip{f,g}_\pi = \int f g d\pi$. The \textit{Dirichlet form} $\mathcal{E}:L^2(\pi) \times L^2(\pi) 
 \rightarrow \mathbb{R}$ associated with the transition kernel K is defined as 
$$
\mathcal{E}_K(f,g) = \frac{1}{2}\int (f(x)-g(y))^2 K(x, dy) \pi(x) dx.
$$ In the equation (70) and (72)-(i) of \cite{chen2020fast}, assuming the chain is $\zeta$-lazy,  they prove that for any $f \in  L^2(\pi)$
\begin{align}\label{e:eq4}
\mathcal{E}_{\mathcal{P}^2}(f,f) \ge 2\zeta \mathcal{E}_{\mathcal{P}}(f,f).
\end{align}

Instead of using laziness, by noting that $\mathcal{E}_K(f,f) = \ip{f, (\mathbb{I}-K)f}_\pi$, we can follow the same arguments in~\eqref{e:eq2} to obtain that 
\begin{align}\label{e:eq5}
\mathcal{E}_{\mathcal{P}^2}(f,f) \ge (1+\lambda_{\min}) \mathcal{E}_{\mathcal{P}}(f,f),
\end{align}
where $\lambda_{\min}$ is the minimum eigenvalue of $\mathcal{P}$. Replacing~\eqref{e:eq4} by~\eqref{e:eq5} in the proof for Lemma 3 in \cite{chen2020fast}, and keeping the rest unchanged, we yield that with $\eta$-warm start $\mu_0$,
$$
\tau_2(\mu_0, \epsilon) \le \int_{4/\eta}^{8/\epsilon^2} \frac{16 d v}{(1+\lambda_{\min})v \tilde{\Phi}^2(v)},
$$
where $\tau_2(\eta, \epsilon)$ is defined as the $\mathcal{L}_2$-mixing time in \cite{chen2020fast},
$$
\tau_2(\mu_0, \epsilon) = \inf \{k \in \mathbb{N}: d_2(\mu_0\mathcal{P}^k, \pi)\le \epsilon\} \text{ where } d_2(\mu, \nu) = \sqrt{\int_{\mathbb{R}^d}\pa{\frac{\mu(x)}{\nu(x)}-1}^2 \nu(x) dx}.
$$
By $\chi^2(\mu||\nu) = d_2^2(\mu, \nu) $, we have
$$
t^{\chi^2}_{\Psi}(\eta, \epsilon) = \sup_{\mu_0 \text{ is a $\eta$-warm start}}\tau_2(\mu_0,  \sqrt{\epsilon}) \le \int_{4/\eta}^{8/\epsilon} \frac{16 d v}{(1+\lambda_{\min})v \tilde{\Phi}^2(v)}.
$$
Due to nonnegative spectrum,  $\lambda_{\min} \ge 0$, and thus
$$
t^{\chi^2}_{\Psi}(\eta, \epsilon) \le \int_{4/\eta}^{8/\epsilon} \frac{16 d v}{v \tilde{\Phi}^2(v)}.
$$
The lemma follows.

\subsection{Proof of Lemma 4.8}\label{a:proof-transform}
\begin{proof}[Proof of Lemma 4.8]
Suppose the Markov chain $\Psi$ has an associated triple $(\nu, \mathcal{P}, \pi)$. Then, its $T$-transformed Markov chain has the triple $(\nu_T, \mathcal{P}_T, \pi_T)$. 
First, note that
\[
\mu \mathcal P_T = T_\# (((T^{-1})_\#\mu)\mathcal P).
\]
In particular, for $\mu=\nu_T$, 
$\nu_T \mathcal P_T = T_\# (\nu \mathcal P)$. Iterating this and putting $\mu = T_\# \nu$ gives $\nu_T \mathcal P_T^k = T_\# (\nu \mathcal P^k)$. 

By the invariance of TV distance under one-to-one transformation, we have for any $k \in \mathbb{N}^{+}$ that
$$\TV(\nu \mathcal{P}^k,\pi)=\TV(\nu_T \mathcal{P}^{k}_T,\pi_T).$$
Furthermore, $T$ being a bijection implies that $\nu_T$ is also an $\eta$-warm start, and thus the lemma follows.
\end{proof}
\subsection{Proof of Lemma 4.9}\label{s:proof-lasso-iso}
\begin{proof}[Proof of Lemma 4.9]
The target of the $\varphi$-marginal of the $T$-transformed 
target distribution of LassoDA 
\begin{align} 
\pi_{T_\varphi}(\varphi)=\pi_T(\varphi|y) \propto  \exp\{- \lambda \|\varphi\|_1\} \int_{\R^+} \rho^{n+2\alpha -2} \exp\left\{-\frac{1}{2}\|\rho y - X\varphi\|_2^2 -\rho^2 \xi\right\} d\rho.
\end{align} 
is in general weakly-log-concave. We use Lemma 4.2 to relate the Cheeger constant of the target of LassoDA to the known Cheeger constant of the double exponential distribution (Lemma 4.1(1)). Let $\mu(\varphi)=(\frac{\lambda}{2})^d e^{- \lambda \|\varphi\|_1}$ be the reference double exponential distribution. To utilize Lemma 4.2, we need to measure the infinity-divergence between $\pi_{T_\varphi}$ and $\mu$ (the $L^{\infty}$ norm of their ratio):
\begin{align}
\left\| \frac{d \pi_{T_\varphi}}{d \mu}\right\|_{L^{\infty}} &= \max_{\varphi} \frac{ e^{- \lambda \|\varphi\|_1}\int_{\rho\in \R^+} \rho^{n+2\alpha -2} e^{-\rho^2 \xi -\frac{1}{2}\|\rho y - X\varphi\|_2^2 } d\rho }{\int_{\rho\in \R^+} \rho^{n+2\alpha -2} e^{-\rho^2 \xi} \int_{\varphi\in \R^d} e^{- \lambda \|\varphi\|_1-\frac{1}{2}\|\rho y - X\varphi\|_2^2 } d\varphi d\rho } \frac{(2/\lambda)^d}{ e^{- \lambda \|\varphi\|_1}} \nonumber\\
&= (2/\lambda)^d \max_{\varphi} \frac{ \int_{\rho\in \R^+}  e^{ -\frac{1}{2}\|\rho y - X\varphi\|_2^2} \rho^{n+2\alpha -2} e^{-\rho^2 \xi } d\rho }{\int_{\rho\in \R^+} \rho^{n+2\alpha -2} e^{-\rho^2 \xi} \int_{\varphi\in \R^d} e^{- \lambda \|\varphi\|_1-\frac{1}{2}\|\rho y - X\varphi\|_2^2 } d\varphi d\rho } \nonumber \\
&\le (2/\lambda)^d  \frac{ \int_{\rho\in \R^+}   \rho^{n+2\alpha -2} e^{-\rho^2 \xi } d\rho }{\int_{\rho\in \R^+} \rho^{n+2\alpha -2} e^{-\rho^2 \xi} \int_{\varphi\in \R^d} e^{- \lambda \|\varphi\|_1-\frac{1}{2}\|\rho y - X\varphi\|_2^2 } d\varphi d\rho } \label{e:lasso-important-quantity}\\
&=  \frac{ (2/\lambda)^d \frac{1}{2}\Gamma(\frac{n+2\alpha-1}{2}) \xi^{-\frac{n+2\alpha-1}{2}} }{\int_{\rho\in \R^+} \rho^{n+2\alpha -2} e^{-\rho^2 \xi} \int_{\varphi\in \R^d} e^{- \lambda \|\varphi\|_1-\frac{1}{2}\|\rho y - X\varphi\|_2^2 } d\varphi d\rho }. \nonumber
\end{align}
It remains to lower bound the partition function in the denominator. Since $\|\varphi\|_1 = \sum_{j=1}^d |\varphi_j| \le \sum_{j=1}^d (\varphi_j^2+1)= d + \|\varphi\|_2^2$, we have 
\begin{align*}
&\int_{\rho\in \R^+} \rho^{n+2\alpha -2} e^{-\rho^2 \xi} \int_{\varphi\in \R^d} e^{- \lambda \|\varphi\|_1-\frac{1}{2}\|\rho y - X\varphi\|_2^2} d\varphi d\rho \\
&\ge  e^{-\lambda d}\int_{\rho\in \R^+} \rho^{n+2\alpha -2} e^{-\rho^2 \xi} \int_{\varphi\in \R^d} e^{- \lambda \|\varphi\|_2^2-\frac{1}{2}\|\rho y - X\varphi\|_2^2 } d\varphi d\rho \\
&= e^{-\lambda d}\int_{\rho\in \R^+} \rho^{n+2\alpha -2} e^{-\rho^2 \xi} \int_{\varphi\in \R^d} e^{-\frac{1}{2}(\varphi^T(X^TX+2\lambda \mathbb{I})\varphi -2\rho y^T X \varphi + \rho^2 y^Ty)} d\varphi d\rho \\
&= e^{-\lambda d} \int_{\rho\in \R^+} \rho^{n+2\alpha -2} e^{-\rho^2 \xi-\frac{1}{2}\rho^2 y^T(\mathbb{I}-X(X^TX+2\lambda \mathbb{I})^{-1}X^T)y}\\
&\quad \cdot \int_{\varphi\in \R^d} e^{-\frac{1}{2}(\varphi-\rho(X^TX+2\lambda\mathbb{I})^{-1} X^T y)^T(X^TX+2\lambda\mathbb{I})(\varphi-\rho(X^TX+2\lambda\mathbb{I})^{-1} X^T y)} d\varphi d\rho \\
&= e^{-\lambda d} (2\pi)^{d/2} |(X^TX+2\lambda \mathbb{I})^{-1}|^{1/2} \int_{\rho\in \R^+} \rho^{n+2\alpha -2} e^{-\rho^2 \xi-\frac{1}{2}\rho^2 y^T(\mathbb{I}-X(X^TX+2\lambda \mathbb{I})^{-1}X^T)y} d\rho \\
&= \frac{1}{2}e^{-\lambda d} (2\pi)^{d/2} |(X^TX+2\lambda \mathbb{I})^{-1}|^{1/2} \int_\gamma \gamma^{\frac{n+2\alpha -3}{2}} e^{-\gamma (\xi+\frac{1}{2} y^T(\mathbb{I}-X(X^TX+2\lambda \mathbb{I})^{-1}X^T)y)} d\gamma \; \pa{\gamma=\rho^2}\\
&=\frac{1}{2}e^{-\lambda d} (2\pi)^{d/2} |(X^TX+2\lambda \mathbb{I})^{-1}|^{1/2}  \Gamma\left(\frac{n+2\alpha-1}{2}\right)  \\
& \cdot \pa{\xi+\frac{1}{2} y^T(\mathbb{I}-X(X^TX+2\lambda \mathbb{I})^{-1}X^T)y}^{-\frac{n+2\alpha-1}{2}} .
\end{align*}
Therefore, 
\begin{align} \label{e:distance-upperbound}
   \left\|\frac{d\pi_{T_\varphi}}{d\mu}\right\|_{L^{\infty}} &\le  \underbrace{e^{\lambda d} \pa{\frac{\sqrt{2}}{\lambda \sqrt{\pi}}}^d}_{(a)} \underbrace{|X^TX+2\lambda \mathbb{I}|^{1/2}}_{(b)} \underbrace{\pa{\frac{\xi+\frac{1}{2} y^T(\mathbb{I}-X(X^TX+2\lambda \mathbb{I})^{-1}X^T)y}{\xi}}^{\frac{n+2\alpha-1}{2}}}_{(c)}.
\end{align} 
Next, we analyze the dependency on $n$ and $d$ of the logarithm of the three parts in~\eqref{e:distance-upperbound}. 
\paragraph*{Part (a)}
\begin{align*}
    \log e^{\lambda d} \pa{\frac{\sqrt{2}}{\lambda \sqrt{\pi}}}^d = \lambda d + d \log \pa{\frac{\sqrt{2}}{\lambda \sqrt{\pi}}}=\mathcal{O}(d).
\end{align*} 
\paragraph*{Part (b)}  Suppose $ \norm{X}_{\mathrm{op}}^2 = \lambda_d \ge \cdots \ge \lambda_1 \ge 0$ are the eigenvalues of $X^TX$. Then, we have 
\begin{align*}
    \log |X^TX+2\lambda \mathbb{I}| = \log \prod_{i=1}^d(\lambda_i+2\lambda) \le \log \prod_{i=1}^d(\lambda_d+2\lambda)=d \log (\norm{X}_{\mathrm{op}}^2 +2\lambda) =_{(i)}\mathcal{O}(d\log nd)  
\end{align*} 
where in (i) we use the assumption $\|X\|_{\mathrm{op}}=\Poly(nd)$.
\paragraph*{Part (c)} We first notice that $$y^T(\mathbb{I}-X(X^TX+2\lambda \mathbb{I})^{-1}X^T)y 
\le \|y\|_2^2 (1-\lambda_{\min}(X(X^TX+2\lambda \mathbb{I})^{-1}X^T)) \le \|y\|_2^2, $$
where the last inequality comes from the fact that $X(X^TX+2\lambda \mathbb{I})^{-1}X^T$ is positive semi definite. Then, using the assumption $\|y\|_2=\Poly(n)$,
\begin{align*}
    \log \pa{\frac{\xi+\frac{1}{2} y^T(\mathbb{I}-X(X^TX+2\lambda \mathbb{I})^{-1}X^T)y}{\xi}}^{\frac{n+2\alpha-1}{2}} \le \log \pa{\frac{\xi+\frac{1}{2} \|y\|_2^2}{\xi}}^{\frac{n+2\alpha-1}{2}} =\mathcal{O}(n\log n).
\end{align*}
Putting the three parts together, we get that 
\begin{align*}
    \log\left\|\frac{d\pi_{T_\varphi}}{d\mu}\right\|_{L^{\infty}} = \mathcal{O}(d\log d + n\log n).
\end{align*}
Applying Lemma 4.2 and Lemma 4.1(1), the Cheeger constant of $\pi_{T_\varphi}$ satisfies  
\begin{align*}
    \Ch(\pi_{T_\varphi}) \le c(d\log d + n\log n) \Ch(\mu) = \mathcal{O}(d\log d + n\log n).
\end{align*}
\end{proof}

\subsection{Proof of Lemma 4.11}\label{s:ig-proof}
\begin{proof}[Proof of Lemma 4.11]
WLOG, we assume that $\mu_2>\mu_1$. Let $f_i(x),F_i(x)$ be the pdf and the cdf of $\IG(\mu_i,\lambda)$ at $x$, respectively. By standard formulae, 
\begin{align*}
    f_i(x)&=\sqrt{\frac{\lambda}{2\pi x^3}}\exp\pa{-\frac{\lambda(x-\mu_i)^2}{2\mu_i^2x}} \\
    F_i(x)&=\Phi\pa{\sqrt{\frac{\lambda}{x}}\pa{\frac{x}{\mu_i}-1}}+\exp\pa{\frac{2\lambda}{\mu_i}}\Phi \pa{-\sqrt{\frac{\lambda}{x}}\pa{\frac{x}{\mu_i}+1}}.
\end{align*}
Solving for $f_1(x)=f_2(x)$, we get a unique solution $x^*=\frac{2\mu_1\mu_2}{\mu_1+\mu_2}$. Therefore, 
\begin{align*}
    \TV(\IG(\mu_1, \lambda), \IG(\mu_2, \lambda)) &=\int_0^{x^*} f_1(x)-f_2(x)dx =F_1(x^*)-F_2(x^*).
\end{align*}
Then, we consider the limiting distribution as $\mu \rightarrow \infty$. Letting the pdf and cdf of the limiting distribution be $f_\infty(x), F_\infty(x)$, respectively, we have 
\begin{align*}
    f_\infty(x)&=\sqrt{\frac{\lambda}{2\pi x^3}}\exp\pa{-\frac{\lambda}{2x}}, \\
    F_\infty(x)&=2\Phi(-\sqrt{\lambda/x}).
\end{align*}
We denote the error function as $\Erf(x)=\frac{2}{\sqrt{\pi}}\int_0^x e^{-t^2} dt$. We have 
\begin{align*}
\TV(\IG(\mu_i, \lambda), \IG(\infty, \lambda)) 
        &\le \int_{2\mu_i}^\infty f_\infty (x)-f_{\mu_i}(x) dx \le \int_{2\mu_i}^\infty f_\infty (x) dx = 1-2\Phi\pa{-\sqrt{\frac{\lambda}{2\mu_i}}} \\
    &=_{(i)}-\Erf\pa{-\frac{\sqrt{\lambda}}{2\sqrt{\mu_i}}} =_{(ii)}\Erf\pa{\frac{\sqrt{\lambda}}{2\sqrt{\mu_i}}} \le_{(iii)} \sqrt{\frac{\lambda}{\pi \mu_i}} ,
\end{align*}
where (i) is due to $\Phi(x)=\frac{1}{2}(1+\Erf(x/\sqrt{2}))$, (ii) is because that the error function is odd, and (iii) comes from the fact that  $\Erf^{\prime}(x) = \frac{2}{\sqrt{\pi}} e^{-x^2}$.
Hence, 
\begin{align*}
    \TV(\IG(\mu_1, \lambda), \IG(\mu_2, \lambda)) & \le \TV(\IG(\mu_1, \lambda), \IG(\infty, \lambda)) + \TV(\IG(\mu_2, \lambda), \IG(\infty, \lambda)) \\
    &\le \sqrt{\frac{\lambda}{\pi \mu_1}} + \sqrt{\frac{\lambda}{\pi \mu_2}} \le  2\sqrt{\frac{\lambda}{\pi \min\{\mu_1, \mu_2\}}} .
\end{align*}
\end{proof}


\subsection{Proof of Lemma~\ref{l:lasso-feasible}} \label{a:lasso-feasible-proof}
\begin{proof}[Proof of Lemma~\ref{l:lasso-feasible}]
Since the map $T$ is a bijection, we have 
$$
\sup_{A} \frac{\nu_\dagger(A)}{\pi^{\Lasso}(A)} = \sup_{A} \frac{T_{\#}\nu_\dagger(A)}{T_{\#}\pi^{\Lasso}(A)} = \sup_{A} \frac{\nu^\prime_\dagger(A)}{T_{\#}\pi^{\Lasso}(A)},
$$
where the supremum is taken over all measurable sets $A \subseteq \mathbb{R}^d$. By equation (30), we have that 
\begin{align*}
   T_{\#}\pi^{\Lasso}(A) &\propto \rho^{n+2\alpha -2} \exp\pa{-\frac{1}{2}\|\rho y - X\varphi\|_2^2 - \lambda \|\varphi\|_1-\rho^2 \xi}  \\
   &\ge \rho^{n+2\alpha -2} \exp\pa{-\frac{1}{2}\|\rho y - X\varphi\|_2^2 - \lambda \|\varphi\|^2_2-\rho^2 \xi-\lambda d} , 
\end{align*}
where the inequality is due to $\|\varphi\|_1 \le \|\varphi\|_2^2+d$. Therefore, 
\begin{align*}
\sup_{A} \frac{\nu^\prime_\dagger(A)}{T_{\#}\pi^{\Lasso}(A)} & \le e^{\lambda d}\frac{\int_{\rho\in \R^+} \int_{\varphi\in \R^d} \rho^{n+2\alpha -2} \exp\pa{-\frac{1}{2}\|\rho y - X\varphi\|_2^2 - \lambda \|\varphi\|_1-\rho^2 \xi} d \rho d\varphi} {\int_{\rho\in \R^+} \int_{\varphi\in \R^d} \rho^{n+2\alpha -2} \exp\pa{-\frac{1}{2}\|\rho y - X\varphi\|_2^2 - \lambda \|\varphi\|^2_2-\rho^2 \xi} d \rho d \varphi} \\
&\le e^{\lambda d}\frac{\int_{\rho\in \R^+} \int_{\varphi\in \R^d} \rho^{n+2\alpha -2} \exp\pa{- \lambda \|\varphi\|_1-\rho^2 \xi} d \rho d\varphi} {\int_{\rho\in \R^+} \int_{\varphi\in \R^d} \rho^{n+2\alpha -2} \exp\pa{-\frac{1}{2}\|\rho y - X\varphi\|_2^2 - \lambda \|\varphi\|^2_2-\rho^2 \xi} d \rho d \varphi} \\
&= e^{\lambda d}\frac{\int_{\varphi\in \R^d} e^{-\lambda \|\varphi\|_1} d\varphi \int_{\rho\in \R^+}  \rho^{n+2\alpha -2} \exp\pa{-\rho^2 \xi } d\rho}{\int_{\rho\in \R^+} \int_{\varphi\in \R^d} \rho^{n+2\alpha -2} \exp\pa{-\frac{1}{2}\|\rho y - X\varphi\|_2^2 - \lambda \|\varphi\|^2_2-\rho^2 \xi} d \rho d \varphi} \\
&=e^{\lambda d}\frac{(2/\lambda)^d \int_{\rho\in \R^+}  \rho^{n+2\alpha -2} \exp\pa{-\rho^2 \xi } d\rho}{\int_{\rho\in \R^+} \int_{\varphi\in \R^d} \rho^{n+2\alpha -2} \exp\pa{-\frac{1}{2}\|\rho y - X\varphi\|_2^2 - \lambda \|\varphi\|^2_2-\rho^2 \xi} d \rho d \varphi}.
\end{align*}
The fraction is the same quantity as in~\eqref{e:lasso-important-quantity}. Following the same derivation as in Section~\ref{s:proof-lasso-iso}, we can obtain 
$$
\sup_{A} \frac{\nu^\prime_\dagger(A)}{T_{\#}\pi^{\Lasso}(A)} \le e^{\lambda d}e^{\mathcal{O}(d\log d + n\log n)}=e^{\mathcal{O}(d\log d + n\log n)}  .
$$
\end{proof}

\section{Auxiliary Proofs}
The proofs in this Appendix are not new. We present them here to make the paper self-contained.

\subsection{Proof of Lemma 4.2} \label{a:proof-transfer}
The complete proof is dispersed in a series of papers \cite{milman2008isoperimetric, milman2010isoperimetric, milman2012properties}, where the authors consider distributions satisfying a general class of isoperimetric inequalities and a general convexity condition on manifolds. For simplicity, we present the proof restricted to the log-concave measures on $\mathbb{R}^d$ satisfying the Cheeger-type isoperimetric inequality. 

The proof utilizes the equivalence between Cheeger-type isoperimetric inequality and a type of concentration inequality for log-concave measures. Specifically, a probability measure $\mu$ on $\mathbb{R}^d$ is said to satisfy the \textit{concentration inequality} with log-concentration profile $\alpha:\mathbb{R}^{+} \rightarrow \mathbb{R}$, if for any Borel set $A \subseteq \mathbb{R}^d$ with $\mu(A)\ge \frac{1}{2}$, we have 
$$
1-\mu(A^r) \le \exp\{-\alpha(r)\} \quad \forall r \ge 0.
$$

We first introduce the concepts and a lemma that will be used in the proof. Given a probability measure on $\mathbb{R}^d$, the \textit{isoperimetric profile} is a pointwise maximal function $\mathcal{I}: [0,1] \rightarrow \mathbb{R}^{+}$, so that $\mu^{+}(A) \ge \mathcal{I}(\mu(A))$ for all Borel sets $A \subseteq \mathbb{R}^d$. The \textit{isoperimetric minimizer} for a measure $v \in (0,1)$ is a Borel set $A \in \mathbb{R}^d$ satisfying $\mu(A)=v$ and $\mu^{+}(A)=\mathcal{I}(v)$. Furthermore, we denote the \textit{$\mu$-total curvature} of an isoperimetric minimizer $A$ as $H_{\mu}(A)$. The definition of $\mu$-total curvature is not important in this proof. We use it only in the following lemma. We refer readers interested in this quantity to Section 2.3 of \cite{milman2010isoperimetric}. 
\begin{lemma}[{\cite[Theorem 2, Remark 3]{morgan2005manifolds} and \cite[Theorem 2.3]{milman2010isoperimetric}}] \label{l:curvature} Let $A \subseteq \mathbb{R}^d$ be an isoperimetric minimizer for a given measure $v \in (0,1)$. Then, for any $r \ge 0$, 
$$
\mu(A^r) - \mu(A) \le \mu^{+}(A) \int_0^r \exp\{H_\mu(A)t\} dt.
$$ 
\end{lemma}

\begin{lemma}[{\cite[Corollary 3.3]{milman2008isoperimetric}}]\label{l:curvature_ub}Let $A \subseteq \mathbb{R}^d$ be an isoperimetric minimizer for a given measure $v \in (0,1)$. Then, 
$$
H_{\mu}(A) \le \frac{\mathcal{I}(v)}{v}.
$$
\end{lemma}
\begin{proof}[Proof of Lemma~\ref{l:curvature_ub}]
By Lemma~\ref{l:curvature}, 
$$
1-\mu(A^c) \le \mu^{+}(A^c) \int_0^\infty \exp\{H_\mu(A^c)t\}dt.
$$
Since $\mu^{+}(A)=\mu^{+}(A^c)$ and $H_\mu(A)= - H_\mu(A^c)$, we have 
$$
\frac{\mu(A)}{\mu^{+}(A)} \le \int_0^\infty \exp\{-H_{\mu}(A) t\} dt .
$$
If $H_{\mu}(A) \ge 0$, this implies 
$$
H_{\mu}(A) \le \frac{\mu^{+}(A)}{\mu(A)} = \frac{\mathcal{I}(v)}{v}.
$$
Otherwise, the statement trivially holds. 
\end{proof}

\begin{proof}[Proof of Lemma 4.2] At a high level, the proof is structured as three steps. First, we translate the isoperimetric inequality of $\mu_1$ into a concentration inequality. Second, using the condition that $\|\frac{d\mu_2}{d\mu_1}\|_{L^\infty} \le \exp(D)$, we transfer the concentration inequality for $\mu_1$ into a concentration inequality of $\mu_2$. One can see the transference between concentration inequalities is straightforward. Finally, we translate the concentration inequality of $\mu_2$ into its isoperimetric inequality. 
\paragraph*{Step 1: Isoperimetrc inequality for $\mu_1$ $\implies$ Concentration inequality for $\mu_1$ \cite[Proposition 1.7]{milman2008isoperimetric}}  Consider any Borel set $B \subseteq \mathbb{R}^d $ with measure $\mu_1(B) \ge \frac{1}{2}$. Define $f(r)=-\log(1-\mu_1(B^r))$. We have
\begin{align*}
    \frac{df}{dr}&= - \frac{1}{1-\mu_1(B^r)}(-\mu_1^{+}(B^r)) \\&\ge_{(i)} \frac{1}{1-\mu_1(B^r)} \frac{1}{\Ch(\mu_1)} \min\{\mu_1(B^r), \mu_1((B^r)^c)\} =_{(ii)} \frac{1}{\Ch(\mu_1)},
\end{align*}
where (i) is by the Cheeger-type isoperimetric inequality for $\mu_1$ and (ii) is by $\mu_1(B^r) \ge \mu_1(B) \ge \frac{1}{2}$. Then, 
$$
f(r) \ge f(0) + \int_0^r \frac{1}{\Ch(\mu_1)} dt = 
-\log(1-\mu_1(B)) + \frac{r}{\Ch(\mu_1)} \ge \log2 + \frac{r}{\Ch(\mu_1)},
$$
which is equivalent to the following concentration inequality 
\begin{align} \label{e:concentrate_1_a}
    \mu_1(B) \ge \frac{1}{2} \implies 1-\mu_1(B^r) \le \exp\bc{-\pa{\log2 +\frac{r}{\Ch(\mu_1)}}}.
\end{align}

\paragraph*{Step 2: Concentration inequality for $\mu_1$ $\implies$ Concentration inequality for $\mu_2$ \cite[Lemma 3.1]{milman2012properties}} The concentration inequality for $\mu_1$ in equation~\eqref{e:concentrate_1_a} is equivalent to its contrapositive: considering $A=(B^r)^c$, we have 
\begin{align} \label{e:concentate_1_b}
    \mu_1(A) > \exp\bc{-\pa{\log2 +\frac{r}{\Ch(\mu_1)}}} \implies  \mu_1(A^r) > \frac{1}{2}. 
\end{align}
To obtain a concentration inequality for $\mu_2$, consider any Borel set $S \subseteq \mathbb{R}^d$ with measure $\mu_2(S) \ge \frac{1}{2}$. We need to construct a related set with $\mu_1$ measure greater than $\frac{1}{2}$ to invoke the concentration inequality for $\mu_1$. Since $\|\frac{d\mu_2}{d\mu_1}\|_{L^{\infty}} \ge \exp(D)$ , $\mu_1(S) \ge \mu_2(S) (\|\frac{d\mu_2}{d\mu_1}\|_{L^{\infty}})^{-1} \ge \exp\{-(\log2 + D)\}$. By equation~\eqref{e:concentate_1_b}, for any $r>\frac{D}{\Ch(\mu_1)}$, $\mu_1(S^r)>\frac12$, therefore,
$$
\mu_1(\overline{S^{r_1}}) \ge \frac{1}{2}, \quad\text{for }  r_1 = D \Ch(\mu_1),
$$
where $\overline{S^{r_1}}$ is the closure of $S^{r_1}$.
By the concentration inequality for $\mu_1$ in equation~\eqref{e:concentrate_1_a}, 
$$
1-\mu_1(\overline{S^{r_1+r}}) \le \exp\bc{-\pa{\log2 +\frac{r}{\Ch(\mu_1)}}}.
$$
Again, by $\|\frac{\mu_2}{\mu_1}\|_{L^{\infty}} \ge \exp(D)$, $1-\mu_1(S^{r_1+r})=\mu_1(\mathbb{R}^d \setminus S^{r_1+r}) \ge \mu_2(\mathbb{R}^d \setminus S^{r_1+r})\exp(-D)$. Therefore, we obtain a concentration inequality for $\mu_2$: for any Borel set $ A\subseteq \mathbb{R}^d, \mu_2(A) \ge \frac{1}{2}$, we have
\begin{align*}
  1-\mu_2(S^{r_1+r}) \le \exp\bc{-\pa{\log2 +\frac{r}{\Ch(\mu_1)}-D}} \quad  \text{for }r_1 = D \Ch(\mu_1) . 
\end{align*}
This can be written in the standard form 
\begin{align} \label{e:concentrate_2_a}
  \mu_2(S) \ge \frac{1}{2} \implies 1-\mu_2(S^r) \le \exp\{-\alpha_2(r)\} ,  
\end{align}
where 
\vspace{-3em}
\begin{center}\[ \alpha_2(r) =  \left\{
\begin{array}{ll}
      \log2 & r \le 2D\Ch(\mu_1) \\
      \log2 + \frac{r}{\Ch(\mu_1)} - 2D & r \ge 2D\Ch(\mu_1) \\
\end{array} 
\right. \]
\end{center}
\paragraph*{Step 3: Concentration inequality for $\mu_2$ $\implies$ Isoperimetric inequality for $\mu_2$ \cite[Theorem 1.1 and Corollary 3.4]{milman2010isoperimetric}} Given an isoperimetric minimizer $A$ of measure $v \in (0, \frac{1}{2})$, we define $r_v= \alpha_2^{-1}(\log(\frac{1}{v})) = \Ch(\mu_1)( \log(\frac{1}{v})+2D-\log2) > 2D\Ch(\mu_1)$. By the contrapositive of equation~\eqref{e:concentrate_2_a}, 
$$
\mu_2(S) > \exp\{-\alpha_2(r)\} \implies \mu_2(S^r) > \frac{1}{2},
$$
so we have $\mu_2(\overline{A^{r_v}}) \ge \frac{1}{2}$. Applying Lemma~\ref{l:curvature}, 
$$
\frac{1}{2} - v \le \mu_2(\overline{A_{r_v}}) - \mu_2(A) \le \mu_2^{+}(A) \int_0^{r_v} \exp\{H_{\mu_2}(A)t\}dt.
$$
Using Lemma~\ref{l:curvature_ub}, \edit{letting $\mathcal I_2$ be the isoperimetric profile of $\mu_2$,}
$$
\int_0^{r_v} \exp\{H_{\mu_2}(A)t\}dt \le \int_0^{r_v} \exp\bc{\frac{\mathcal{I}_2(v)}{v}t}dt \le r_v \exp\bc{\frac{\mathcal{I}_2(v)}{v} r_v }.
$$
Let $f(v)=\frac{\mathcal{I}_2(v)}{v} r_v$. Then, 
$$
\frac{1}{2}-v \le \mu^{+}_2(A) r_v \exp\bc{\frac{\mathcal{I}_2(v)}{v} r_v } \implies f(v) + \log f(v) \ge \log\pa{\frac{1}{2v}-1}.
$$
Then, $f(v) \ge b(v)$, where $b(v)$ is the unique solution of $x+\log x- \log(\frac{1}{2v}-1)=0$. We have for $v \in (0, \frac{1}{2})$ that
\begin{align*}
\mathcal{I}_2(v) \ge v b(v)\frac{1}{r_v} \ge  v \frac{b(v)}{\log(1/v)}\frac{\log(1/v)}{r_v} &= \frac{v}{\Ch(\mu_1)}\frac{b(v)}{\log(1/v)}\frac{\log(1/v)}{\log(\frac{1}{v})+2D-\log2} \\ &\ge \frac{v}{\Ch(\mu_1)}\frac{b(v)}{\log(1/v)}\frac{\log 2}{2D}.
\end{align*}
Since $\mu_2$ is log-concave, $\mathcal{I}_2(v)$ is increasing on $[0, \frac{1}{2}]$. Therefore, for $v \in (0, \frac{1}{2}]$,
$$
\mathcal{I}_2(v) \ge \frac{1}{\Ch(\mu_1)} \frac{\log 2}{2D} 
\sup_{\lambda \in (0,v]} \frac{\lambda b(\lambda)}{\log(1/\lambda)}
$$
It is elementary to check that $
\sup_{\lambda \in (0,v]} \frac{\lambda b(\lambda)}{\log(1/\lambda)}
\ge c v $ for some universal constant $c>0$. 
Thus, 
$$
\mathcal{I}_2(v) \ge c\frac{v}{\Ch(\pi_1) D}, \quad\forall v \in (0, \frac{1}{2}].
$$
By the symmetry of the isoperimetric profile, 
$$
\mathcal{I}_2(v) \ge \frac{c}{\Ch(\pi_1)D}\min\{v, 1-v\}, \quad \forall v \in (0,1).
$$
The cases with $v=1$ and $v=0$ trivially hold. We prove the lemma by recalling the definition of $\Ch(\mu_2)$.
\end{proof}

\subsection{Proof of Lemma 4.4} \label{a:proof-main}
The method of using conductance-based arguments and isoperimetric inequalities to analyze the mixing of Markov chains can be found in \cite{cousins2014cubic, dwivedi2019log, narayanan2016randomized, chen2018fast, lovasz1999hit, lovasz2007geometry, lovasz2004hit, mou2019sampling}. \cite{chewi2023log} generalizes the argument in \cite{dwivedi2019log} to  Cheeger-type isoperimetric inequalities. 
\begin{proof} [Proof of Lemma 4.4]
In order to fit in the conductance-based argument, we need the isoperimetric inequalities to be in the “integral” form. Specifically, consider any measurable partition of the state space $\mathbb{R}^d=S_1 \sqcup S_2 \sqcup S_3$. We define $$r=d(S_1, S_2)=\inf \{\|x-y\|_2 : x\in S_1, y\in S_2\}.$$ Integrating both sizes of the Cheeger-type isoperimetric inequality from $0$ to $r$ yields 
$$
\int_0^r \pi^{+}(S_1^\omega) d\omega \ge \int_0^r \frac{1}{\Ch(\pi)} \min\{\pi(S_1^\omega), \pi((S_1^\omega)^c) \} d \omega.
$$
The definition of Minkowski content $\pi^{+}(S_1^\omega)=\lim_{\epsilon \rightarrow 0} \frac{\pi((S_1^\omega)^\epsilon)-\pi(S_1^\omega)}{\epsilon}$ implies that \\$\int_0^r \pi^{+}(S_1^\omega) d\omega = \pi(S_1^r) - \pi(S_1) = \pi(S_1^r \setminus S_1)$. It follows from $S_1 \not\subseteq S_1^r \setminus S_1$ and $S_2 \not\subseteq S_1^r \setminus S_1$ that $S_1^r \setminus S_1 \subseteq S_3$, and thus $\pi(S_1^r \setminus S_1) \le \pi(S_3).$ One the other hand, since $S_2 \subseteq (S_1^\omega)^c$, $\min\{\pi(S_1^\omega), \pi((S_1^\omega)^c) \} \ge \min\{\pi(S_1), \pi(S_2) \}$ for all $\omega \le r$. Therefore, 
\begin{align}\label{e:iso-integral}
    \pi(S_3) \ge \frac{r}{\Ch(\pi)} \min\{\pi(S_1), \pi(S_2)\}.
\end{align}

In order to lower bound the conductance, we need to study the probability flows across all measurable partitions. Consider an arbitrary partition $\mathbb{R}^d=A_1 \sqcup A_2$ and define the bad sets in $A_1$ and $A_2$ by 
\begin{align*}
    B_1=\bc{u \in \mathbb{R}^d: \mathcal{P}_u(A_2) \le \revise{\frac{h}{2}}}\\
    B_2=\bc{v \in \mathbb{R}^d: \mathcal{P}_v(A_1) \le \frac{h}{2}}
\end{align*}
We regard the rest as the good set $G=\mathbb{R}^d \setminus(B_1 \cup B_2)$. 
\paragraph*{The Good Case: $\pi(B_1) \le \frac{1}{2} \pi(A_1)$ or $\pi(B_2) \le \frac{1}{2}\pi(A_2)$} WLOG, assume $\pi(B_1) \le \frac{1}{2} \pi(A_1)$. Then, 
\begin{align*}
   \int_{A_1} \mathcal{P}_u(A_2) d\pi(u) &\ge \int_{A_1 \setminus B_1} \mathcal{P}_u(A_2) d\pi(u) \ge_{(i)} \frac{h}{2} \pi(A_1 \setminus B_1) \\&\ge_{(ii)} \frac{h}{4} \pi(A_1) \ge \frac{h}{4} \min\{\pi(A_1), \pi(A_2)\},
\end{align*}

where (i) is by the definition of $B_1$ and (ii) is by $\pi(B_1) \le \frac{1}{2} \pi(A_1)$. 

\paragraph*{The Bad Case: $\pi(B_1) \ge \frac{1}{2} \pi(A_1)$ and $\pi(B_2) \ge \frac{1}{2}\pi(A_2)$.} We have 
\begin{align*}
\int_{A_1} \mathcal{P}_u(A_2) du &= \frac{1}{2}\pa{\int_{A_1} \mathcal{P}_u(A_2) d\pi(u)+\int_{A_2} \mathcal{P}_v(A_1) d\pi(v)} \\
& \ge \frac{1}{2}\pa{\int_{A_1 \setminus B_1} \mathcal{P}_u(A_2) d\pi(u)+\int_{A_2 \setminus B_2} \mathcal{P}_v(A_1) d\pi(v)} \\
& \ge \revise{\frac{h}{4}} \pa{\pi(A_1 \setminus B_1) + \pi(A_2 \setminus B_2)} =  \frac{h}{4}\pi(G).
\end{align*} Then, substituting $S_1=B_1, S_2=B_2$, and $S_3=G$ into the integral form of the isoperimetric inequality~\eqref{e:iso-integral}, we have 
\begin{align*}
    \pi(G) \ge \frac{d(B_1, B_2)}{\Ch(\pi)} \min\{\pi(B_1), \pi(B_2)\} \ge \frac{d(B_1, B_2)}{2\Ch(\pi)} \min\{\pi(A_1), \pi(A_2)\}.
\end{align*}
The one-step overlap condition makes sure the two bad sets are far apart in Euclidean distance because for any $u \in B_1, v \in B_2$ and 
\begin{align*}
   \TV(\mathcal{P}_u, \mathcal{P}_v) &\ge \mathcal{P}_u(A_1) - \mathcal{P}_v(A_1) = 1 - \mathcal{P}_u(A_2) - \mathcal{P}_v(A_1) \ge 1 - \frac{h}{2}-\frac{h}{2} = 1-h \\&\implies \|u-v\|_2 \ge \Delta. 
\end{align*}
Therefore, 
\begin{align*}
    \int_{A_1} \mathcal{P}_u(A_2) d\pi(u) &\ge \frac{h}{4} \pi(G) \ge   \frac{ d(B_1, B_2)h}{\revise{8}\Ch(\pi)} \min\{\pi(A_1), \pi(A_2)\} \\&\ge   \frac{\Delta h}{\revise{8}\Ch(\pi)} \min\{\pi(A_1), \pi(A_2)\}.
\end{align*}
Combining the two cases, the conductance satisfies 
$$
\Phi = \sup_A \frac{\int_{A} \mathcal{P}_u(A^c) du}{\min\{\pi(A), \pi(A^c)\}} \ge c h\min\bc{2, \frac{\Delta}{\Ch(\pi)} }.
$$
By assuming $\frac{\Delta}{\Ch(\pi)} \le 2$, we prove the lemma. 

\end{proof}
\end{document}